\documentclass[oneside,english,reqno]{amsart}
\usepackage[T1]{fontenc}
\usepackage[latin9]{inputenc}
\usepackage{amsthm}
\usepackage{amssymb}
\usepackage{amsmath}
\usepackage{xcolor}
\usepackage[normalem]{ulem}
\usepackage{cancel}
\usepackage{esint}
\usepackage[margin=1in]{geometry}

\makeatletter
\usepackage{amsthm,amsmath,amsfonts,amssymb,color}
\usepackage[bookmarks]{hyperref}
\usepackage{mathtools}
\numberwithin{equation}{section}
\numberwithin{figure}{section}
\theoremstyle{plain}
\newtheorem{thm}{Theorem}[section]
\newtheorem{prop}[thm]{Proposition}

\newtheorem{lem}[thm]{Lemma}

\newtheorem{rem}[thm]{Remark}

  \newcounter{casectr}
  
\theoremstyle{definition}

\theoremstyle{remark}

\newcommand{\ZZZ}{\mathbb{Z}}
\newcommand{\TTT}{\mathbb{T}}
\newcommand{\N}{\mathbb{N}}
\newcommand{\NN}{\mathcal{N}}
\newcommand{\NNN}{P_{\leq N}\mathcal{N}}
\newcommand{\RRR}{\mathbb{R}}

\newcommand{\PPP}{\mathbb{P}}

\newcommand{\EEE}{\mathbb{E}}

\begin{document}

\title[Random data]{2D-Defocusing Nonlinear Schr\"odinger Equation with \\Random Data on irrational tori}

\author[C. Fan]{Chenjie Fan$^1$}
\address{
Department of Mathematics, University of Chicago, Chicago, IL}
\email{cjfanpku@gmail.com}

\thanks{$^1$ C.F. is  funded in part by   an AMS-Simons Foundation travel grant.}

\author[Y. Ou]{Yumeng Ou$^2$}
\address{
Department of Mathematics, Baruch College, City University of New York, New York, NY}
\email{yumeng.ou@baruch.cuny.edu}
\thanks{$^2$  Y.O. is funded in part by NSF DMS-1854148.}

\author[G. Staffilani]{Gigliola Staffilani$^3$}
\address{Department of Mathematics, Massachusetts Institute of Technology, Cambridge, MA}
\email{gigliola@math.mit.edu} 

\thanks{$^3$  G.S. is funded in part by NSF DMS-1462401 and DMS-1764403, and the Simons Foundation. }

\author[H. Wang]{Hong Wang$^4$}
\address{Institute for Advanced Study, Princeton, New Jersey}
\email{hongwang@ias.edu}
\thanks{$^4$ H.W. is funded by the S.S. Chern Foundation for Mathematics Research Fund and the National Science Foundation.}

\begin{abstract} We revisit the work of Bourgain on the invariance of the Gibbs measure for the cubic, defocusing nonlinear Schr\"odinger equation in 2D on a square torus, and we prove the equivalent result on any tori.
\end{abstract}

\maketitle

\section{introduction}
The purpose of this work is to revisit the famous work of Bourgain on the invariance of the Gibbs measure for the 2D defocusing cubic nonlinear Schr\"odinger equation (NLS) on a square torus $\TTT^2$, \cite{bourgain1996invariant}, and extend his proof to any torus. Since later we  often use the definition of {\it rational} or {\it irrational} torus, we readily give it here. Assume that a 2D torus  $\TTT^2$ has periods $T_1$ and $T_2$. If $T_1/T_2$ is rational we call  $\TTT^2$ a rational torus, otherwise irrational. As one can see from the proof for  Strichartz estimates on 
rational tori \cite{bourgain1993fourier}, Bourgain uses a fundamental property of the linear solution of the Schr\"odinger 
equation defined on a rational torus, the fact that the solution is also periodic in time. In the proof this fact is used in reducing the Strichartz inequality  to estimating the cardinality of a  set of   lattice points $(x,y)\in \ZZZ^2$ that satisfy a quadratic equation $x^2+ay^2=R^2$, where $a, R$ are natural numbers. Then  an elementary theorem  from number theory  is invoked to give a sharp bound in terms of $R$. If one wants to repeat Bourgain's proof for generic tori one has to obtain the same sharp bound when counting the 
lattice points in a region  such as $\{(x,y)\, \, /\, \, x^2+\alpha y^2\leq R^2+O(1)\}$, where now $\gamma, R^2>0$. In general,  this number of lattice points is larger than the sharp bound above. Intuitively though the same type of Strichartz estimates available for rational tori should be available for irrational  one. In fact one expects even better ones since the irrationality of the torus should translate into fewer interactions among  linear Schr\"odinger solutions. After almost two decades the full range of Strichartz estimates on any torus were proved  by Bourgain and Demeter \cite{bourgain2014proof}, who obtained them  as a corollary of their proof of the {\it $l^2$ decoupling conjecture}, hence without invoking any number theory. Shortly later Deng, Germain and Guth \cite{DGG2017}  proved that indeed Strichartz estimates on a {\it generic}  irrational torus are better, in the sense that they {\it live} for  a longer interval of time. 

Let us now go back to  Bourgain's work on the invariance of the Gibbs measure in \cite{bourgain1996invariant}. If one considers the   nonlinear Schr\"odinger equation with solution $u$ and  Hamiltonian $H(u)$ in the frequency space instead of the physical space, then it can be recast as an infinite dimensional Hamiltonian system with variables $(q_n(t),p_n(t))$ such that $\hat u(t,n)=q_n(t)+ip_n(t)$ for frequencies  $n\in\ZZZ^2$. For this infinite dimensional system one can define a Gibbs measure that formally can be written as $d\mu=1/Ze^{H(u)}du$, where $Z$ is a normalizing factor to make it a probability measure, and its support lives  in $H^{s}(\TTT^2),\, s<0$,  see \cite{LRS1988, GJ1989}. 
Bourgain had already proved \cite{bourgain1994periodic} that for the 1D quintic NLS, where a similar measure can be defined with support in $H^{s}(\TTT),\, s<1/2$  \cite{LRS1988}, the Schr\"odinger flow  indeed leaves the measure invariant, meaning for any set $A$ in the support of the measure, its evolution  with respect to the Schr\"odinger flow at any later times has the same measure as $A$ itself.  Moreover using this invariance he proved that the flow can be defined globally almost surely.
Clearly such a question could have  been asked in 2D for the cubic defocusing NLS\footnote{It is known that while for the 1D quintic focusing NLS the Gibbs measure can be defined as long as the $L^2$ norm is smaller than a certain absolute constant, in 2D no Gibbs measure can be defined for the focusing case \cite{LRS1988}.} as well. The issue that Bourgain faced was that while in the 1D case the flow was (deterministically) defined, at least locally,  for any data in $H^s(\TTT),\, s>0$, and hence on the whole support of the Gibbs measure, which as recalled is in $H^{s}, \, s<1/2$, for the cubic 2D case  also the flow was  only known to be  defined for data in $H^s(\TTT^2),\, s>0$, hence missing  the support of the Gibbs measure, which is in $H^{s}, \, s<0$. To overcome this, and other serious analytic obstacles along the way, Bourgain used probabilistic tools, such as Wick ordering and  large deviation estimates, combined with more deterministic ones, such as Strichartz type estimates and counting lemmata 
similar to the ones recalled above. This brings us to the motivation of our paper.  Indeed here we   prove new counting lemmata, see Section \ref{sec: counting}, that hold more generally for any torus, and we  rework the local almost sure well-posedness in details\footnote{Here we will not repeat the argument that upgrades the local well-posedness  to the  global since the rationality or not of the torus plays no role.} to show that the bounds obtained in the counting lemmata, although weaker than the ones in \cite{bourgain1996invariant} still are enough to conclude the proof.  Although our paper follows the scheme of Bourgain's proof,   we decided  to add full details  for the convenience of the reader and because  along the way we  could  point out with remarks where special care needed to be put in place in order to treat the general case and rationality cannot be invoked.

Finally, we recently learned that  Deng, Nahmod and Yue \cite{dny2019} have extended Bourgain's result in \cite{bourgain1996invariant} for any nonlinearity $2r+1$, where $r\in \N$. This  is a remarkable feat since the high order of nonlinear interactions was previously considered an almost insurmountable obstacle in obtaining an almost sure  local flow in the support of the Gibbs measure.

\subsection{Statement of main result}
 In this paper, we study the 2D cubic Wick ordered NLS equation on irrational or rational tori.
 We will pose the NLS on a rectangular torus and rescale the $\Delta$. Let $\gamma\in (1,2)$ be any real number (possibly irrational) that  determines on $\mathbb{T}^2=[0,2\pi]^2$ the operator
\[
\Delta_{\gamma}:= \partial_{x}^{2}+\frac{1}{\gamma}\partial_{y}^{2}.
\]
The free solution to the linear Schr\"odinger  initial value problem
\begin{equation}
\begin{cases}
iu_{t}-\Delta_{\gamma}u=0,\\
u_{0}=\sum_{n\in \ZZZ^{2} }a_{n}e^{i n\cdot x}
\end{cases}
\end{equation}
is of the form
\[
S(t)u_{0}\equiv e^{it\Delta_{\gamma}}u_{0}\equiv\sum_{n\in \ZZZ^{2}}a_{n}e^{i n\cdot x}e^{i n^{2}t},
\]
where we let
\[
n^{2}:=n\cdot n=\langle n,n\rangle_{\gamma}=n_{1}^{2}+\gamma n_{2}^{2}.
\]
Following the set up of Bourgain \cite{bourgain1996invariant} we revisit the Wick ordered truncated NLS with random initial data
\begin{equation}\label{eq: wnls}
\begin{cases}
i\partial_{t}u_{N}+\Delta_{\gamma}u_{N}=P_{\leq N}[(|u_{N}|^{2}-M_{N})u_{N}],\\
u_{0,N}=\sum_{|n|\leq N}\frac{g_{n}(\omega)}{|n|}e^{i n \cdot x}, \, x\in \TTT^{2},
\end{cases}
\end{equation}where $\{g_n(\omega):\, n\in\mathbb{Z}^2\}$ are independent $L^2$-normalized complex Gaussians and 
\[
M_{N}:=2\int_{\TTT^{2}}|u_{N}(t,x)|^2\,dx=2\int_{\TTT^{2}}|u_{0,N}|^{2}.
\]
The operator $P_{\leq N}$ here denotes the projection  onto frequencies  $|n|\leq N$.
By definition of $M_N$, one may rewrite \eqref{eq: wnls} as 
 \begin{equation}\label{eq: wnlsworking}
\begin{cases}
i\partial_{t}u_{N}-\Delta_{\gamma}u_{N}=P_{\leq N}\NN(u_{N}),\\
u_{0,N}=\sum_{|n|\leq N}\frac{g_{n}(\omega)}{|n|}e^{i n \cdot x}, \,x\in \TTT^{2}
\end{cases}
\end{equation}by defining the Wick ordered nonlinearity
\begin{equation}\label{defnN}
\begin{split}
\NN(f,g,h)(x)&=\sum_{n_{i}\in \ZZZ^{2},\, n_{2}\neq n_{1},n_{3}}\widehat{f}(n_{1})\overline{\widehat{g}(n_{2})}\widehat{h}(n_{3})e^{i\langle n_{1}-n_{2}+n_{3},x\rangle}-\sum_{n\in \ZZZ^{2}}\widehat{f}(n)\overline{\widehat{g}(n)}\widehat{h}(n)e^{in\cdot x}\\
&=:\NN_1(f,g,h)+\NN_2(f,g,h), 
\end{split}
\end{equation}
and we write 
$
\NN(f):=\NN(f,f,f),$ similarly for $\NN_i,\, i=1,2$.

One may also study the formal limit equation of \eqref{eq: wnlsworking}
\begin{equation}\label{eq: limit}
\begin{cases}
i\partial_{t}u-\Delta_{\gamma}u=\NN(u),\\
u_{0}=\sum_{n}\frac{g_{n}(\omega)}{|n|}e^{i n\cdot x}, \,x\in \TTT^{2}.
\end{cases}
\end{equation}
For every $\omega$ and every $N$ fixed, equation \eqref{eq: wnlsworking} is finite dimensional and thus  an ODE. It hence has a local solution, and in fact also a global solution due to the mass conservation law. Therefore, one is mainly interested in a local theory for \eqref{eq: wnlsworking} that is independent of $N$. More precisely, the main result of the paper is the following
\begin{thm}\label{thm: main}
Let $u_{N}^{\omega}$ be the solution to \eqref{eq: wnlsworking},

\begin{equation}\label{eq: duhamelmain}
u_{N}^{\omega}(t,x)=e^{it\Delta_{\gamma}}u_{0,N}+i\int_{0}^{t}e^{i(t-s)\Delta_{\gamma}}\NNN(u_{N})\,ds, \quad\forall t<t_{\omega}.
\end{equation}
There exists $s_0>0,\epsilon_{0}>0$, so that for almost every $\omega\in \Omega$, there exists a $t_{\omega}$ independent of $N$, such that
\begin{equation}
\|u_{N}^{\omega}(t,x)-e^{it\Delta_{\gamma}}u_{0,N}\|_{X^{s_{0},\frac{1}{2}+\epsilon_{0}}}\lesssim 1.
\end{equation}
Moreover, for any $0<s'<s_0$, $w_{N}^{\omega}:=u_{N}^{\omega}(t,x)-e^{it\Delta_{\gamma}}u_{0,N}$ converges strongly in $ X^{s',\frac{1}{2}+\epsilon_{0}}[0,t_{\omega}]$ to some limit $w$. Furthermore, the limit $u^\omega:=w+e^{i\Delta t}u_{0}$ (called the solution to the Wick ordered NLS \eqref{eq: limit}) satisfies the Duhamel formula
\begin{equation}\label{eq: duhamelmain2}
u^{\omega}(t,x)=e^{it\Delta_{\gamma}}u_{0}+i\int_{0}^{t}e^{i(t-s)\Delta_{\gamma}}\NN(u^\omega)\,ds, \quad \forall t<t_{\omega}.
\end{equation}
\end{thm}

We refer the reader to Subsection \ref{sec:Xsb} for the definition of the $X^{s,b}$ space and its related properties. In the following, when the dependence on the parameter $\omega\in\Omega$ is clear from the context, we sometimes will drop the superscript in $u^\omega$ and write $u$ for short.

 Theorem \ref{thm: main} follows from the following quantitative version (independent of $N$) of the main result.
\begin{thm}\label{thm: main2}
Let $u^{\omega}_{N}$ be as in Theorem \ref{thm: main},
there exists $s_0>0, \epsilon_{0}, \alpha_{0}, t_{0}>0$ so that for every $t<t_{0}$, up to a set of probability measure $e^{-t^{-\alpha_{0}}}$, one has that  $w_{N}:=u_{N}-e^{it\Delta}u_{0,N}$ satisfies  $\|w_{N}\|_{X^{s_0,\frac{1}{2}+\epsilon_{0}}[0,t]}\lesssim 1$. Furthermore, for all $0<s'<s_0$, $w_{N}$ converges to some limit $w$ in $X^{s',\frac{1}{2}+\epsilon_{0}}$ and $u=w+e^{it\Delta_\gamma} u_{0}$ solves the wick ordered NLS in the sense that Duhamel formula \eqref{eq: duhamelmain2} is satisfied.
\end{thm}

%We similarly define $P_{N,k}$. The convention is that when we use capital number $N,M$,  we have that $N,M$ are dyadic numbers, and $k$ is usually used to denote usual integer.

%And, to apply Strichartz$X^{s,b}$ estimate
%\begin{equation}
%\|u\|_{L_{t,x}^{4}}\lesssim \|u\|_{X^{+,\frac{1}{2}-}}
%\end{equation}
%One needs to fix$ 0+=\epsilon_{1}$,$ \frac{1}{2}-=\frac{1}{2}-\epsilon_{1}$.

\subsection{Duhamel formula, Picard iteration and main propositions}
The proof follows from a Picard iteration scheme. One would like to write \eqref{eq: wnlsworking} into its Duhamel form, \eqref{eq: duhamelmain}. It is convenient to introduce an extra time cut off $\phi_{\delta}(t)=\phi(t/\delta)$, where $\phi\in C_0^\infty(\mathbb{R})$ is equal to $1$ on $[-1/2, 1/2]$ and $0$ outside $[-1,1]$, and consider instead the following slightly modified version of \eqref{eq: duhamelmain}:
\begin{equation}\label{eq: duhamelworking}
u'_{N}(t,x)=\phi_{\delta}(t)e^{it\Delta_{\gamma}}u_{0,N}+i\phi_{\delta}(t)\int_{0}^{t}e^{i(t-s)\Delta_{\gamma}}(P_{\leq N}\NN(\phi(t) u'_{N})\,ds.
\end{equation}
Note that when $t<\delta/2$, $u'_{N}$ is no different from $u_{N}$. 
In what  follows, for  convenience   we will not distinguish $u_{N}$ and $u'_{N}$.
Let
\begin{equation}
\Gamma_{N,\delta}u=i\phi_{\delta}(t)\int_{0}^{t}e^{i(t-s)\Delta_{\gamma}}P_{\leq N}\NN(\phi_\delta(t) u)\,ds,
\end{equation} and consider its formal limit as $N\to \infty$:
\begin{equation}
\Gamma_{\delta}u=i\phi_{\delta}(t)\int_{0}^{t}e^{i(t-s)\Delta_{\gamma}}\NN(\phi_{\delta}(t) u)\,ds.
\end{equation}
From the perturbative viewpoint, let
\begin{equation}
u_{N}=\phi_{\delta}(t)e^{it\Delta_{\gamma}}u_{0,N}+w_{N}(x,t).
\end{equation}
Then \eqref{eq: duhamelworking} is equivalent to 
\begin{equation}\label{eq: picardworking}
w_N(x,t)=\Gamma_{N,\delta}(\phi_{\delta}(t)e^{it\Delta_{\gamma}}u_{0,N}+w_N(x,t)),
\end{equation}which reduces Theorem \ref{thm: main2} to the following three propositions,

\begin{prop}\label{prop: final}
There exists a sufficiently small $\delta_{0}>0$ and $s_0\gg \epsilon_0>0$, and some $\alpha_{0}>0$,  such that  for every $0<\delta<\delta_{0}$, up to a set of measure $e^{-\delta^{-\alpha_0}}$ for some $\alpha_0$ depending on $\epsilon_0$, the map
\begin{equation}
w\mapsto \Gamma_{N,\delta}(\phi_{\delta}(t)e^{it\Delta_{\gamma}}u_{0,N}+w(x,t))
\end{equation}
is a contraction map on the space
\begin{equation}
\{w:\, \|w\|_{X^{s_0,b_0}}\leq 1\},
\end{equation}
where $b_0=1/2+\epsilon_0$.
\end{prop}

\begin{prop}\label{prop: cor1tofinal}
There exists a sufficiently small $\delta_{0}>0$ and $s_0\gg \epsilon_0>0$, and $\alpha_{0}>0$ s.t. for every $0<\delta<\delta_{0}$, up to a set of measure $e^{-\delta^{-\alpha_0}}$ for some $\alpha_0$ depending on $\epsilon_0$, the map
\begin{equation}
w\mapsto \Gamma_{\delta}(\phi_{\delta}(t)e^{it\Delta_{\gamma}}u_{0}+w(x,t))
\end{equation}
is a contraction map on the space
\begin{equation}
%\{w:\, \|w\|_{X^{s,1/2+}}\leq 1\}.
\{w:\, \|w\|_{X^{s_0,b_0}}\leq 1\},
\end{equation}where $b_0=1/2+\epsilon_0$.
\end{prop}

\begin{prop}\label{prop:cor2tofinal}
Let $\delta_{0}$, $s_0$, $\epsilon_0$ be as in Proposition \ref{prop: final} and \ref{prop: cor1tofinal}. Let $w_{N}$ be the unique function (fixed point) in 
$\{w:\,\|w\|_{X^{s_0,b_0}}\leq 1\}$ such that 
\begin{equation}
w_{N}=\Gamma_{N,\delta}(\phi_{\delta}(t)e^{it\Delta_{\gamma}}u_{0,N}+w_{N}(x,t)),
\end{equation}
and 
let $w^{*}$ be the unique 
 function (fixed point) in 
$\{w:\, \|w\|_{X^{s_0,b_0}}\leq 1\}$ such that
\begin{equation}
w^{*}=\Gamma_{\delta}(\phi_{\delta}(t)e^{it\Delta_{\gamma}}u_{0,N}+w^{*}(x,t)).
\end{equation}
Then one has for all $s'<s_0$ and  as $N\rightarrow \infty$  that
\begin{equation}
w_{N}\rightarrow w^{*}  \text{ in } X^{s',\frac{1}{2}+\epsilon_0}.
\end{equation}
\end{prop}
\begin{rem}
Note that Proposition \ref{prop: final} is stated  uniformly over all $N>0$. Thus to prove Proposition \ref{prop: final} is equivalent to prove Proposition \ref{prop: cor1tofinal}. 
For those who are familiar with the Picard iteration scheme, Proposition \ref{prop:cor2tofinal} is a stability argument that is essentially equivalent to the local existence argument giving Proposition \ref{prop: final} and \ref{prop: cor1tofinal}. However, to take into account  the difference between $\NNN(u_{N})$ and $\NN(u_{N})$, one will need to use extra derivative, which is the reason why the convergence in Proposition \ref{prop:cor2tofinal} only holds for $s'<s_0$.
\end{rem}
We will focus on the proof of Proposition \ref{prop: cor1tofinal}, which is the same as the one for  Proposition  
\ref{prop: final}, then  Proposition   \ref{prop:cor2tofinal} follows by using the same argument as in Section 5 of \cite{bourgain1996invariant}. One may also use the invariance  of the Gibbs measure to upgrade the local 
well-posedness to  a global one as  in \cite{bourgain1996invariant}.

\section{Proof of Proposition \ref{prop: cor1tofinal}: initial reduction and structure of the proof}

In this section, we outline the proof of Proposition \ref{prop: cor1tofinal}.
%{\color{red}{I think we can use this section to introduce the reduction to the trilinear expression and decompose into different frequencies $N_1, N_2, N_3$? And maybe include a map of the rest of the proof.}}
To begin with, fix $w$, one has by definition that
\begin{equation}\label{eqn: reduction}
\Gamma_\delta (\phi_\delta(t)e^{it\Delta_\gamma}u_0+w)=i\phi_\delta(t)\int_0^t e^{i(t-s)\Delta_\gamma}\mathcal{N}(\phi(t)e^{it\Delta_\gamma}u_0+\phi(t)w)\,ds=: A+B,
\end{equation}where, according with \eqref{defnN},  $A$ corresponds to $\NN_1$ and $B$ to $\NN_2$ respectively. We also use $\phi(t)\phi(t/\delta)=\phi(t)$. The estimate for part $B$ follows from standard $X^{s,b}$ space estimates, which we present in the end of Section \ref{sec:PreDet} for the sake of completeness. In order to study part $A$, we consider the Wick ordered nonlinearity $\NN_1$ as a trilinear expression, replacing the $w$ above by three functions $w_1, w_2, w_3$.  

Using $X^{s,b }$ smoothing \eqref{eq:Xsb} and duality, Proposition \ref{prop: cor1tofinal} will follow from
\begin{prop}\label{prop: onemore}
There exist $\delta,\delta_{0}, \alpha_{0},\epsilon_{0},b$ as in Proposition \ref{prop: cor1tofinal} satisfying $\epsilon_{0}\ll\epsilon_{1}\ll s_{0}$, so that for any $h(x,t)$ with
$\|h\|_{X^{0, 1-b_{0}}}\leq 1$, $h_0(x,t):=\phi(t/\delta)h(x,t)$,  one has estimate
\begin{equation}\label{eqn: mainest}
\left|\int_{\RRR\times \TTT^{2}}<D>^{s_{0}}(\NN_{1}(u_{1},u_{2},u_{3}))\overline{h_0}\,dxdt\right|\lesssim \delta^{\epsilon_{1}},
\end{equation}
\begin{equation}\label{eqn: mainest2}
\left|\int_{\RRR\times \TTT^{2}}<D>^{s_{0}}(\NN_{2}(u_{1},u_{2},u_{3}))\overline{h_0}\,dxdt\right|\lesssim \delta^{\epsilon_{1}},
\end{equation}
where $u_i$ is either $\phi_\delta(t)e^{it\Delta_\gamma}u_0$ or $w_i$.
\end{prop}
Here $<\cdot>$ is the Japanese bracket, $<D>:=\sqrt{1-\Delta}$.
\begin{rem}
We will neglect any loss of $\delta^{-C\epsilon_{0}}$ throughout the proof, since eventually all such loss will be compensated by the gain of $\delta^{\epsilon_{1}}$. In particular, one should not be concerned  about the loss in $X^{s,b}$ localization by multiplying $\phi(t/\delta)$.
\end{rem}

In the two estimates above, (\ref{eqn: mainest2}) follows easily from deterministic estimates, whose proof will be given at the end of Section \ref{sec:PreDet}. The majority of the rest of the paper is devoted to proving (\ref{eqn: mainest}).

More precisely, the proof of (\ref{eqn: mainest}) splits into eight different cases depending on whether the input functions $u_i$ are of the \emph{regular} (in the space $X^{s_0,b_0}$) or \emph{probabilisitic} forms. In addition, we further decompose each $u_i$ into pieces corresponding to different spatial Fourier frequencies (i.e. replacing $u_i$ with $P_{ N_i} u_i$ for some dyadic numbers $N_i$), then the desired result follows from a case by case study depending on the relative sizes of the spatial frequencies $N_1, N_2, N_3$. Note that the roles of $N_1$ and $N_3$ are completely symmetric as shown in the definition of $\NN$, so without loss of generality we may assume $N_1\geq N_3$ throughout. There are thus two main cases: $N_2 \geq N_1\geq N_3$ and $N_1 \geq N_2$.

The first case $N_2\geq N_1\geq N_3$ turns out to be easier, which we treat in Section \ref{sec: ProofPartI}. The second case needs to be further decomposed depending on the relative sizes of $N_2, N_3$ and where the random terms appear.  Following Bourgain's notation, we will use $II$ to denote the regular case (i.e. $u_i=w_i$) and $I$ to denote the probabilistic case (i.e. $u_i$ is the cutoff of the free solution with random initial data $u_0$). In Section \ref{sec: ProofPartII} and \ref{sec: ProofPartIII}, we will first estimate two typical cases: $N_1(II)\geq N_2(I) \geq N_3(II)$ (corresponding to ``case (a)'' of  \cite{bourgain1996invariant}) and $N_1(I)\geq N_2(II) \geq N_3(II)$ (corresponding to ``case (c)'' of  \cite{bourgain1996invariant}). These two cases are typical in the sense that all essential elements of the proof and  ideas will be displayed in the study of these two cases. Essentially  this is because in the two cases, the random term appears in relatively higher frequencies hence there is less decay in terms of $N_1$ that one would expect; there is also no additional random term present which prevents one to fully exploit the cancellation brought by randomization. Note that these two sections are the main part of our proof. We will discuss  the treatment of other cases in Section \ref{sec: ProofpartIV}.
\subsection{A list of cases}
Following Bourgain, we need to study
\begin{itemize}
\item Case (0): $N_{2}\gg N_{1}$;
\item Case (a): $N_1(II)\geq N_2(I)\geq N_3(II)$;
\item Case (b): $N_1(II)\geq N_3(I)\geq N_2(II)$;
\item Case (c): $N_1(I)\geq N_2(II)\geq N_3(II)$;
\item Case (d): $N_1(I)\geq N_3(II)\geq N_2(II)$;
\item Case (e): $N_1(II)\geq N_2(I)\geq N_3(I)$;
\item Case (f): $N_1(II)\geq N_3(I)\geq N_2(I)$;
\item Case (g): $N_1(I)\geq N_2(I)\geq N_3(II)$;
\item Case (h): $N_1(I)\geq N_3(I)\geq N_2(II)$;
\item Case (i): $N_1(I)\geq N_2(II)\geq N_3(I)$;
\item Case (j): $N_1(I)\geq N_3(I)\geq N_2(II)$;
\item Case (k): $N_1(I)\geq N_2(I)\geq N_3(I)$;
\item Case (l): $N_1(I)\geq N_3(I)\geq N_2(I)$.
\end{itemize}
\begin{rem}
Strictly speaking, one will need to study, for example in case (a), $N_{2}\lesssim N_{1}, N_{2}\geq N_{3}$. The analysis will be the same as for $N_{1}\geq N_{2}\geq N_{3}$, we neglect this issue.
\end{rem}

\subsection{Notation}
For the sake of notational convenience, we will denote $\langle,\rangle_{\gamma}$ by $\langle, \rangle$ in short. 
We use $P_N$, $P_{\leq N}$ to denote Littlewood Paley projections in the physical space ($x$ variable), as mentioned above. We will  use $P_{|\tau|<M}$ as Littlewood Paley projections in the time space ($t$ variable). We will also use $P_{|\tau-n^{2}|<M}$ to denote space time Littlewood Paley projections with respect to paraboloids.

%We similarly define $P_{N,k}$. The convention is that when we use capital number $N,M$,  we have that $N,M$ are dyadic numbers, and $k$ is usually used to denote usual integer.

For the sake of convenience, we sometimes abuse notation by identifying  $P_{N}^{2}=P_{N}$, $\phi(t)^{2}=\phi(t)$.
%And, to apply Strichartz$X^{s,b}$ estimate
%\begin{equation}
%\|u\|_{L_{t,x}^{4}}\lesssim \|u\|_{X^{+,\frac{1}{2}-}}
%\end{equation}
%One needs to fix$ 0+=\epsilon_{1}$,$ \frac{1}{2}-=\frac{1}{2}-\epsilon_{1}$.
Throughout the paper, we  use several parameters, and we always require 
\begin{equation}
1\geq s_{0}\gg s_{1}\geq \epsilon_{1}\gg \epsilon_{0}>0.
\end{equation}

%%%%%%%%
%%%%%%%%%%%%%%%%
%%%%%%%%
%%%%%%%%
%%%%%%%%

\section{Preliminaries}\label{sec:Pre}

\subsection{The $X^{s,b}$ space}\label{sec:Xsb}

In this subsection, we recall the definition of the $X^{s,b}$ space and summarize some classical estimates  that will be used in the proof. One may refer to \cite{bourgain1993fourier}, \cite{caffarelli1999hyperbolic}, \cite{bourgain2014proof} for more details.

 Let $v(t,x)$ be a function on $\RRR\times \mathbb{T}^{2}$. Let $\widehat{v}$ be the Fourier transform of $v$, i.e. 
\[
v(t,x)=\int_{\RRR} \sum_{n\in \ZZZ^{2}}\widehat{v}(n,\tau)e^{in\cdot x}e^{i\tau t}\,d\tau,
\]
and the $X^{s,b}$ norm can be defined as
\begin{equation}\label{eq: Xsbnorm}
\|v\|_{X^{s,b}}^{2}:=\int_{\RRR}\sum_{n\in \mathbb{Z}^2} <n>^{2s}<\tau-n^{2}>^{2b}|\widehat{v}(n,\tau)|^{2}\,d\tau,
\end{equation}where $<n>:=\sqrt{1+n^2}$ is the Japanese bracket.

Note that another convenient way to define the $X^{s,b}$ norm is via the ansatz
\begin{equation}
v(t,x)=\int_{\RRR}\sum_{n\in \ZZZ^{2}} a(n,\lambda)e^{in\cdot x}e^{in^{2}t}e^{i \lambda t}\,d\lambda,
\end{equation}
which gives
\begin{equation}
\|v\|_{X^{s,b}}^{2}=\int_\RRR\sum_{n\in\mathbb{Z}^2} |a(n,\lambda)|^{2}<n>^{2s}<\lambda>^{2b}\,d\lambda.
\end{equation}
The $X^{s,b}$ space is very useful in dispersive PDE for at least two reasons: first, it inherits the Strichartz estimates enjoyed by free solutions of the Schr\"odinger equation; second, it exploits the smoothing effect of the Duhamel formula. 

We now recall the Strichartz estimates on tori, rational or irrational, \cite{bourgain1993fourier},\cite{bourgain2014proof},
\begin{equation}
\|e^{it\Delta_\gamma} (P_Bf)\|_{L_{t,x}^{4}([0,1]\times \TTT^{2})}\lesssim_\epsilon N^\epsilon \|f\|_{L^2_x},
\end{equation}
where $P_B$ is the Littlewood-Paley projection onto the spatial frequency ball $B$ of radius $N$ (not necessarily centered at the origin).

By the Minkowski inequality and Cauchy-Schwarz, this implies that 
\begin{lem}For any $u\in X^{0, \frac{1}{2}+\epsilon'}$, there holds
\begin{equation}\label{eq: Xsbstri} 
\|P_Bu\|_{L_{t,x}^{4}([0,1]\times \TTT^{2})}\lesssim_{\epsilon,\epsilon'} N^{\epsilon} \|u\|_{X^{0, \frac{1}{2}+\epsilon'}}.
\end{equation}
\end{lem}
Via an interpolation with the Hausdorff-Young inequality, the estimate above can be upgraded to
\begin{equation}\label{eq: Xsbhy}
\|P_Bu\|_{L_{t,x}^{4}([0,1]\times \TTT^{2})}\lesssim_\epsilon N^\epsilon \|u\|_{X^{0, \frac{1}{2}-\frac{\epsilon}{4}}}.
\end{equation}

%Note that a more accurate and explicit way to state \eqref{eq: Xsbhy} is as follows: for all $\epsilon>0$ small enough, one has 
%\begin{equation}
%\|u\|_{L_{t,x}^{4}([0,1]\times \TTT^{2})} \lesssim_{\epsilon} \|u\|_{
%X^{\epsilon, \frac{1}{2}-\frac{1}{4}\epsilon}}.
%\end{equation}

We also record another estimate, which follows immediately by interpolating (\ref{eq: Xsbstri}) with the trivial bound
\[
\|u\|_{L_{t,x}^2([0,1]\times \mathbb{T}^{2})}\lesssim \|u\|_{X^{0,0}}.
\]
\begin{lem}\label{lem: interpolate}For any $u\in X^{0,\frac{1}{3}}$, there holds
\begin{equation}\label{eq: interpo}
\|P_Bu\|_{L_{t,x}^{3}([0,1]\times \mathbb{T}^{2})}\lesssim_{\epsilon} N^\epsilon \|u\|_{X^{0,\frac{1}{3}}}.
\end{equation}
\end{lem}

As mentioned earlier,  the $X^{s,b}$ space also exploits the smoothing effect of the Duhamel formula, which can be made precise by the following estimate. 
\begin{lem}\label{lem: Xsbsmoothing}
For all $s\geq 0, b > \frac{1}{2}$ and time cut-off function $\phi$ as above, there holds
\begin{equation}\label{eq:Xsb}
\left\|\phi(t)\int_{0}^{t}e^{i(t-s)\Delta_\gamma} v(s)\,ds\right\|_{X^{s,b}}
\lesssim \|v\|_{X^{s,(b-1)}}.
\end{equation}
\end{lem}

Before ending this subsection, we also record the following localization properties of the $X^{s,b}$ space:
\begin{lem}\label{lem: xsblocaliztion}
Let $u\in X^{s,b}$, then
\begin{equation}\label{eq: loc}
\|\phi_\delta(t)u\|_{X^{s,b}}\lesssim_b \begin{cases} \|u\|_{X^{s,b}}, & 0<b<\frac{1}{2},\\ \delta^{\frac{1}{2}-b}\|u\|_{X^{s, b}}, & \frac{1}{2}<b<1.\end{cases}
\end{equation}Moreover, for all $0\leq b'\leq b<1/2$, there holds
\begin{equation}
\|\phi_\delta(t) u\|_{X^{s,b'}}\lesssim_{\epsilon } \delta^{b-b'-\epsilon}\|u\|_{X^{s,b}}.
\end{equation}
\end{lem}

\subsection{Deterministic Estimates}\label{sec:PreDet}
In this subsection, we collect several by now standard deterministic estimates. All of them  were introduced when studying  standard local theory of deterministic NLS on tori. We start with an estimate that exploits the time localization. One may refer to \cite{bourgain1993fourier}, \cite{caffarelli1999hyperbolic} for  proof. We provide a brief sketch of proof of the lemma in Appendix \ref{asectimeloc} for the convenience of the reader.

\begin{lem}\label{lem: smallpower}
Let $P_B$ be the Littlewood-Paley projection onto the spatial frequency ball $B$ of radius $N$  and $0<s\ll 1$. Then one has for all $\epsilon\ll s$ that
\begin{equation}\label{eq: xsbstr}
\|\phi(t)P_{B}u\|_{L_{t,x}^{4}}\lesssim_{\epsilon}N^{C\epsilon}\|u\|_{X^{0,1/2-\epsilon}},
\end{equation}
\begin{equation}\label{eqq: xsbstrloc}
\|\phi_\delta(t) (P_Bu)\|_{L_{t,x}^{4}}\lesssim_{\epsilon} N^{s+C\epsilon}\|u\|_{X^{0,\frac{1}{2}-\epsilon}}\delta^{\frac{s}{8}}.
\end{equation}
\end{lem}
The number $1/8$ is not meant to be sharp, one can for example upgrade it to $1/4-$. 

Throughout the rest of the section, we write
\begin{equation}
f_{i}(x,t)=\sum_{n\in \mathbb{Z}^2}f_{i}(n,t)e^{in\cdot x}, \quad i=1,2,3,
\end{equation}
i.e. $f_{i}(n,t)$ is the space Fourier transform of $f_i$. For the sake of brevity, we abbreviate $f_{i}(n,t)$ as $f_{i}(n)
$.

We summarize below several standard estimates that will be frequently used in the proofs that will come  later. One may refer to 
\cite{bourgain1996invariant, bourgain1993fourier, caffarelli1999hyperbolic}.
We will also sketch them in Appendix B for the convenience of the reader.
\begin{lem}\label{lem: deter}
Let $N_{1}\gtrsim N_{2}, N_{3}$, $1\gg s_{1}\gg \epsilon_{0}$, and $\psi(t)$ be a Schwartz function. Decompose $P_{N_{1}}=\sum_{J\in \mathcal{J}}P_{J}$, where $J\in \mathcal{J}$ are finitely overlapping balls in the region $|n|\sim N_{1}$ with radius $\sim \max(N_{2},N_{3})$.
 then one has 
\begin{equation}\label{deter1}
\left|\int\psi(t)\bar{h}P_{N_{1}}f_{1}\overline{P_{N_{2}}{f}_{2}}P_{N_{3}}f_{3}\right|\lesssim (\max(N_{2},N_{3}))^{C\epsilon_{0}}\|h\|_{X^{0,1-b_{0}}}\prod_{i}\|P_{N_{i}}f_{i}\|_{X^{0,b_{0}}},
\end{equation}
\begin{equation}\label{deter2}
\left|\int \psi(t)\bar{h}\NN_{1}(P_{N_{1}}f_{1},P_{N_{2}}{f}_{2},P_{N_{3}}f_{3})\right|\lesssim(\max(N_{2},N_{3}))^{C\epsilon_{0}}\|h\|_{X^{0,1-b_{0}}}\prod_{i}\|P_{N_{i}}f_{i}\|_{X^{0,b_{0}}},
\end{equation}
\begin{equation}\label{deter1.1}
\left|\int\psi(t/\delta)\bar{h}P_{N_{1}}f_{1}\overline{P_{N_{2}}f_{2}}P_{N_{3}}f_{3}\right|\lesssim \delta^{s_{1}/8}(\max(N_{2},N_{3}))^{s_{1}+C\epsilon_{0}}\|h\|_{X^{0,1-b_{0}}}\prod_{i}\|P_{N_{i}}f_{i}\|_{X^{0,b_{0}}},
\end{equation}
\begin{equation}\label{deter1.2}
\left|\int\psi(t/\delta)\bar{h}\NN_{1}(P_{N_{1}}f_{1}, P_{N_{2}}{f}_{2},P_{N_{3}}f_{3})\right|\lesssim \delta^{s_{1}/8}(\max(N_{2},N_{3}))^{s_{1}+C\epsilon_{0}}\|h\|_{X^{0,1-b_{0}}}\prod_{i}\|P_{N_{i}}f_{i}\|_{X^{0,b_{0}}},
\end{equation}
\begin{equation}\label{deter3}
\begin{aligned}
&\left|\int \psi(t)\bar{h}\NN_{1}(P_{N_{1}}f_{1},P_{N_{2}}{f}_{2},P_{N_{3}}f_{3})\right|
\\
\lesssim &(\max(N_{2},N_{3}))^{C\epsilon_{0}}\|h\|_{X^{0,1/3}}\|\|f_{2}\|_{X^{0,1/3}}\|f_{3}\|_{X^{0,1/3}}\sup_{J}\|P_{J}f_{1}\|_{L_{t,x}^\infty}\\
+&{\bf 1}_{N_{1}\sim N_{2}}\|P_{N_{1}}f_{1}\|_{X^{0,b_{0}}}\|P_{N_{1}}f_{2}\|_{X^{0,1/3}}\|P_{N_{3}}f_{3}\|_{X^{0,1/3}}\|{{{P_{N_{3}}h}}}\|_{X^{0,1/3}}\\
 +&{\bf 1}_{N_{2}\sim N_{3}}\|P_{N_{1}}f_1\|_{X^{0,b_{0}}}\|P_{N_{2}}f_{2}\|_{X^{0,1/3}}\|P_{N_{2}}f_{3}\|_{X^{0,1/3}}\|P_{N_{1}}h\|_{X^{0,1/3}}
 \\
+&{\bf 1}_{N_{1}\sim N_{2}\sim N_{3}}\min_i\left(\|P_{N_{1}}h\|_{X^{0,1-b_{0}}}\|f_{i}\|_{X^{0,b_{0}}}\sup_{|n|\sim N_{1}}\prod_{j\neq i}\|f_{j}(n)e^{in\cdot x}\|_{X^{0,b_{0}}}\right),
\end{aligned}
\end{equation}
\begin{equation}\label{deter3.5}
\begin{aligned}
&\left|\int \psi(t)\bar{h}\NN_{1}(P_{N_{1}}f_{1},P_{N_{2}}{f}_{2},P_{N_{3}}f_{3})\right|
\\
\lesssim &N_{1}^{C\epsilon_{0}}\|h\|_{X^{0,1/3}}\|P_{N_{1}}f_{1}\|_{L_{t,x}^\infty}\|f_{2}\|_{X^{0,1/3}}\|f_{3}\|_{X^{0,1/3}}\\
+&{\bf 1}_{N_{1}\sim N_{2}}\|P_{N_{1}}f_{1}\|_{X^{0,b_{0}}}\|P_{N_{1}}f_{2}\|_{X^{0,1/3}}\|P_{N_{3}}f_{3}\|_{X^{0,1/3}}\|{\color{black}{P_{N_{3}}h}}\|_{X^{0,1/3}}\\
 +&{\bf 1}_{N_{2}\sim N_{3}}\|P_{N_{1}}f_1\|_{X^{0,b_{0}}}\|P_{N_{2}}f_{2}\|_{X^{0,1/3}}\|P_{N_{2}}f_{3}\|_{X^{0,1/3}}\|P_{N_{1}}h\|_{X^{0,1/3}}
 \\
+&{\bf 1}_{N_{1}\sim N_{2}\sim N_3}\min_i\left(\|P_{N_{1}}h\|_{X^{0,1-b_{0}}}\|f_{i}\|_{X^{0,b_{0}}}\sup_{|n|\sim N_{1}}\prod_{j\neq i}\|f_{j}(n)e^{in\cdot x}\|_{X^{0,b_{0}}}\right),
\end{aligned}
\end{equation}
\begin{equation}\label{deter4}
\begin{aligned}
&\left|\int \psi(t)\bar{h}\NN_{1}(P_{N_{1}}f_{1},P_{N_{2}}{f}_{2},P_{N_{3}}f_{3})\right|\\
\lesssim &(\max(N_{2},N_{3}))^{C\epsilon_{0}}
\|h\|_{X^{0,1/3}}\|P_{N_{2}}f_{2}\|_{L_{t,x}^\infty}\|P_{N_{1}}f_{1}\|_{X^{0,1/3}}\|P_{N_{3}}f_{3}\|_{X^{0,1/3}}\\
+&{\bf 1}_{N_{1}\sim N_{2}}
\|P_{N_{1}}f_{1}\|_{X^{0,1/3}}\|P_{N_{1}}f_{2}\|_{X^{0,b_{0}}}\|P_{N_{3}}f_{3}\|_{X^{0,1/3}}\|{\color{black}{P_{N_{3}}h}}\|_{X^{0,1/3}}
\\
+&{\bf 1}_{N_{2}\sim N_{3}}\|P_{N_{1}}f\|_{X^{0,1/3}}\|P_{N_{2}}f_{2}\|_{X^{0,b_{0}}}\|P_{N_{2}}f_{3}\|_{X^{0,1/3}}\|P_{N_{1}}h\|_{X^{0,1/3}}
\\
+& {\bf 1}_{N_{1}\sim N_{2}\sim N_3}\min_i\left(\|P_{N_{1}}h\|_{X^{0,1-b_{0}}}\|f_{i}\|_{X^{0,b_{0}}}\sup_{|n|\sim N_{1}}\prod_{j\neq i}\|f_{j}(n)e^{in\cdot x}\|_{X^{0,b_{0}}}\right),
\end{aligned}
\end{equation}
\begin{equation}\label{deter4.5}
\begin{aligned}
&\left|\int \psi(t)\bar{h}\NN_{1}(P_{N_{1}}f_{1},P_{N_{2}}{f}_{2},P_{N_{3}}f_{3})\right|\\
\lesssim &(\max(N_{2},N_{3}))^{C\epsilon_{0}}
\|h\|_{X^{0,1/3}}\|P_{N_{2}}f_{2}\|_{X^{0,1/3}}\|P_{N_{1}}f_{1}\|_{X^{0,1/3}}\|P_{N_{3}}f_{3}\|_{L_{t,x}^{\infty}}\\
+&{\bf 1}_{N_{1}\sim N_{2}}
\|P_{N_{1}}f_{1}\|_{X^{0,1/3}}\|P_{N_{1}}f_{2}\|_{X^{0,1/3}}\|P_{N_{3}}f_{3}\|_{X^{0,b_{0}}}\|{\color{black}{P_{N_{3}}h}}\|_{X^{0,1/3}}\\
+& {\bf 1}_{N_{2}\sim N_{3}}\|P_{N_{1}}f\|_{X^{0,1/3}}\|P_{N_{2}}f_{2}\|_{X^{0,1/3}}\|P_{N_{2}}f_{3}\|_{X^{0,b_{0}}}\|P_{N_{1}}h\|_{X^{0,1/3}}
\\
+& {\bf 1}_{N_{1}\sim N_{2}\sim N_3}\min_i\left(\|P_{N_{1}}h\|_{X^{0,1-b_{0}}}\|f_{i}\|_{X^{0,b_{0}}}\sup_{|n|\sim N_{1}}\prod_{j\neq i}\|f_{j}(n)e^{in\cdot x}\|_{X^{0,b_{0}}}\right),
\end{aligned}
\end{equation}
where with ${\bf 1}_{N_i\sim N_j}$ we denote the indicator ${\bf 1}_{N_i\sim N_j}=1$ if  $N_{i}\sim N_{j}$, $0$ otherwise. Moreover, estimate \eqref{deter3}, \eqref{deter3.5}, \eqref{deter4}, \eqref{deter4.5} are also valid if one replaces $\NN_{1}(P_{N_{1}}f_{1},P_{N_{2}}f_{2},P_{N_{3}}f_{3})$ by $P_{N_{1}}f_{1}\overline{P_{N_{2}}{f}_{2}}P_{N_{3}}f_{3}$.
\end{lem}
Similarly, one also has 
\begin{lem}\label{lem: onemoredeter}
If $N_{2}\gtrsim N_{1}\geq N_{3}$, one has 
\begin{equation}\label{deter5}
\left|\int \phi(t/\delta)\bar{h}P_{N_{1}}f_{1}\overline{P_{N_{2}}f_{2}}P_{N_{3}}f_{3}\right|\lesssim \delta^{s_{1}/8}N_{1}^{s_{1}}N_{1}^{C\epsilon_{0}}\|h\|_{X^{0,1-b_{0}}}\prod_{i}\|P_{N_{i}}f_{i}\|_{X^{0,b_{0}}},
\end{equation}
\begin{equation}\label{deter6}
\left|\int \phi(t)\bar{h}P_{N_{1}}f_{1}\overline{P_{N_{2}}f_{2}}P_{N_{3}}f_{3}\right|\lesssim N_{1}^{C\epsilon_{0}}\|h\|_{X^{0,1/3}}\|P_{N_{i}}f_i\|_{L_{t,x}^{\infty}}\prod_{j\neq i}\|P_{N_{i}}f_{i}\|_{X^{0,1/3}}, \quad i=1,2,3.
\end{equation}
\end{lem}

We also record the following deterministic estimate which will (almost directly) handle the $\NN_{2}$ part in the Wick ordered nonlinearity.

\begin{lem}\label{lemma: N2}
Let $b_{0}=\frac{1}{2}+\epsilon_{0}$, then
\begin{equation}\label{deter8}
\left|\int \phi(t/\delta)\bar{h}\NN_{2}(P_{N_{1}}f_{1}, P_{N_{1}}f_{2}, P_{N_{1}}f_{3})\right|\lesssim 
\min_{i}\left(\delta^{1/4}\|{\color{black}{P_{N_{1}}h}}\|_{X^{0,1-b_{0}}}\|f_{i}\|_{X^{0,b_{0}}}\sup_{|n|\sim N_{1}}\prod_{j\neq i}\|f_{j}(n)e^{in\cdot x}\|_{X^{0,b_{0}}}\right).
\end{equation}
\end{lem}
We sketch the proofs of Lemma \ref{lem: deter}, \ref{lem: onemoredeter} and \ref{lemma: N2} in Appendix \ref{asecdeter} for the convenience of the reader. In the following, we provide a proof of the easier estimate in our main result Proposition \ref{prop: onemore}.

\begin{proof}[Proof of (\ref{eqn: mainest2}) of Proposition \ref{prop: onemore}]
We choose $N_{0}$ large, up to dropping a set of probability $e^{-N_{0}^{c_{s_{1}}}}$, we have
\begin{equation}
\begin{cases}
|g_{n}(\omega)|\leq N_{0}^{s_{1}}, |n|\leq N_{0},\\
|g_{n}(\omega)|\leq |n|^{s_{1}}, |n|\geq N_{0}.
\end{cases}
\end{equation}

And in particular, no matter whether $u_{i}=w_{i}$ such that $\|w_{i}\|_{X^{s_{0},b_{0}}}\lesssim 1$ or $u_{i}=\phi(t)\sum_{|n|\sim N_{i}}\frac{g_{n}(\omega)}{|n|}e^{in^{2}t}e^{inx}$, we always have
\begin{equation}
\begin{cases}
\|P_{N}u_{i}|_{X^{0,b_{0}}}\lesssim N_{0}^{s_{1}},  N\leq N_{0},\\
\|P_{N}u_{i}\|_{X^{0,b_{0}}}\lesssim N^{s_{1}}, |N|\geq N_{0},
\end{cases}
\end{equation}
Observe that
\begin{equation}
\int \phi(t/\delta)\bar{h}\NN_{2}(u_{1}, u_{2}, u_{3})=\sum_{N}\int \phi(t/\delta)P_{N}\bar{h}\NN_{2}(P_{N}u_{1}, P_{N}u_{2}, P_{N}u_{3})
\end{equation}

Let $\delta=N_{0}^{-100}$, applying estimate \eqref{deter8}, we have
\begin{itemize}
\item $N\leq N_{0}$
\begin{equation}
\int \phi(t/\delta)P_{N}\bar{h}\NN_{2}(P_{N}u_{1}, P_{N}u_{2}, P_{N}u_{3})\lesssim \delta^{1/4}N^{Cs_{1}},
\end{equation}
\item $N\geq N_{0}$,
\begin{equation}
\int \phi(t/\delta)P_{N}\bar{h}\NN_{2}(P_{N}u_{1}, P_{N}u_{2}, P_{N}u_{3})\lesssim N^{Cs_{1}}N^{-2},
\end{equation}
\end{itemize}
Sum all $N$, (when $s_{0}$ small enough),  and desired estimate follow.
\end{proof}

\subsection{Probabilistic estimates} 
We collect the elementary but crucial probabilistic estimates here.
\begin{lem}\label{lem: generalwiener}
Let $\{g_{n}(\omega)\}$ be i.i.d complex Gaussian on the probability space $\Omega$, and $\{c_{n_1,\ldots,n_k}\}$ be a sequence of complex numbers for some integer $k\geq 1$. Define
\[
F_{k}(\omega)=\sum_{n_1,\,\ldots,\,n_k} c_{n_{1},\ldots,n_k}g_{n_{1}}g_{n_2}\cdots g_{n_{k}}.
\]Then one has for all $1<p<\infty$
\begin{equation}
\|F_{k}\|_{L^{p}(\Omega)}\lesssim \sqrt{k+1}(p-1)^{k/2}\|F_{k}\|_{L^{2}(\Omega)}.
\end{equation}Moreover, there holds the associated large deviation type estimate
\begin{equation}
\PPP\{|F_{k}|>\lambda\}\leq \exp\left( \frac{-C\lambda^{2/k}}{\|F_{k}\|_{L^{2}(\Omega)}^{2/k}}\right),\qquad \forall \lambda>0.
\end{equation}
\end{lem}
 In the lemma above, it is very important that $\{c_{n_{1}, \ldots ,n_{k}}\}$ are numbers instead of random variables.
One may refer to \cite{thomann2010gibbs}, \cite{tzvetkov2010construction}.

The following  lemma will also be frequently used.
\begin{lem}\label{lem: linfinity}
Let $\{g_{n}(\omega)\}$ be i.i.d complex Gaussian on the probability space $\Omega$, and assume that
\[
\sum_{n\in \ZZZ^{2}}|a_{n}|^{2}\lesssim 1.
\]
Then, for any integer $N>0$, up to a set of probability measure $e^{-N^{\alpha}}$ for some $\alpha>0$ depending on $\epsilon$, there holds
\begin{equation}
\left\|\sum_{n\in \ZZZ^{2},|n|\leq N}a_{n}g_{n}(\omega)e^{in\cdot x}\right\|_{L^{\infty}(\TTT^{2})}\lesssim N^\epsilon,\quad \forall \epsilon>0.
\end{equation}
\end{lem} 
\begin{proof}
It is easy to see that for any fixed $x$, the function is bounded as desired outside a small exceptional set, so the key point is to show that the exceptional set can be made independent of $x$. To do this, given $\epsilon>0$, first note that $\mathbb{T}^2$ can be covered by a mesh of size $1/ N^{M}\times 1/N^{M}$ centered at $\sim N^{2M}$ lattice points for a large number $M$ to be determined later. We first bound the function at the lattice points, which is easy as the function at each lattice point has size $\lesssim N^{\epsilon_0}$ (for some $\epsilon_0<\epsilon$) up to probability $ e^{-N^{\alpha(\epsilon_0)}}$ according to Lemma \ref{lem: generalwiener}, and there are only $N^{2M}$ many points. Therefore, one has that the function satisfies the desired bound outside an exceptional set of measure up to $\sim e^{-N^{\alpha-}}$.   %The trick is to use mesh of size $1/ N^{100}\times 1/N^{100}$ to cover the $\mathbb{T}^{2}$.
%On  those $N^{10000}$ lattice points, the function is of size 1 up to probability $ e^{-N^{\alpha}}N^{10000}\sim e^{-N^{\alpha-}}$.

To pass from here to the bound of the function on the entire $\mathbb{T}^2$, it suffices to obtain a uniform control of the derivative of the function (independent of $x$):
\begin{equation}
N\sum_{|n|\leq N}|a_{n}||g_{n}(\omega)|\lesssim N^2\sup_{n} |g_{n}(\omega)|.
\end{equation} 
The probability of the derivative being larger than $N^{4}$ is smaller than $e^{-N^{2}}N^{2}$, as the probability of $|g_n(\omega)|>N^2$ for each $n$ is controlled by $e^{-N^2}$. Hence, by removing this additional exceptional set and recalling that every point $x$ lies within the distance of $1/N^{M}$ from some lattice point, one has that the function at $x$ is bounded by $N^{\epsilon_0}+N^4\cdot 1/N^M\lesssim N^\epsilon$ as long as $M$ is chosen sufficiently large. 
\end{proof}

%\section{reduction to resonance system}
%
%\subsection{interpolation with $L^{\infty}$ reduce to localzed $Xsb$}
%\begin{lem}\label{lem: interpolate}
%\begin{equation}\label{eq: interpo}
%\|u\|_{L_{t,x}^{3}([0,1]\times \mathbb{T}^{2})}\lesssim \|u\|_{X^{0+},\frac{1}{3}+}
%\end{equation}
%\end{lem}
%\begin{proof}
%Recall
%\begin{equation}
%\|u\|_{L_{t,x}^{4}([0,1]\times \TTT^{2})}\lesssim \|u\|_{X^{0+,\frac{1}{2}+}}
%\end{equation}
%and
%\begin{equation}
%\|u\|_{L_{t,x}}^{2}([0,1]\times \mathbb{T}^{2})\lesssim \|u\|_{X^{0,0}}
%\end{equation}
%Thus, one derived that via (real) interpolation formula \eqref{eq: interpo}.
%
%\end{proof}
%\section{counting lemma}
%\begin{lem}\label{lem:counting 1}
%Let $\kappa\ll 1$, Let $S$ be an $O(N^{\kappa})$  neiborhood of a circle of radius 1. (The irrationally is in the )
%\end{lem}
%\begin{lem}\label{lem:counting 2}
%retangle counting
%\end{lem}
\section{Counting lemma}\label{sec: counting}
One of the key ingredients in the proof of our main theorem is an extension of the lattice counting argument of Bourgain, \cite{bourgain1996invariant} to the irrational setting. We present them in this section. We start with two  auxiliary lemmata. The first has a geometric flavor, while the second is an elementary number theoretical result.

\begin{lem}\label{lem: circlepart}
Let $A$ be the $O(\frac{1}{N})$-neighborhood of a circle of radius $\sim N$, and $\Lambda=\mathbb{Z}\times \gamma\mathbb{Z}$ for some real number $\gamma\in (1,2)$. Suppose $A_1\subset A$ is the $O(\frac{1}{N})$ neighborhood of an arc of the circle of length $N_2\leq N$. Then, $A_1$ contains at most $\max\left(\frac{N_2}{N^{1/3}},1\right)$ points of $\Lambda$.
\end{lem}

\begin{proof}
Let $C\subset A$ be the $O(\frac{1}{N})$-neighborhood of any arc of the circle of angular size $\theta=\frac{1}{1000}N^{-2/3}$, then it suffices to show that $C$ contains at most $O(1)$ points from $\Lambda$. Indeed, if $N_2>N^{1/3}$, $A_1$ corresponds to an arc of angular size $\frac{N_2}{N}$, which can be decomposed into $\sim\frac{N_2}{N^{1/3}}$ smaller arcs each of which containing at most $O(1)$ points from $\Lambda$. 

Denote $B_1$ the circular sector bounded by the outer arc of $C$, and $B_2$ the triangle with vertices being the center of the circle and the two endpoints of the inner arc of C. Observe that any triangle $P_1P_2P_3$ with $P_i\in \Lambda\cap C$ must be contained in the region $B_1-B_2$. Moreover, it is easy to see that annulus $A$ can contain straight line of length at most $O(1)$. Therefore, suppose $C$ contains more than $O(1)$ points from $\Lambda$, then there must exist three points $P_1,P_2,P_3\in \Lambda\cap C$ that formulate a non-degenerate triangle. By definition of $\Lambda$, the area of the triangle is at least $\frac{1}{2}$, hence the area of $B_1-B_2$ needs to be at least $\frac{1}{2}$ as well.

On the other hand, via Taylor expansion, the area of $B_1-B_2$ is bounded by $\frac{1}{2}N^2(\theta-\sin\theta)+O(\frac{1}{N})N\theta\lesssim N^2\theta^3+\theta\leq \frac{1}{10}$, which is a contradiction. Therefore, $C$ must contain at most $O(1)$ points from $\Lambda$ and the proof is complete.
\end{proof}

\begin{lem}\label{lem:countpart}
Given an integer $M\neq 0$, then 
\[
\#\{(a,b)\in \mathbb{Z}\times \mathbb{Z}:\, ab=M\}\leq C_\epsilon |M|^\epsilon,\quad \forall \epsilon>0.
\]
\end{lem}
\begin{proof}
Without loss of generality we assume that $M>0$. As an integer, $M$ has an unique representation by its prime factors:
\[
M=p_1^{r_1}p_2^{r_2}\cdots p_\ell^{r_\ell},\quad p_1<p_2<\cdots <p_\ell,\quad r_i>0,\,\forall i.
\]Then, the total number of pairs of integers whose product is $M$ is bounded by $\prod_{i=1}^\ell (r_i+1)$. For any fixed $\epsilon>0$, there exists a smallest integer $N$ such that $N^\epsilon>2$. Let $p_j$ be the first factor that is larger than $N$, then there holds
\[
\prod_{i=j}^\ell(r_i+1)\leq \prod_{i=j}^\ell 2^{r_i}<\prod_{i=j}^\ell p_i^{r_{i}\epsilon}<M^{\epsilon}.
\]On the other hand, there are only $O_\epsilon(1)$ many $p_i$ that are smaller than $N$. Therefore, write $M=e^m$, one has
\[
\prod_{i=1}^{j-1} (r_i+1) \leq (\log M)^{O_\epsilon(1)}=m^{O_\epsilon(1)}.
\]There exists a large number $M_0=M_0(\epsilon)$ so that $m^{O_\epsilon(1)}\leq e^{\epsilon m}$ whenever $M=e^m>M_0$, hence the desired estimate follows if $M>M_0$. If $M\leq M_0$, one can simply take $C_\epsilon=M_0^2$. The proof is complete.
\end{proof}

We now fix  $\mu$ and $N_1, N_2, N_3\geq 0$, and we let
\[
S:=\{(n_1,n_2,n_3):\, |n_i|\sim N_i,\, n_2\neq n_1, n_3,\, \text{and }\langle n_2-n_1,n_2-n_3\rangle=\mu+O(1)\}.
\]
We observe here   that in the  rational case $S$ is a curve while in the general case, since $+O(1)$ appears the set is thick. 

Define
\[
S(n_1)=\{(n_2, n_3):\, (n_1, n_2, n_3)\in S\},
\]and similarly for other $S(n_i)$, if $n_i$ is fixed and  $S(n_i, n_j)$, if $n_i, n_j$ are fixed. We have the following counting lemmata regarding the size of these sets. In the following, we sometimes use $N^1, N^2, N^3$ to denote $N_1, N_2, N_3$ rearranged in the non-increasing order and assume $\mu=O(N^1)$. 

\begin{lem}\label{lem: linecounting}
$\# S(n_1, n_2)\lesssim N_3$ and $\#S(n_2,n_3)\lesssim N_1$.
\end{lem}
Compared to the rational case studied by Bourgain, this estimate is equally good, since ultimately it is a linear estimate.
\begin{proof}
We will only prove the first estimate, as the second one follows from the same argument. Fixing $n_1\neq n_2$, one has from
\[
\langle n_2-n_1,n_2-n_3\rangle=\mu+O(1)
\]that $n_3$ lies in an $O(\frac{1}{|n_1-n_2|})\leq O(1)$ neighborhood of a straight line. Since $|n_3|\sim N_3$, there are at most $\sim N_3$ choices of $n_3$.
\end{proof}

It is in the next lemmata that one sees a difference with respect to the estimates of Bourgain that are generated by the possible irrationality of the torus.
\begin{lem}\label{lem: circle}
Assume $N_1 \geq N_2, N_3$. Then, 
\[
\#S(n_1,n_3)\lesssim \begin{cases} N_2^{2/3}, & \text{if } N_1\sim N_3\gg N_2,\\ \max(\frac{N_2}{(N_1)^{1/3}},1), & \text{otherwise}.\end{cases}
\]
\end{lem}
\begin{proof}
From the definition of the set $S$, with $n_1, n_3$ fixed, $n_2$ must lie in an annulus given by the formula
\[
\left|n_2-\frac{n_1+n_3}{2}\right|^2=\frac{|n_1-n_3|^2}{4}+\mu+O(1).
\]Denote the inner and outer radius of the annulus by $R_1,R_2$ respectively and recall that $\mu\lesssim N_1$. 

Therefore, when $N_1\gg N_3$, both the inner and outer radius are roughly $\sim \sqrt{(N_1)^2+\mu}\sim N_1$. In order to determine the thickness of the annulus, one observes from $R_1^2-R_2^2=O(1)$ and $R_1, R_2\sim N_1$ that there holds $R_1-R_2\leq O(\frac{1}{N_1})$, hence the thickness is bounded by $O(\frac{1}{N_1})$. Then, the desired estimate $\max(\frac{N_2}{N_1^{1/3}},1)$ follows immediately from $|n_2|\sim N_2\leq N_1$ and Lemma \ref{lem: circlepart} above.

When $N_1\sim N_3$, assume that the inner and outer radius are roughly $\sim R\gg 1$ (if $R\leq O(1)$, the estimate is trivial). Note that $R\lesssim N_1$. Then $n_2$ lies inside an annulus of radius $\sim R$ and thickness bounded by $O(\frac{1}{R})$. Suppose $N_1\sim N_2\sim N_3$, then again by Lemma \ref{lem: circlepart} above, the total number of $n_2$ is bounded by 
\[\max\left(\frac{R}{R^{1/3}},1\right)=\max(R^{2/3},1)\lesssim \max(N_1^{2/3},1)\sim \max\left(\frac{N_2}{N_1^{1/3}},1\right).\]

On the other hand, if $N_1\sim N_3\gg N_2$, still denoting $R$ as roughly the inner and outer radius of the annulus, one has 
\[
\#S(n_1,n_3)\lesssim \max\left(\frac{\min(N_2,R)}{R^{1/3}},1\right)\lesssim N_2^{2/3}.
\]
\end{proof}

Lemma \ref{lem: circle} above can be extended to estimate other sets of similar type. For example, let $n:=n_1-n_2+n_3$ and suppose that $N_1=N^1$. Then for any fixed $n_2$ and $n$, via a similar argument, one has the following estimate:
\[
\#\{n_1:\, |n_i|\sim N_i,\, n_2\neq n_1,n_3,\, \text{and }\langle n_1-n, n_2-n_1\rangle=\mu+O(1)\}\lesssim N_1^{2/3}.
\]

Indeed, suppose $N_1\gg N_2$, then $n_1$ in the above lies in an annulus of radius $\sim N_1$ and thickness $\sim O(\frac{1}{N_1})$. Hence by Lemma \ref{lem: circlepart}, the total number is at most $\max(\frac{N_1}{N_1^{1/3}},1)=N_1^{2/3}$. Otherwise, if $N_1\sim N_2$, one has $R$, the radius of the annulus, is bounded by $N_1$. Hence, the total possible number of $n_1$ is at most $\lesssim R^{2/3}\lesssim N_1^{2/3}$.

Similarly, when $N_1\gg N_2$, one also has the following counting
\[
\#\{n_{3}:\, \langle n_{3}-n_{2}, n_{3}-n\rangle=\mu+O(1)\}\lesssim \max(\frac{N_3}{N_1^{1/3}},1).
\]

Moreover, from the two lemmata above, one can already obtain some estimate for sets $S(n_i)$. For instance, by first fixing $n_2$ and applying Lemma \ref{lem: linecounting}, one can show that $\#S(n_1)\lesssim N_2^2N_3$. Depending on the relative sizes the  $N_i$, sometimes such estimates are already good enough. However, in some other cases one needs 
to use a  more sophisticated argument, and this is the contents of the following counting lemma.

\begin{lem}\label{lem: count strong}
$\# S(n_1)\lesssim (N^1)^\epsilon N_2N_3$, $\# S(n_2)\lesssim (N^1)^\epsilon N_1N_3$, and $\# S(n_3)\lesssim (N^1)^\epsilon N_1N_2$.
\end{lem}
\begin{proof}
We only prove the estimate of $\# S(n_1)$ as the other two can be treated very similarly. Write $n_2-n_1=r(a, \gamma b)$, where $r\in\mathbb{N}$, $a, b\in \mathbb{Z}$ and $(a,b)=1$. Decompose all choices of $n_2$ into dyadic scales. In other words, at each scale, we have dyadic number $A, B\in\mathbb{Z}$ fixed such that $|a|\sim A$, $|b|\sim B$, and there holds $A, B\lesssim \max(N_1, N_2)$. We also write $n_2-n_3=(x, \gamma y)$, $x,y\in\mathbb{Z}$.

Assume $a,b\neq 0$ and fix $A, B, r$. We want to count the number of $(a,b,x,y)$ satisfying
\[
r(ax+\gamma^2by)=\mu+O(1),\quad \text{and }|a|\sim A, \, |b|\sim B.
\]Note that $r\neq 0$ because $n_1\neq n_2$, and $x,y$ cannot both be zero as $n_2\neq n_3$. Without loss of generality, suppose $y\neq 0$, then the equality above can be rewritten as
\[
\frac{ax}{by}+\gamma^2=\frac{\mu}{byr}+O(\frac{1}{|byr|}).
\]Since $a,b,x,y\in\mathbb{Z}$, for any fixed value $H=by$, the value of $ax$ is inside an $O(1)$-neighborhood of an integer $G=G(H)$.

Moreover, observe that the number of possible values of $H$ is bounded by $\sim BN_3$, as for each fixed $b$ (hence the second coordinate of $n_2$ is fixed) there are $\sim N_3$ many choices of $y$. Then, by a simple number theory observation (Lemma \ref{lem:countpart} below) one has for any $\epsilon>0$ that
\[
\#(a,b,x,y)\lesssim \#\{H\}\cdot |H|^\epsilon \cdot |G|^\epsilon\lesssim BN_3 (B\max(N_2,N_3))^\epsilon (A\max(N_2,N_3))^\epsilon\lesssim \max(A,B)^{1+\epsilon}N_3\max(N_2,N_3)^\epsilon.
\]

It is thus left to sum over $r$ and then $A,B$. Note that for fixed $A, B$,
\[
\# r\lesssim\frac{N_2}{\max(A,B)},
\]therefore, one has in this case that
\[
\# S(n_1)\lesssim \sum_{A, B} \frac{N_2}{\max(A,B)} \max(A,B)^{1+\epsilon}N_3\max(N_2,N_3)^\epsilon\lesssim \log(N^1)^2 (N^1)^\epsilon N_2 N_3\lesssim (N^1)^\epsilon N_2N_3.
\]

Assume now that $a=0$ (then $b\neq 0$ as $n_2\neq n_1$). This means $n_1,n_2$ have the same first coordinate, hence the total number of choices of $n_2$ is bounded by $N_2$. Moreover, one has that the first coordinate of $n_3$ is free and its second coordinate is determined by
\[
\langle n_2-n_1, n_2-n_3\rangle=\mu+O(1),
\]hence is inside an $O(1)$-neighborhood of a determined value. Indeed, the formula above can be written as
\[
r\gamma^2 by=\mu+O(1)
\]which implies
\[
y=\frac{\mu}{r\gamma^2b}+O\left(\frac{1}{|r\gamma^2b|}\right).
\]Therefore, in this case one has $\# S(n_1)\lesssim N_2N_3$. The $b=0$ case can be treated in the same way which we omit.
\end{proof}

\section{Proof of Proposition \ref{prop: onemore}: case by case study, case 0}\label{sec: ProofPartI}
In this section, we treat the  case: $N_2 \gg N_1\geq N_3$.   We will  prove
\begin{equation}\label{eq: mainestimate0}
\left|\sum_{N_{2}\gg N_{1}\geq N_{3}}N_{2}^{s_{0}}\int \NN_{1}(P_{N_{1}}u_{1}, P_{N_{2}}u_{2},P_{N_{3}}u_{3})\bar{h}\phi(t/\delta)\right|\lesssim \delta^{\epsilon_{1}}, \text{ for some }  \epsilon_{1}\gg \epsilon_{0},
\end{equation}
where  $u_i$ is either $w_i$ with $\|w_i\|_{X^{s_0,b_0}}\leq 1$ or
\[
u_i=\phi(t)e^{it\Delta_\gamma}u_0=\phi(t)e^{it\Delta_\gamma}\sum_{n}\frac{g_n(\omega)}{|n|}e^{i n\cdot x},\quad x\in \mathbb{T}^2,
\]and $\|h\|_{X^{0,1-b_{0}}}\leq 1$.

First observe that, in this case, the Wick ordered nonlinearity is the same as the usual cubic nonlinearity, i.e. 
$\NN=\NN_{1}$ and $ \NN_{2}=0$. 
We only need to prove (up to an exceptional set of probability $e^{-\delta^{-c}}$) that
\begin{equation}
\left|\sum_{N_{2}\gg N_{1}\geq N_{3}}N_{2}^{s_{0}}\int_{\RRR\times \TTT^{2}}P_{N_{1}}u_{1}\overline{P_{N_{2}}{u}_{2}}P_{N_{3}}u_{3}\bar{h}\phi(t/\delta)\right|\lesssim \delta^{\epsilon_{1}} \text{ for some } \epsilon_{1}\gg \epsilon_{0}.
\end{equation}

There are several subcases.
We start with \textbf{subcase 1} : $u_{i}=w_{i}, i=1,2,3$. 
 Let $s_{1}$ be chosen such that $\epsilon_{0}\ll s_{1}\ll s_{0}$. Via \eqref{deter5} and observing that one has $h=P_{N_2}h$ in the integral in this case, we obtain
\begin{equation}
\begin{aligned}
&\left|\int_{\RRR\times \TTT^{2}}P_{N_{1}}u_{1}\overline{P_{N_{2}}{u}_{2}}P_{N_{3}}u_{3}\bar{h}\phi(t/\delta)\right|\\
\lesssim 
&N_{1}^{2s_{1}}\|P_{N_{1}}u_{1}\|_{X^{0,b_{0}}}\|P_{N_{2}}u_{2}\|_{X^{0,b_{0}}}\|P_{ N_{3}}u_{3}\|_{X^{0,b_{0}}}\|P_{N_{2}}h\|_{X^{0,1-b_{0}}}\delta^{\frac{s_{1}}{8}}\\
\lesssim &N_{1}^{-s_{0}/2}\|u_{1}\|_{X^{s_{0},b_{0}}}\|u_{3}\|_{X^{s_{0},b_{0}}}\|P_{N_{2}}u_{2}\|_{X^{0,b_{0}}}\|P_{N_{2}}h\|_{X^{0,b_{0}}}.
\end{aligned}
\end{equation}
Sum over $N_{2}\gg N_{1}\geq N_{3}$, and the desired estimate follows.

Next we discuss \textbf{subcase 2}: at least one $u_{i}$ is $\phi(t)\sum_{|n_{i}|\sim N_{i}}\frac{g_{n_{i}}}{|n_{i}|}e^{in_{ix}}e^{in_{i}^{2}t}$.  
We only study the case $u_{1}=\phi(t)\sum_{|n_{1}|\sim N_{1}}\frac{g_{n_{1}}}{|n_{1}|}e^{in_{ix}}e^{in_{1}^{2}t}$, as other cases can be treated similarly.

Let $N_{2,0}$ be a large parameter such that $N_{2,0}^{100}=\frac{1}{\delta}$. Note that up to an exceptional set of probability $e^{-\delta^{-c}}\sim e^{-N_{2,0}^{c'}}$, we have

\begin{equation}\label{eq: randomn2}
\begin{cases}
|g_{n}(\omega)|\leq N_{2,0}^{s_{1}}, & |n|\leq N_{2,0},\\
|g_{n}(\omega)|\leq N_{2}^{s_{1}},& |n|\sim N_{2}\geq N_{2,0}.
\end{cases}
\end{equation}
In particular, one always has for all $u_{i}$ that
\begin{equation}\label{eq: randomxsb}
\begin{cases}
\|P_{N}u_{i}\|_{X^{0,b_{0}}}\lesssim N_{2,0}^{s_{1}}, & N\leq N_{2,0},\\
\|P_{N}u_{i}\|_{X^{0,b_{0}}}\lesssim N^{s_{1}}, &N\geq N_{2,0}. 
\end{cases}
\end{equation}
and, dropping another exceptional set of probability $e^{-\delta^{-c}}$ if necessary, one has 
\begin{equation}\label{eq: randomlinfty}
\|P_{N_{1}}u_{1}\|_{L^{\infty}_{t,x}}\leq N_{1}^{s_{1}}, \quad N_{1}\leq N_{2}, \quad N_{2,0}\leq N_{2}.
\end{equation}
 
 Now we split the \textbf{subcase 2} further into the following subsubcases.

In \textbf{subsubcase 2.1}, we restrict ourselves to the regime $N_{2}\leq N_{2,0}$, and use estimate \eqref{deter5} to derive
\begin{equation}\label{eq:deter121}
\left|\int_{\RRR\times \TTT^{2}}P_{N_{1}}u_{1}\overline{P_{N_{2}}{u_{2}}}u_{3}\bar{h}\phi(t/\delta)\right|
\lesssim \delta^{s_
{1}/8}N_{1}^{2s_{1}}\|u_{1}\|_{X^{0,b_{0}}}\|u_{2}\|_{X^{0,b_{0}}}\|u_{3}\|_{X^{0,b_{0}}}\|h\|_{X^{0,1-b_{0}}}
\lesssim \delta^{s_{1}}N_{1}^{2s_{1}}.
\end{equation}
Summing over $N_{1},N_{3}\leq N_{2}\leq N_{2,0}$, one obtains $\lesssim \delta^{s_{1}} N_{2,0}^{3s_1}$ and the desired estimate follows.

We are left with \textbf{subsubcase 2.2}, where $N_{2}\geq N_{2,0}$. We will prove
\begin{equation}\label{eq: xsbinterapp}
\left|\int_{\RRR\times \TTT^{2}}P_{N_{1}}\phi(t)u_{1}\overline{P_{N_{2}}\phi(t){u_{2}}}\phi(t)P_{N_{3}}u_{3}\phi(t)\bar{h}\right|
\lesssim N_{2}^{-1/10},
\end{equation}which will then imply the desired estimate by summing \eqref{eq: xsbinterapp} over $N_{1},N_{2},N_{3}$.

As remarked above, we only prove estimate \eqref{eq: xsbinterapp} for the case that $u_{1}$ is random. Using \eqref{deter6} and Lemma \ref{lem: linfinity}, one derives
 \begin{equation}\label{eqn: case0 1}
 \begin{aligned}
&\left|\int_{\RRR\times \TTT^{2}}P_{N_{1}}\phi(t)u_{1}\overline{P_{N_{2}}\phi(t){u_{2}}}\phi(t)P_{N_{3}}u_{3}\phi(t)\bar{h}\right|\\
\lesssim &N_{2}^{2s_{1}}\|\phi(t)P_{N_{2}}u_{2}\|_{X^{0,1/3}}\|\phi(t)P_{N_{3}}u_{3}\|_{X^{0,1/3}}\|\phi(t)P_{N_{2}}h\|_{X^{0,1/3}}.
\end{aligned}
\end{equation}

Let $F_{i}(\tau_{i},n_{i})$ be the space-time Fourier transform of $\phi(t)u_{i}$, $i=1,2,3$, and $F_{4}(\tau_{4},n_{4})$ be the Fourier transform of $\phi(t)h$, the integral being estimated is non-zero only if
\begin{equation}
\begin{cases}
n_{1}-n_{2}+n_{3}-n_{4}=0,\\
\tau_{1}-\tau_{2}+\tau_{3}-\tau_{4}=0,
\end{cases}
\end{equation}
which implies
\begin{equation}
 \sum_{i}(-1)^{i}(\tau_{i}-n^{2}_{i})=n_{1}^{2}-n_{2}^{2}+n_{3}^{2}-n_{4}^{2}=-2\langle n_{2}-n_{1},n_{2}-n_{3}\rangle \sim N_{2}^{2}, \text{ since } N_{2}\gg N_{1}, N_{3}.
\end{equation}

Observe that the Fourier transform of $\phi(t)u_{1}$ is essentially  supported on $|\tau_{1}-n_{1}^{2}|\lesssim 1$, thus, at least for one $i\in \{2,3,4\}$, one has $N_{2}^{2}\lesssim |\tau_{i}-n_{i}^{2}|$, hence one can upgrade estimate \eqref{eqn: case0 1} to 
 \begin{equation}
 \begin{aligned}
&\left|\int_{\RRR\times \TTT^{2}}P_{N_{1}}\phi(t)u_{1}\overline{P_{N_{2}}\phi(t){u_{2}}}\phi(t)P_{N_{3}}u_{3}\phi(t)\bar{h}\right|\\
\lesssim 
& N_{2}^{2s_{1}} N_{2}^{{\color{black}{2(-1/6-\epsilon_{0})}}}\|\phi(t)P_{N_{2}}u_{2}\|_{X^{0,b_{0}}}\|\phi(t)P_{N_{3}}u_{3}\|_{X^{0,b_{0}}}\|\phi(t)P_{N_{2}}h\|_{X^{0,1-b_{0}}}\\
\lesssim
&N_{2}^{Cs_{1}}N_{2}^{{\color{black}{-1/6}}}\lesssim N_{2}^{-1/10}.
\end{aligned}
\end{equation}

To make the above argument rigorous, one may decompose $\phi(t)u_{1}$ into
\begin{equation}\label{eqn: case0 2}
(P_{|\tau|\leq N_{2}}\phi(t))u_{1}+ (P_{|\tau|>N_{2}}\phi(t))u_{1},
\end{equation}
where the first term corresponds to frequency localization  at $|\tau_{1}-n_{1}^{2}|\leq N_{2}\ll N_{2}^{2}$, and  hence the above argument can be applied. For the second term, one simply observes that
\begin{equation}
\|P_{|\tau|>N_{2}}\phi(t)\|_{L_{t}^{\infty}}\lesssim N_{2}^{-100}.
\end{equation}
This concludes the proof.
\section{Proof of Proposition \ref{prop: onemore}: case by case study, case $(a)$}\label{sec: ProofPartII}
In this section, we consider case (a): $N_{2}(I)\lesssim N_1(II),  \quad N_2(I)\geq N_3(II)$. We aim to prove for all $w_1, v_2, w_3$ satisfying
\[
\|w_{1}\|_{X^{s_{0},b_{0}}}\lesssim 1, \qquad \|w_{3}\|_{X^{s_{0}, b_{0}}}\lesssim 1,\qquad v_{2}=\phi(t)e^{it\Delta_\gamma}\left(\sum_{n_2}\frac{g_{n_2}(\omega)}{|n_2|}e^{in_2\cdot x}\right),
\]and $\|h\|_{X^{0,1-b_{0}}}\lesssim 1$ that, up to an exceptional set
\begin{equation}\label{eq: mainestimatea}
\left|\sum_{N_{2}\lesssim N_{1}, N_{2}\geq N_{3}}N_{1}^{s_{0}}\int \NN_{1}(P_{N_{1}}w_{1}, P_{N_{2}}v_{2},P_{N_{3}}w_{3})\bar{h}\phi(t/\delta)\right|\lesssim \delta^{\epsilon_{1}}, \text{ for some }\epsilon_{1}\gg \epsilon_{0}.
\end{equation}

Fix $N_{2,0}^{100}=\frac{1}{\delta}$ and recall that any loss of $\delta^{-C\epsilon_{0}}$ will be irrelevant in the analysis. The values of the parameters $\epsilon_{0}\ll s_{1}\ll s_{0}$ will be determined later. 

By dropping a set of probability $e^{-\delta^{-c_{s_{1}}}}$, we will assume the following throughout the whole section:

\medskip

\begin{equation}\label{eq: crudepro}
\begin{cases}
|g_{n}(\omega)|\leq N_{2,0},& |n|\leq N_{2,0},\\
|g_{n}(\omega)|\leq |n|^{s_{1}},& |n|> N_{2,0}.
\end{cases}
\end{equation}
And one has in particular 
\begin{equation}\label{eq: crudexsb}
\begin{cases}
\|P_{N_{2}}\phi(t)v_{2}\|_{X^{0,b_{0}}}\lesssim N_{2,0}^{s_{1}}, & N_{2}\leq N_{2,0},\\
\|P_{N_{2}}\phi(t)v_{2}\|_{X^{0,b_{0}}}\lesssim N_{2}^{s_{1}}, & N_{2}> N_{2,0}.
\end{cases}
\end{equation}

\subsection{Standard reduction}\label{subsec: sr}
The goal in this subsection is to reduce the estimates of this case to Lemma \ref{lem: rdfa} and Lemma \ref{lem: rdfa1}, which will be stated at the end of this subsection.

Note that in the discussion of all the cases (b)--(l), there will be a similar reduction argument. We will present the full details of the reduction in this case, and only sketch it in other cases.

We first split the summation $\sum_{N_{2}\lesssim N_{1}, N_{2}\geq N_{3}}$ into two parts $N_{2}\leq N_{2,0}$ and $N_{2}> N_{2,0}$.

\subsubsection{The low frequency part: $N_2\leq N_{2,0}$}
We aim to prove
 \begin{equation}\label{eq: as1}
\sum_{N_{2}\leq N_{2,0}, N_{2}\lesssim N_{1}, N_{2}\geq N_{3}}N_{1}^{s_{0}}\left|\int \NN_{1}(P_{N_{1}}w_{1}, P_{N_{2}}v_{2},P_{N_{3}}w_{3})\bar{h}\phi(t/\delta)\right|\lesssim \delta^{\epsilon_{1}}, \text{ for some }\epsilon_{1}\gg \epsilon_{0}.
\end{equation}

Observe that, when $N_{1}\gg N_{2}$, one can replace the $h$ in \eqref{eq: as1} by $P_{N_{1}}h$, and when $N_{1}\sim N_{2}$, $h$ can be replaced by $P_{<N_{1}}h$.

Thus, via estimate \eqref{deter1.2} and \eqref{eq: crudexsb}, one has

\begin{itemize}
\item If $N_{1}\sim N_{2}$ (in particular, $N_{1}\lesssim N_{2,0}$),
\begin{equation}\label{eq: alow1}
\begin{aligned}
&N_{1}^{s_{0}}\left|\int \NN_{1}(P_{N_{1}}w_{1}, P_{N_{2}}v_{2},P_{N_{3}}w_{3})\bar{h}\phi(t/\delta)\right|\\
\lesssim 
&\delta^{s_{1}/8}N_{2}^{2s_{1}}\|P_{N_{1}}w_{1}\|_{X^{s_{0},b_{0}}}\|P_{N_{2}}v_{2}\|_{X^{0,b_{0}}}\|P_{N_{3}}w_{3}\|_{X^{0,b_{0}}}\|h\|_{X^{0,1-b_{0}}}\\
\lesssim 
&\delta^{s_{1}/8}N_{2,0}^{4s_{1}}\|P_{N_{1}}w_{1}\|_{X^{s_{0},b_{0}}}N_{1}^{-s_{1}}.
\end{aligned}
\end{equation}
\item If $N_{1}\gg N_{2}$,
\begin{equation}\label{eq: alow2}
\begin{aligned}
&N_{1}^{s_{0}}\left|\int \NN_{1}(P_{N_{1}}w_{1}, P_{N_{2}}v_{2},P_{N_{3}}w_{3})\bar{h}\phi(t/\delta)\right|\\
\lesssim
&\delta^{s_{1}/8}N_{2}^{2s_{1}}\|P_{N_{1}}w_{1}\|_{X^{s_{0},b_{0}}}\|P_{N_{2}}v_{2}\|_{X^{0,b_{0}}}\|w_{3}\|_{X^{0,b_{0}}}\|P_{N_{1}}h\|_{X^{0,1-b_{0}}}\\
\lesssim 
&\delta^{s_{1}/8}N_{2}^{3s_{1}}\|P_{N_{1}}w_{1}\|_{X^{s_{0},b_{0}}}\|P_{N_{1}}h\|_{X^{0,1-b_{0}}}.
\end{aligned}
\end{equation}
 \end{itemize}
 The desired estimate will follow if one sums over the associated $N_{1}, N_{2}, N_{3}$ and apply Cauchy inequality in the sum on $N_{1}$.

\begin{rem}\label{rmk: practicallwp}
We point out that the low frequency case is always the easier part in random data problems, and essentially follows  from  deterministic estimates usually used in the local well-posedness argument, we will not repeat this part in the rest of the article.
\end{rem}

\subsubsection{Reduction to resonant part}\label{disnoncasea}
Now we are left with the case $N_{2}> N_{2,0}$, we aim to prove
\begin{equation}
  \sum_{N_{1}\geq N_{2}\geq N_{3}, N_{2}> N_{2,0}}N_1^{s_0}\left|\int \phi(t/\delta)\bar{h} \NN_1(P_{N_{1}}\phi(t/\delta)w_{1},P_{N_{2}} \phi(t)v_{2},P_{N_{3}} \phi(t/\delta)w_{3})\,dx dt\right|
\lesssim N_{2,0}^{-\epsilon_{1}}
 \end{equation}
 for some $\epsilon_{1}\gg \epsilon_{0}$.
 
 We will not explore the time localization $\phi(t/\delta)$ in this part. Observe that $\phi(t)h=\phi(t)\phi(t/\delta)h$, we may hence define $\tilde{h}$ as $\phi(t/\delta)h$ and use $\phi(t)\tilde{h}$ in the following estimate. Note that we still have $\|\tilde{h}\|_{X^{0,1-b_{0}}}\lesssim 1$. For the sake of brevity, we still denote $\tilde{h}$ as $h$.
 
 Our aim is to prove for fixed $N_{1},N_{2},N_{3}$ satisfying $N_{2}\lesssim N_{1}, N_{2}\geq N_{3}, N_{2}> N_{2,0}$ that 
 \begin{itemize}
 \item
 If $N_{1}\gg N_{2}$,
 \begin{equation}\label{eq: reduced1}
 N_1^{s_0}\left|\int \phi(t)\bar{h} \NN_1(P_{N_{1}}\phi(t)w_{1},P_{N_{2}} \phi(t)v_{2},P_{N_{3}} \phi(t)w_{3})\,dx dt\right|\lesssim N_{2}^{-\epsilon_{1}}\|P_{N_{1}}w_{1}\|_{X^{s_{0},b_{0}}}\|P_{N_{1}}h\|_{X^{0,1-b_{0}}},
 \end{equation}
 \item if $N_{1}\sim N_{2}$,
 \begin{equation}\label{eq: reduced2}
 N_1^{s_0}\left|\int \phi(t)\bar{h} \NN_1(P_{N_{1}}\phi(t)w_{1},P_{N_{2}} \phi(t)v_{2},P_{N_{3}} \phi(t)w_{3})\,dx dt\right|\lesssim N_{2}^{-\epsilon_{1}}\|P_{N_{1}}w_{1}\|_{X^{s_{0},b_{0}}}\|P_{<N_{1}}h\|_{X^{0,1-b_{0}}}.
 \end{equation}
 \end{itemize}

 We will focus on the proof of \eqref{eq: reduced1}, and it will be easy to see that \eqref{eq: reduced2} follows similarly (almost line by line).
 
 Observe that, since $N_{1}\gg N_{2}\geq N_{3}$, one has 
 \begin{equation}
 \begin{aligned}
  N_1^{s_0}\left|\int \phi(t)\bar{h} \NN_1(P_{N_{1}}\phi(t)w_{1},P_{N_{2}} \phi(t)v_{2},P_{N_{3}} \phi(t)w_{3})\,dx dt\right|\\
  =N_1^{s_0}\left|\int \phi(t)\overline{P_{N_{1}}h} \NN_1(P_{N_{1}}\phi(t)w_{1},P_{N_{2}} \phi(t)v_{2},P_{N_{3}} \phi(t)w_{3})\,dx dt\right|.
  \end{aligned}
 \end{equation}
 
To carry on the proof of \eqref{eq: reduced1}, we introduce another parameter $M=N_{2}^{100s_{1}}$. One may split $w_{i}$, $i=1,3$, and $v_{2}$ as 
 \begin{equation}\label{eq: splita}
 \phi(t)w_{i}:=P_{|\tau_{i}-n_{i}^{2}|<M}\phi(t)w_{i}+P_{|\tau_{i}-n_{i}^{2}|>M}\phi(t)w_{i}, \quad \phi(t)v_{2}(t)=P_{|\tau|<M}\phi(t)v_{2}+P_{|\tau|>M}\phi(t)v_{2},
 \end{equation}and the same for $h$. 
  
  \begin{rem}\label{rem: timelocal}
with  such a splitting one may lose the time localization. This can be  overcome by writing for example $P_{<M}\phi(t)$ as $\tilde{\phi}(t)P_{<M}\tilde{\tilde\phi}$, such that $\phi,\tilde{\phi}$ are Schwarz uniform in $M$. Or, one may further require $\phi(t)=\tilde{\psi}(t)^{4}$, and  indeed split as $\phi(t)=\phi_{1}(t)+\phi_2(t)$ where $\phi_{1}(t)=|P_{|\tau|<M}\tilde{\psi}(t)|^{4}$.  To make the proof clean, we leave further details to the interested reader, and  allow ourselves  to freely multiplying an extra time localization $\psi(t)$ in the proof.
\end{rem}

Via \eqref{eq: splita}, one can naturally split the left hand side of \eqref{eq: reduced1} into $2^{4}$ parts. Each part is of the form 
$N_{1}^{s_{0}}\left|\int \NN_1(P_{N_1}f_{1}, P_{N_2}{f}_{2}, P_{N_3}f_{3})\overline{P_{N_1}{f}_{4}}\right|$, where $f_{i}=P_{|\tau_{i}-n_{i}^{2}|<M}\phi(t)w_{i}$, or $P_{|\tau_{i}-n_{i}^{2}|>M}\phi(t)w_{i}$ for $i=1,3$, $f_{2}=P_{|\tau|<M}\phi(t)v_{2}$ or $P_{|\tau|>M}\phi(t)v_{2}$, and $f_{4}=P_{|\tau-n^{2}|<M}\phi(t)h$ or $P_{|\tau-n^{2}|>M}\phi(t)h$.

Then, applying \eqref{deter4}, one has for some Schwartz function $\psi(t)$ that
\begin{equation}\label{eq: caseakd}
\begin{aligned}
N_{1}^{s_{0}}\left|\int \NN_1(P_{N_1}f_{1}, P_{N_2}{f}_{2}, P_{N_3}f_{3})\overline{P_{N_1}{f}_{4}}\right|
\lesssim &N_1^{s_0}\left|\int \NN_1(P_{N_1}f_{1}, P_{N_2}{f}_{2}, P_{N_3}f_{3})\overline{P_{N_1}{f}_{4}}\psi(t)\right|\\
\lesssim& {\color{black}{N_{2}^{s_{1}}\|P_{N_{1}}f_{1}\|_{X^{s_{0},1/3}}\|P_{N_{2}}f_{2}\|_{L_{t,x}^{\infty}}\|P_{N_{3}}f_{3}\|_{X^{0,1/3}}\|f_{4}\|_{X^{0,1/3}}}}\\
&\qquad {\color{black}{+\|P_{N_{1}}f_{1}\|_{X^{s_{0},1/3}}\|P_{N_{2}}f_{2}\|_{X^{0,b_{0}}}\|P_{N_{2}}f_{3}\|_{X^{0,1/3}}\|P_{N_{1}}f_{4}\|_{X^{0,1/3}}}}.
\end{aligned}
\end{equation}
(In the second line, we add a time localization $\psi(t)$, following Remark \ref{rem: timelocal}.  Also recall we have $s_{1}\gg \epsilon_{0}$.)

Unless $f_{i}=P_{|\tau_{i}-n_{i}^{2}|<M}\phi(t)w_{i}$, $i=1,3$, $f_{4}=P_{<M}\phi(t)h$, and $f_{2}=P_{|\tau|<M}\phi(t)v_{2}$, at least one of the following estimates will be true (after dropping an extra set of probability $e^{-N_{2}^{c}}$ if necessary):
\begin{itemize}
\item  $\|P_{N_{1}}f_{1}\|_{X^{s_{0},1/3}}\lesssim N_{2}^{-10s_{1}}\|P_{N_{1}}f_{1}\|_{X^{s_{0},b_{0}}}$,\\
\item   $\|P_{N_{2}}f_{2}\|_{L_{t,x}^{\infty}}+\|P_{N_{2}}f_{2}\|_{X^{0,b_{0}}}\lesssim N_{2}^{-10s_{1}}$,\\
\item$\|P_{N_{3}}f_{3}\|_{X^{0,1/3}}\lesssim N_{2}^{-10s_{1}}\|P_{N_{3}}f_{3}\|_{X^{0,b_{0}}}$,\\
\item  $\|P_{N_1}f_{4}\|_{X^{0,1/3}}\lesssim N_{2}^{-10s_{1}}\|P_{N_1}f_{4}\|_{X^{0,1-b_{0}}}$,
\end{itemize}
and we always have
\begin{equation}
\|P_{N_{2}}f_{2}\|_{X^{0,b_{0}}}+\|P_{N_{2}}f_{2}\|_{L_{t,x}^{\infty}}\lesssim N_{2}^{s_{1}}.
\end{equation}

The desired estimate follows by inserting the  above ones  into \eqref{eq: caseakd}.

\begin{rem}\label{rem: m1}
The numerology in the above calculation is in fact very simple modulo lower order terms. The term $\int \NN_{1}(P_{N_{1}}u_{1},P_{N_{2}}u_{2},P_{N_{3}}u_{3})h\psi(t)$ can essentially  be thought as $\int P_{N_{1}}u_{1}\overline{P_{N_{2}}u_{2}}P_{N_{3}}u_{3}\bar{h}\psi(t)$, and will only miss the desired estimate by at most a factor $N_{2}^{10s_{1}}$ via \eqref{deter1}. On the other hand, when there is some $u_{i}=v_{i}$, which is hence already essentially localized at $|\tau-n^{2}|\lesssim 1$, then for all the rest of the functions $h$, $u_{j}$, one can gain at least $1/2-\epsilon_{0}-1/3$ derivative. Therefore, unless all the other terms have space-time frequency localization in $|\tau-n^{2}|<N_{2}^{100s_{1}}$, the desired estimate will automatically follow.
\end{rem}

 Now, we are left with  the case where  $f_{i}=P_{|\tau_{i}-n_{i}^{2}|<M}\phi(t)w_{i}$, $i=1,3$, $f_{4}=P_{|\tau-n^2|<M}\phi(t)h$, and $f_{2}=P_{|\tau|<M}\phi(t)v_{2}$.

 Let $d_{1}(n,t), r_{2}(n,t),d_{3}(n,t), H(n,t)$ be the space Fourier transform of $P_{|\tau_{1}-n_{1}^{2}|<M}\phi(t)w_{1}$, $P_{|\tau|<M}\phi(t)v_{2}$, $P_{|\tau_{3}-n_{3}^{2}|<M}\phi(t)w_{3}$, $P_{|\tau-n^{2}|<M}\phi(t)P_{N_{1}}h$.
 We abbreviate $d_{1}(n,t)$, $r_{2}(n,t)$, $d_{3}(n,t)$, $H(n,t)$ as $d_{1}(n)$, $r_{2}(n)$, $d_{3}(n)$, $H(n)$ respectively.
 Observe that
 \begin{equation}\label{eq: basea}
 \begin{cases}
 N_{i}^{2s_{0}}\sum_{n\sim N_{i}} \|d_{i}(n)e^{-in_i^{2}t}\|_{H_{t}^{b_{0}}}\lesssim \|P_{N_{1}}w_{i}\|^{2}_{X^{s_{0},b_{0}}},\\
 r_{2}(n_{2},t)=\psi(t)\frac{g_{n_{2}}}{|n_{2}|}, \text{ for some Schwartz function $\psi$},\\
\sum_{|n|\sim N_{1}} \|H(n)e^{-in^{2}t}\|_{H_{t}^{1-b_{0}}}\lesssim  \|P_{N_{1}}h\|_{X^{0,1-b_{0}}}^{2}.
 \end{cases}
 \end{equation}
 (One may observe, for example,  that $\|d_{i}(n,t)e^{-in^{2}t}\|_{H_{t}^{b_{0}}}\sim \|d_{i}(n,t)e^{in\cdot x}\|_{X^{s_{0},b_{0}}}$. We also point out that we have estimated $P_{|\tau|<M}\phi(t)$ just as some Schwartz function $\psi(t)$. Furthermore, one may observe that
 $\|d_{i}(n,t)e^{-in^{2}t}\|_{L_{t}^{p}}=\|d_{i}(n,t)\|_{L_{t}^{p}}$.) We will show that
 \begin{equation}\label{eq: rr1}
 \begin{aligned}
 &N_1^{s_0}\left|\int \phi(t)\bar{h} \NN_1(P_{N_{1}}\phi(t)w_{1},P_{N_{2}} \phi(t)v_{2},P_{N_{3}} \phi(t)w_{3})\,dx dt\right|\\
=&N_{1}^{s_{0}}\left|\sum_{|n_{i}|\sim N_{i}, n_{1}-n_{2}+n_{3}=n}\int d_{1}(n_{1})\overline{r_{2}(n_{2})}d_{3}(n_{3})\overline{H(n)}\,dt\right|\\
 \lesssim &N_{2}^{-\epsilon_{1}}\|P_{N_{1}}w_{1}\|_{X^{s_{0},b_{0}}}\|P_{N_{1}}h\|_{X^{0,1-b_{0}}}\|P_{N_{3}}w_{3}\|_{X^{s_{0},b_{0}}}, \text{ for some } \epsilon_{1}\gg\epsilon _{0}.
 \end{aligned}
 \end{equation}
 
Observe further that $d_{1}(n_{1},t)$, $r_{2}(n_{2},t)$, $d_{3}(n_{3},t)$, and $H(n,t)$ are Fourier supported in $|\tau_{i}-n_{i}^{2}|\lesssim M, i=1,2,3$ and $|n-\tau^{2}|\lesssim M$. Thus for the integral $\int d_{1}(n_{1})\overline{r_{2}(n_{2})}d_{3}(n_{3})\overline{H(n)}\,dt$ to be non-zero, one necessarily has
 \begin{equation} 
 n_{1}^{2}-n_{2}^{2}+n_{3}^{2}-n^{2}=O(N_{2}^{100s_{1}}).
 \end{equation}We thus have
 \begin{equation}
 \begin{aligned}
 &N_{1}^{s_{0}}\left|\sum_{|n_{i}|\sim N_{i}, n_{1}-n_{2}+n_{3}=n}\int d_{1}(n_{1})\overline{r_{2}(n_{2})}d_{3}(n_{3})\overline{H(n)}\,dt\right|\\ \leq
 &N_1^{s_0}\sum_{
\begin{subarray}
\quad |n_{i}|\sim N_{i}, n_{1}-n_{2}+n_{3}-n_{4}=0,\\
n_{1}^{2}-n_{2}^{2}+n_{3}^{2}-n^{2}=O(N_{2}^{100s_{1}})
\end{subarray}} \left|\int d_{1}(n_{1})\overline{r_{2}(n_{2})}d_{3}(n_{3})\overline{H(n)}\,dt\right|.
\end{aligned}
 \end{equation}

 To summarize, to prove \eqref{eq: reduced1}, we are left with showing the following:
 \begin{lem}\label{lem: rdfa}
 Let $N_{1}\gg N_{2}\geq N_{3}$, then one has for some $\epsilon_1\gg \epsilon_0$ that
 \begin{equation}\label{eq: rdfa}
 \begin{aligned}
 &N_1^{s_0}\sum_{
\begin{subarray}
\quad |n_{i}|\sim N_{i}, n_{1}-n_{2}+n_{3}-n_{4}=0,\\
n_{1}^{2}-n_{2}^{2}+n_{3}^{2}-n^{2}=O(N_{2}^{100s_{1}})
\end{subarray}} \left|\int d_{1}(n_{1})\overline{r_{2}(n_{2})}d_{3}(n_{3})\overline{H(n)}\,dt\right|\\
\lesssim &
N_{2}^{-\epsilon_{1}}\|P_{N_{1}}w_{1}\|_{X^{s_{0},b_{0}}}\|P_{N_{1}}h\|_{X^{0,1-b_{0}}}\|P_{N_{3}}w_{3}\|_{X^{s_{0},b_{0}}}.
 \end{aligned}
 \end{equation}
 
 \end{lem}
 
 We also write down the corresponding lemma that will imply \eqref{eq: reduced2}.
 \begin{lem}\label{lem: rdfa1}
Let $N_{1}\sim N_{2}\geq N_3$, then the same estimate (\ref{eq: rdfa}) holds if one replaces the $P_{N_{1}}h$ by $P_{<N_{1}}h$. 
\end{lem}
One can easily check  that the proof of Lemma \ref{lem: rdfa} also works for Lemma \ref{lem: rdfa1} (almost line by line).
\subsection{Random data type estimate: Proof of Lemma \ref{lem: rdfa}}
 Recall that we always assume \eqref{eq: crudepro} and that we are in the regime $N_{1}\gg N_{2}\geq N_{3}$.

First note that for all $n_{3}\sim N_{3}$, we have $$\|d_{3}(n_{3})\|_{L_{t}^{\infty}}=\|d_{3}(n_{3})e^{-in_{3}^{2}t}\|_{L_{t}^{\infty}}\lesssim \|d_{3}(n_{3})e^{-in_{3}^{2}t}\|_{H_{t}^{b_{0}}}\sim \|d_{3}(n_{3})e^{in_{3}\cdot x}\|_{X^{0,b_{0}}}\leq \|P_{N_{3}}w_{3}\|_{X^{0,b_{0}}}. $$

Now, for all $|n_{3}|\sim N_{3}$ fixed, we have
\begin{equation}\label{eq: a1}
\begin{aligned}
&\sum_{
\begin{subarray}
\quad |n_{i}|\sim N_{i}, i=1,2, \,n_{1}-n_{2}+n_{3}-n=0,\\
n_{1}^{2}-n_{2}^{2}+n_{3}^{2}-n^{2}=O(N_{2}^{100s_{1}})
\end{subarray} }\left|\int d_{1}(n_{1})\overline{r_{2}(n_{2})}d_{3}(n_{3})\overline{H(n)}\,dt\right|\\
\lesssim
&\|P_{N_{3}}w_{3}\|_{X^{0,b_{0}}}\sum_{\begin{subarray}
\quad |n_{i}|\sim N_{i}, i=1,2, \,n_{1}-n_{2}+n_{3}-n=0,\\
n_{1}^{2}-n_{2}^{2}+n_{3}^{2}-n^{2}=O(N_{2}^{100s_{1}})
\end{subarray}} \|d_{1}(n_{1})\|_{L_{t}^{2}}\|r_{2}(n_{2})\|_{L_{t}^{\infty}}\|H(n)\|_{L_{t}^{2}}\\
\lesssim
&\|P_{N_{3}}w_{3}\|_{X^{0,b_{0}}} \left(\sum_{|n|\sim N_{1}}\|H(n)\|^{2}_{L_{t}^{2}}\right)^{1/2}
\times\left( \sum_{|n|\sim N_{1}}\Big\{\sum_{\begin{subarray}
\quad |n_{i}|\sim N_{i}, i=1,2, \,n_{1}-n_{2}+n_{3}-n=0,\\
n_{1}^{2}-n_{2}^{2}+n_{3}^{2}-n^{2}=O(N_{2}^{100s_{1}})
\end{subarray}} \|d_{1}(n_{1})\|_{L_{t}^{2}}\|r_{2}(n_{2})\|_{L_{t}^{\infty}}\Big\}^{2}\right)^{1/2}\\
\lesssim
&\|P_{N_{3}}w_{3}\|_{X^{0,b_{0}}}\|P_{N_{1}}h\|_{X^{0,1-b_{0}}}
\times\left( N_{2}^{1+100s_1}\sum_{\begin{subarray}
\quad |n_{i}|\sim N_{i}, i=1,2,\, n_{1}-n_{2}+n_{3}-n=0,\\
n_{1}^{2}-n_{2}^{2}+n_{3}^{2}-n^{2}=O(N_{2}^{100s_{1}})
\end{subarray}}\{ \|d_{1}(n_{1})\|_{L_{t}^{2}}\|r_{2}(n_{2})\|_{L_{t}^{\infty}}\}^{2}\right)^{1/2}\\
\lesssim &\|P_{N_{3}}w_{3}\|_{X^{0,b_{0}}}\|P_{N_{1}}h\|_{X^{0,1-b_{0}}}
\times\left( N_{2}^{1+100s_1}\sum_{\begin{subarray}
\quad |n_{i}|\sim N_{i}, i=1,2,\langle n_{2}-n_{1}, n_{2}-n_{3} \rangle=O(N_{2}^{100s_{1}})
\end{subarray}} \|d_{1}(n_{1})\|^{2}_{L_{t}^{2}}N_{2}^{-2+2s_{1}}\right)^{1/2}.
\end{aligned}
\end{equation}
In the second step above, we used Cauchy inequality in $n$, while in the second to last step, we used Lemma \ref{lem: linecounting}. Observe as well that $\|H(n)\|_{L_{t}^{2}}\leq \|H(n)e^{-in{t}}\|_{H_{t}^{1-b_{0}}}$, $\|r_{2}(n_{2})\|_{L_{t}^{\infty}}\lesssim |N_{2}|^{s_{1}-1}$ up to an exceptional set, and $\|d_{1}(n_{1})\|_{L_{t}^{2}}\leq \|d_{1}(n_{1})e^{-in_{1}^{2}t}\|_{L_{t}^{2}}$.

Furthermore, by the counting lemma (Lemma \ref{lem: circle}), one has for $n_{1}, n_{3}$ fixed that
\begin{equation}
\sharp\{|n_{2}|\sim N_{2}:\,\langle n_{2}-n_{1}, n_{2}-n_{3} \rangle=O(N_{2}^{100s_{1}})\}\lesssim N_{2}^{100s_{1}}\max\left(\frac{N_{2}}{N_{1}^{1/3}}, 1\right).
\end{equation}
To summarize, we derived that
\begin{equation}\label{eq: a2}
\begin{aligned}
& N_1^{s_0}\sum_{
\begin{subarray}
\quad |n_{i}|\sim N_{i}, n_{1}-n_{2}+n_{3}-n_{4}=0,\\
n_{1}^{2}-n_{2}^{2}+n_{3}^{2}-n^{2}=O(N_{2}^{100s_{1}})
\end{subarray}} \left|\int d_{1}(n_{1})\overline{r_{2}(n_{2})}d_{3}(n_{3})\overline{H(n)}\,dt\right|\\
\lesssim 
&N_{1}^{s_{0}}N_{3}^{2}\|P_{N_{3}}w_{3}\|_{X^{0,b_{0}}}\|P_{N_{1}}h\|_{X^{0,1-b_{0}}}
\left(
N_{2}^{-1+2s_{1}}\sum_{|n_{1}|\sim N_{1}} \|d_{1}(n_{1})\|^{2}_{L_{t}^{2}}N_{2}^{200s_{1}}\max(\frac{N_{2}}{N_{1}^{1/3}},1)
\right)^{1/2}\\
\lesssim
&N_{3}^{2}N_{2}^{Cs_{1}}N_{2}^{-1/2}\left[\max\left(\frac{N_{2}}{N_{1}^{1/3}},1\right)
\right]^{1/2}\|P_{N_{1}}w_{1}\|_{X^{s_{0},b_{0}}}\|P_{N_{1}}h\|_{X^{0,1-b_{0}}}\|P_{N_{3}}w_{3}\|_{X^{s_{0},b_{0}}}.
\end{aligned}
\end{equation}
It is easy to see that the desired estimate will follow if there holds $N_{2}^{1/100}\geq N_{3}$.

When $N_{3}\geq N_{2}^{\frac{1}{100}}$, we may directly go back to \eqref{eq: reduced1}. Applying \eqref{deter2} and using \eqref{eq: crudexsb}, we have
\begin{equation}
\begin{aligned}
&N_1^{s_0}\left|\int \phi(t)\bar{h} \NN_1(P_{N_{1}}\phi(t)w_{1},P_{N_{2}} \phi(t)v_{2},P_{N_{3}} \phi(t)w_{3})\,dx dt\right|\\
\lesssim
&N_{2}^{Cs_{1}}\|P_{N_{1}}w_{1}\|_{X^{s_{0},b_{0}}}\|P_{N_{2}}\phi(t)v_{2}\|_{X^{0,b_{0}}}N_{3}^{-s_{0}}\|P_{3}w_{3}\|_{X^{s_{0,b_{0}}}}\|P_{N_{1}}h\|_{X^{0,1-b_{0}}}\\
\lesssim
&N_{2}^{Cs_{1}}N_{3}^{-s_{0}}\|P_{N_{1}}w_{1}\|_{X^{s_{0},b_{0}}}\|P_{N_{1}}h\|_{X^{0,1-b_{0}}}.
\end{aligned}
\end{equation}

Estimate \eqref{eq: reduced1} then follows since $N_{3}\geq N_{2}^{1/100}$ and $s_{1}\ll s_{0}$.

\section{Proof of Proposition \ref{prop: onemore}: case by case study, case $(c)$}\label{sec: ProofPartIII}
In this case, we have $N_{1}(I)\gtrsim N_2(II)\geq N_{3}(II)$, and we aim to prove
for all $v_{1}, w_{2}, w_{3}$ satisfying 
$$v_{1}=\phi(t)\sum_{n_{1}}\frac{g_{n_{1}}(\omega)}{|n_{1}|}e^{in_{1}x}, \quad \|w_{2}\|_{X^{s_{0},b_{0}}}\lesssim 1, \quad \|w_{3}\|_{X^{s_{0,b_{0}}}}\lesssim 1,$$ and $\|h\|_{X_{0,1-b_{0}}}\lesssim 1$, that (up to an exceptional set)
\begin{equation}\label{eq: ckey}
N_1^{s_0}\left|\int \phi(t/\delta)\bar{h}\NN_1(P_{N_{1}}v_{1}, P_{N_{2}}w_{2}, P_{N_{3}}w_{3})\,dxdt\right|\lesssim \delta^{\epsilon_{1}}, \text{ for some } \epsilon_{1}\gg \epsilon.
\end{equation}
\subsection{Standard reduction: a (detailed) sketch}
We first sketch a reduction, with an argument similar to the one   in Subsection \ref{subsec: sr}. There is indeed some difference between the reduction process in case (a) and case (c), mainly due to the difference of the form of the first term (with the highest frequency). Hence, we will still provide a rather detailed sketch. In all the remaining cases, we only briefly sketch the reduction.

 We may fix $N_{1,0}$ large satisfying $\frac{1}{\delta}=N_{1,0}^{100}$. By dropping a set of probability up to $e^{-N_{1,0}^{-c_{s_{1}}}}$, we have

\begin{equation}\label{eq: caseccrudepro}
\begin{cases}
|g_{n}(\omega)|\leq N_{1,0}^{s_{1}}, & |n|\leq N_{1,0},\\
|g_{n}(\omega)|\leq N_{1}^{s_{1}}, & |n|\sim N_{1}\geq N_{1,0}.
\end{cases}
\end{equation}
By further dropping a set of probability $e^{-N_{1,0}^{-c_{s_{1}}}}$ if necessary, one has 
\begin{equation}\label{eq: casecxsbpro}
\begin{cases}
\|P_{N_{1}}\phi(t)v_{1}\|_{X^{0,b_{0}}}\lesssim N_{1,0}^{s_{1}}, & N_{1}\leq N_{1,0},\\
\|P_{N_{1}}\phi(t)v_{1}\|_{X^{0,b_{0}}}+\|P_{N_{1}}\phi(t)v_{1}\|_{L_{t,x}^{\infty}}\lesssim N_{1}^{s_{1}}, & N_{1}\geq N_{1,0}.
\end{cases}
\end{equation}
We will assume \eqref{eq: caseccrudepro} and \eqref{eq: casecxsbpro} thoughout this section.
Now, split into two parts $N_{1}> N_{1,0}$ and $N_{1}\leq N_{1,0}$. 
For the low frequency part $N_{1}\leq  N_{1,0}$, we may use  \eqref{eq: casecxsbpro} and apply the deterministic estimate \eqref{deter1.2}, one thus derives the analogues of \eqref{eq: alow1} and \eqref{eq: alow2} below
\begin{equation}
\begin{aligned}
&N_{1}^{s_{0}}\left|\int \NN_{1}(P_{N_{1}}v_{1}, P_{N_{2}}w_{2},P_{N_{3}}w_{3})\bar{h}\phi(t/\delta)\right|\\
\lesssim 
&\delta^{s_{0}/8}N_{1,0}^{s_{0}+C\epsilon_{0}}N_{2}^{s_{0}}\|\phi(t)v_{1}\|_{X^{0,b_{0}}}\|w_{2}\|_{X^{s_{0},b_{0}}}\|w_{3}\|_{X^{s_{0},b_{0}}}\|h\|_{X^{0,1-b_{0}}}\\
\lesssim
&\delta^{s_{0}/8}N_{1,0}^{s_{0}+C\epsilon_{0}}N_{2}^{s_{0}}N_{1,0}^{s_{1}}.
\end{aligned}
\end{equation}
(Note that here we only need one estimate rather than two estimates as in \eqref{eq: alow1}, \eqref{eq: alow2}.)

Summing over $N_{1}\leq N_{1,0}$ and the associated $N_{2}, N_{3}$, and using the fact that $\delta^{-1}=N_{1,0}^{100}$, we derive the desired estimate
\begin{equation}
\sum_{N_{1}\leq N_{1,0}, N_1\gtrsim N_{2}\geq N_{3}}N_{1}^{s_{0}}\left|\int \NN_{1}(P_{N_{1}}v_{1}, P_{N_{2}}w_{2},P_{N_{3}}w_{3})\bar{h}\phi(t/\delta)\right|\lesssim \delta^{\epsilon_{1}}, \text{ for some }\epsilon_{1}\gg \epsilon_{0}.
\end{equation}

For the remaining part $N_{1}> N_{1,0}$, we will write $\phi(t/\delta)h$ as $\phi(t)\phi(t/\delta)h$ and note that one still has  $\|\phi(t/\delta)h\|_{X^{0,1-b_{0}}}\lesssim 1$. For notational convenience, we will still denote $\phi(t/\delta)h$ by $h$, and will prove for all $N_1\gtrsim N_{2}\geq N_{3}$ with $N_1>N_{1,0}$ that
\begin{equation}\label{eq: casecreduced}
N_1^{s_0}\left|\int \NN_{1}(P_{N_{1}}v_{1}, P_{N_{2}}w_{2},P_{N_{3}}w_{3})\bar{h}\phi(t)\right|\lesssim N_{1}^{-\epsilon_{1}}, \text{ for some } \epsilon_{1}\gg \epsilon_{0}.
\end{equation}
Then \eqref{eq: ckey} will follow from summing \eqref{eq: casecreduced} over $N_1, N_2, N_3$.

To see \eqref{eq: casecreduced}, we first introduce a parameter $M=N_{1}^{100s_{0}}$. 
\begin{rem}\label{rem: s0}
If one wants to get a rather large $s_{0}<1$, one may need to choose $M$ more carefully. The following argument should still be fine if one chooses $M=O(N_{1}^{(6+)s_{0}})$, where $6+$ denotes any number larger than $6$. However, it is unclear to us whether further improvement is possible. We don't further discuss this issue here. 
\end{rem}

As in Subsubsection \ref{disnoncasea}, we may split the functions $\phi(t)v_1=P_{|\tau|<M}\phi(t)v_{1}+P_{|\tau|>M}\phi(t)v_{1}$, $\phi(t)w_{i}=P_{|\tau_i-n_i^{2}|<M}\phi(t)w_{i}+P_{|\tau_i-n_i^{2}|>M}\phi(t)w_{i}$, $i=2,3$, and $\phi(t)h=P_{|\tau-n^{2}|<M}\phi(t)h+P_{|\tau-n^{2}|>M}\phi(t)h$. Applying the deterministic estimate \eqref{deter3.5}, we reduce the proof of  \eqref{eq: casecreduced} to the following estimate
\begin{equation}\label{eq: prerdfc}
N_1^{s_0}\left|\int \NN_{1}(P_{N_{1}}\psi(t)v_{1}, P_{N_{2}}P_{|\tau_{2}-n_{2}^{2}|\leq M}\phi(t)w_{2},P_{N_{3}}P_{|\tau_{3}-n^{2}_{3}|\leq M}\phi(t)w_{3})\overline{P_{|\tau-n^{2}|\leq M}\phi(t){h}}\right|\lesssim N_{1}^{-\epsilon_{1}}, 
\end{equation}for some $\epsilon_{1}\gg \epsilon_{0}$. Here $\psi(t)=P_{|\tau|<M}\phi(t)$ is a Schwartz function.

Write
\begin{equation}\label{eq: basecc}
\psi(t)v_{1}=\sum_{n_{1}} r_{1}(n_{1},t)e^{in_{1}\cdot x}, \,\, P_{|\tau_{i}-n_{i}^{2}|}\phi(t)w_{i}=\sum_{n_{i}}d_{i}(n_{i},t)e^{in_{i}\cdot x}, \,i=2,3, \,\, P_{|\tau-n^{2}|<M}\phi(t)h=\sum_{n}H(n,t)e^{in\cdot x},
\end{equation}and abbreviate the coefficients as $r_1(n_1)$, $d_i(n_i)$ and $H(n)$ as before, one has
\begin{equation}
\begin{aligned}
&N_1^{s_0}\left|\int \NN_{1}(P_{N_{1}}\psi(t)v_{1}, P_{N_{2}}P_{|\tau_{2}-n_{2}^{2}|\leq M}\phi(t)w_{2},P_{N_{3}}P_{|\tau_{3}-n^{2}_{3}|\leq M}\phi(t)w_{3})\overline{P_{|\tau-n^{2}|\leq M}\phi(t){h}}\right|\\
\leq& N_1^{s_0}\left|\sum_{
\begin{subarray}
\quad n_{i}\sim N_{i}, n_{1}-n_{2}+n_{3}=n, \, n_{2}\neq n_{1},n_{3}\\
n_{1}^{2}-n_{2}^{2}+n_{3}^{2}-n^{2}=O(M)=O(N_{1}^{100s_{0}})
\end{subarray}
}
\int r_{1}(n_{1})\overline{d_{2}(n_{2})}d_{3}(n_{3})\overline{H(n)}\,dt\right|.
\end{aligned}
\end{equation}
Observe that one has in this case the following estimates:
\begin{equation}\label{eq: basec}
\begin{cases}
N_{i}^{2s_{0}}\sum_{n_{i}}\|d_{i}(n_{i})e^{-n_{i}^{2}t}\|_{H_{t}^{b_{0}}}^{2}\lesssim \|P_{N_{i}}w_{i}\|_{X^{s_{0},b_{0}}}^{2}\lesssim 1,\, i=2.3,\\
r_{1}(n_{1},t)=\psi(t)\frac{g_{n_{1}}(\omega)}{|n_{1}|}e^{in_{1}\cdot x+in_{1}^{2}t}, \text{ where } \psi \text{ Schwartz},\\
\sum_{n\sim N_{1}} \|H(n)e^{-in^{2}t}\|_{H_{t}^{1-b_{0}}}^{2}\lesssim \|P_{N_{1}}h\|_{X^{0,1-b_{0}}}^{2}.
\end{cases}
\end{equation}
We also point out that $\|f(t)e^{i\theta t}\|_{L_{t}^{p}}=\|f\|_{L_{t}^{p}}.$
Thus, it remains to prove the following lemma: 
\begin{lem}\label{lem: rdfc}
Assuming \eqref{eq: basec}, for $N_{1}> N_{1,0}$, one has (up to an extra exceptional set of probability $e^{-N_{1}^{c}}$) that
\begin{equation}\label{eqn: lemma c}
N_1^{s_0}\left| \sum_{
\begin{subarray}
\quad n_{i}\sim N_{i}, n_{1}-n_{2}+n_{3}=n,  \, n_{2}\neq n_{1},n_{3}\\
n_{1}^{2}-n_{2}^{2}+n_{3}^{2}-n^{2}=O(M)=O(N_{1}^{100s_{0}})
\end{subarray}
}
\int r_{1}(n_{1})\overline{d_{2}(n_{2})}d_{3}(n_{3})\overline{H(n)}\,dt\right|\lesssim N_{1}^{-\epsilon_{1}}, \text{ for some } \epsilon_{1}\gg \epsilon_{0}.
\end{equation}
\end{lem}

\subsection{Random data type estimate: Proof of Lemma \ref{lem: rdfc}}

We derive three different estimates, which, combined together, will imply the desired bound.

First, one can directly go back to \eqref{eq: casecreduced}, and use estimate \eqref{deter2} and \eqref{eq: casecxsbpro}
to derive
\begin{equation}\label{eq: c3}
\begin{aligned}
&N_{1}^{s_{0}}\left|\int \NN_{1}(P_{N_{1}}v_{1}, P_{N_{2}}w_{2},P_{N_{3}}w_{3})\bar{h}\phi(t)\right|\\
\lesssim 
&N_{1}^{s_{0}}N_{1}^{s_{1}+C\epsilon_{0}}\|P_{N_{1}}v_{1}\|_{X^{0,b_{0}}}N_{2}^{-s_{0}}\|P_{N_{2}}w_{2}\|_{X^{s_{0},b_{0}}}N_{3}^{-s_{0}}\|P_{N_{3}}w_{3}\|_{X^{s_{0},b_{0}}}\\
\lesssim &(N_{1}N_{2}^{-1}N_{3}^{-1})^{s_{0}}N_{1}^{Cs_{1}}.
\end{aligned}
\end{equation}

One can see easily that  the same bound  works for the left hand side of (\ref{eqn: lemma c}) as well. When $N_{1}\sim N_{2}$, one can directly use \eqref{eq: c3} to derive the desire estimate unless $\ln N_{3}\ll \ln N_{1}$. In particular, There is no need to consider the subcase $N_{1}\sim N_{2}\sim N_3$.

Next, by applying Cauchy inequality in $n$, one obtains
\begin{equation}\label{eq: starcc}
\begin{aligned}
&N_1^{s_0} \left|\sum_{
\begin{subarray}
\quad n_{i}\sim N_{i}, n_{1}-n_{2}+n_{3}=n,  \, n_{2}\neq n_{1},n_{3}\\
n_{1}^{2}-n_{2}^{2}+n_{3}^{2}-n^{2}=O(M)=O(N_{1}^{100s_{0}})
\end{subarray}
}
\int r_{1}(n_{1})\overline{d_{2}(n_{2})}d_{3}(n_{3})\overline{H(n)}\,dt\right|\\
\lesssim& N_1^{s_0}\left(\sum_{n} \|H(n)\|_{L_{t}^{2}}^{2}\right)^{1/2}\left(\sum_{n} 
\Big\|\sum_{
\begin{subarray}
\quad n_{i}\sim N_{i}, n_{1}-n_{2}+n_{3}=n, \,n_{2}\neq n_{1},n_{3}\\
n_{1}^{2}-n_{2}^{2}+n_{3}^{2}-n^{2}=O(N_{1}^{100s_{0}})
\end{subarray}
}r_{1}(n_{1})d_{2}(n_{2})d_{3}(n_{3})\Big\|_{L_{t}^{2}}
^{2}\right)^{1/2}\\
\lesssim &N_1^{s_0}\left(\sum_{n} 
\Big\|\sum_{
\begin{subarray}
\quad n_{i}\sim N_{i}, n_{1}-n_{2}+n_{3}=n, \,n_{2}\neq n_{1},n_{3}\\
n_{1}^{2}-n_{2}^{2}+n_{3}^{2}-n^{2}=O(N_{1}^{100s_{0}})
\end{subarray}
}r_{1}(n_{1})d_{2}(n_{2})d_{3}(n_{3})\Big\|_{L_{t}^{2}}
^{2}\right)^{1/2}
\end{aligned}
\end{equation}
In all the summations below, we always have $|n_{i}|\sim N_{i}$,
$n_{1}-n_{2}+n_{3}=n$,  $n_{1}^{2}-n_{2}^{2}+n_{3}^{2}-n^{2}=O(M), n_{2}\neq n_{1}, n_{3}$, and we sometimes omit them for notational convenience.

One also observes that $n_{1}-n_{2}+n_{3}=n$ and $n_{1}^{2}-n_{2}^{2}+n_{3}^{2}-n^{2}=O(M)$ imply
\begin{equation}
\langle n_{2}-n_{1}, n_{2}-n_{3}\rangle=O(M), \quad \langle n_{3}-n, n_{3}-n_{2} \rangle=O(M).
\end{equation}

We have the following further estimate:
\begin{equation}\label{eq: c1}
\begin{aligned}
&\sum_{n} 
\Big\|\sum_{
\begin{subarray}
\quad n_{i}\sim N_{i}, n_{1}-n_{2}+n_{3}=n, \,n_{2}\neq n_{1},n_{3}\\
n_{1}^{2}-n_{2}^{2}+n_{3}^{2}-n^{2}=O(N_{1}^{100s_{0}})
\end{subarray}
}r_{1}(n_{1})d_{2}(n_{2})d_{3}(n_{3})\Big\|_{L_{t}^{2}}
^{2}\\
\lesssim & \sum_n \left(\sum_{n_2} \|d_2(n_2)\|_{L_t^\infty} \Big\|\sum_{n_1,n_3} r_1(n_1)d_3(n_3) \Big\|_{L_t^2}\right)^2\\
\lesssim 
&\left(\sum_{n_{2}}\|d_{2}(n_{2})\|^{2}_{L_{t}^{\infty}}\right)\sum_{n,n_{2}}\Big\|
\sum_{n_{1}, n_{3}\neq n_{2}, \langle n_{2}-n_{1}, n_{2}-n_{3}\rangle=O(M) }
\frac{g_{n_{1}}}{|n_{1}|}e^{in_{1}^{2}t}\psi(t)d_{3}(n_{3})\Big\|_{L_{t}^{2}}^2,
\end{aligned}
\end{equation}where we used Cauchy inequality in $n_{2}$ in the last line. Since $\|d_2(n_2)\|_{L^\infty_t}\lesssim \|d_{2}(n_{2})e^{in_{2}\cdot x}\|_{X^{0,b_{0}}}$, when $N_1\gg N_2$, the above is further bounded by
\begin{equation}
\begin{aligned}
\lesssim &N_{2}^{-2s_{0}}\sum_{n,n_{2}}\Big\|
\sum_{n_{1}, n_{3}\neq n_{2}, \langle n_{2}-n_{1}, n_{2}-n_{3}\rangle=O(M) }
\frac{g_{n_{1}}}{|n_{1}|}e^{in_{1}^{2}t}\psi(t)d_{3}(n_{3})\Big\|_{L_{t}^{2}}^2\\
\lesssim &N_{2}^{-2s_{0}}\sup_{n,n_{2}}\sharp\{n_{3}:\, \langle n_{3}-n_{2}, n_{3}-n\rangle=O(N_{1}^{100s_{0}})\}\sum_{\substack{n_i:\, n_2\neq n_1,n_3, \\ \langle n_{3}-n_{2}, n_{3}-n \rangle=O(N_{1}^{100s_{0}})}}\|r_{1}(n_{1})d_{3}(n_{3})\|_{L_{t}^{2}}^{2}\\
\lesssim &N_{2}^{-2s_{0}}N_{1}^{100s_{0}}\max(\frac{N_{3}}{N_{1}^{1/3}},1)\sum_{\substack{n_i:\, n_2\neq n_1,n_3, \\ \langle n_{3}-n_{2}, n_{3}-n \rangle=O(N_{1}^{100s_{0}})}}N_{1}^{-2+2s_{1}}\|d_{3}(n_{3})\|_{{\color{black}{L_{t}^{2}}}}^{2}\\
\lesssim &N_{1}^{Cs_{0}}N_{1}^{-2}\max(\frac{N_{3}}{N_{1}^{1/3}},1)\sup_{n_{3}}\sharp\{(n_{1},n_{2}):\,{\langle n_{2}-n_{1}, n_{2}-n_{3} \rangle=O(N_{1}^{100s_{0}})}, \,n_{2}\neq n_{1},n_{3}\}\sum_{n_{3}}\|d_{3}(n_{3})\|_{L_{t}^{2}}^{2}\\
\lesssim &N_{1}^{Cs_{0}}N_{1}^{-2}\max(\frac{N_{3}}{N_{1}^{1/3}},1)N_{1}N_{2}.
\end{aligned}
\end{equation}
In the above sequence of estimates, we used H\"older's inequality in the second line, and a variant of the counting Lemma \ref{lem: circle} in the third line (note that we assume $N_{1}\gg N_{2}$, thus one necessarily has $|n|\sim N_{1}$). In the last line we applied the counting Lemma \ref{lem: count strong}. Moreover, note that the $\psi(t)$ in $r_{1}(n_{1})$ gives enough decay in $t$, hence one has $\|d_{i}(n_{i})\|_{L_{t}^{\infty}}=\|d_{i}(n_{i})e^{-n_{i}^{2}t}\|_{L_{t}^{\infty}}\lesssim \|d_{i}(n_{i})e^{-n_{i}^{2}t}\|_{H_{t}^{b_{0}}}$. 

To summarize, when $N_{1}\gg N_{2}$, one has the second estimate 
\begin{equation}\label{eq: c2}
N_1^{s_0} \sum_{
\begin{subarray}
\quad n_{i}\sim N_{i}, n_{1}-n_{2}+n_{3}=n,  \, n_{2}\neq n_{1},n_{3}\\
n_{1}^{2}-n_{2}^{2}+n_{3}^{2}-n^{2}=O(M)=O(N_{1}^{100s_{0}})
\end{subarray}
}
\left|\int r_{1}(n_{1})\overline{d_{2}(n_{2})}d_{3}(n_{3})\overline{H(n)}\,dt\right|\lesssim N_{1}^{Cs_{0}}\frac{N_{2}^{1/2}}{N_{1}^{1/2}}\max\left(\frac{N_{3}^{1/2}}{N_{1}^{1/6}},1\right).
\end{equation}

One may also make use of the  Frobenius norm that  is more suitable when one deals with random data since it exploits better the independence of the random variables involved. The Frobenius  norm together a version of the Cauchy-Schwarz inequality recalled in \eqref{eq: m2} will give the  third estimate. We start from \eqref{eq: starcc} again. By the same argument in (\ref{eq: c1}), one has
\begin{equation}\label{eq: c4}
\begin{aligned}
&\sum_{n} 
\Big\|\sum_{
\begin{subarray}
\quad n_{i}\sim N_{i}, n_{1}-n_{2}+n_{3}=n, \,n_{2}\neq n_{1},n_{3}\\
n_{1}^{2}-n_{2}^{2}+n_{3}^{2}-n^{2}=O(N_{1}^{100s_{0}})
\end{subarray}
}r_{1}(n_{1})d_{2}(n_{2})d_{3}(n_{3})\Big\|_{L_{t}^{2}}
^{2}\\
\lesssim 
&\sum_{n,n_{3}}\Big\|
\sum_{n_{1}, n_{3}\neq n_{2}, \langle n_{2}-n_{1}, n_{2}-n_{3}\rangle=O(M) }
\frac{g_{n_{1}}}{|n_{1}|}e^{in_{1}^{2}t}\psi(t)d_{2}(n_{2})\Big\|_{L_{t}^{2}}^2,
\end{aligned}
\end{equation}
and by applying  the recalled  Cauchy-Schwarz inequality \eqref{eq: m2},  we can  further bound this expression  by
\begin{equation}
\begin{aligned}
\lesssim &N_{3}^{2}\sup_{n_{3}}\sum_{n}\Big\|\sum_{n_{2}}A(n,n_{2})d_{2}(n_{2})\Big\|^2_{L_{t}^{2}}\\
\lesssim 
&N_{3}^{2}\sup_{n_{3}}\left( \max_{n}\left(\sum_{n_{2}}\|A(n, n_{2})\|_{L_{t}^{\infty}}^{2}\right)+\left(\sum_{n\neq n'}\Big\|\sum_{n_{2}}A(n, n_{2})\overline{A({n',n_{2}})}\Big\|^{2}_{L_{t}^{\infty}}\right)^{1/2} \right),
\end{aligned}
\end{equation}
where we defined
\begin{equation}
A(n,n_{2})=A(n,n_{2},t)=
\begin{cases}
r_{1}(n+n_{2}-n_{3}), & \text{ if } \langle n-n_{3}, n_{2}-n_{3} \rangle=O(N_{1}^{100s_{0}}),\\
0,& \text{ otherwise}.
\end{cases}
\end{equation}
In the last line we also used $\sum_{n_{2}}\|d_{2}(n_{2})\|_{L_{t}^{2}}^{2}\lesssim 1$. For the sake of convenience, we also define
\begin{equation}
\sigma(n,n_{2})=
\begin{cases}
\frac{g_{n+n_{2}-n_{3}}}{N_{1}}, & \text{ if } \langle n-n_{3}, n_{2}-n_{3} \rangle=O(N_{1}^{100s_{0}}),\\
0, &\text{ otherwise}
\end{cases}
\end{equation}for later use.

 \begin{rem}\label{rem: matrixlinfty}
  By dropping an extra set of probability $e^{-N_{1}^{c}}$, one can in fact estimate 
$$\max_{n}\left(\sum_{n_{2}}\|A(n, n_{2})\|_{L_{t}^{\infty}}^{2}\right)+\left(\sum_{n\neq n'}\Big\|\sum_{n_{2}}A(n, n_{2})\overline{A({n',n_{2}})}\Big\|^{2}_{L_{t}^{\infty}}\right)^{1/2}$$ as $$\max_{n}\left(\sum_{n_{2}}|\sigma(n, n_{2})|^{2}\right)+\left(\sum_{n\neq n'}\Big|\sum_{n_{2}}\sigma(n, n_{2})\overline{\sigma({n',n_{2}})}\Big|^{2}\right)^{1/2}.$$ 

To see this, observe that $N_{1}\sim |n_{1}|=(n_{2}-n_{3}+n|$, and for any fixed $t\in [0,1]$, the estimate of 
\[
\max_{n}\left(\sum_{n_{2}}|A(n, n_{2})(t)|^{2}\right)+\left(\sum_{n\neq n'}\Big|\sum_{n_{2}}A(n, n_{2})(t)\overline{A({n',n_{2}})(t)}\Big|^{2}\right)^{1/2}.
\]is just the same as 
\[
\max_{n}\left(\sum_{n_{2}}|\sigma(n, n_{2})|^{2}\right)+\left(\sum_{n\neq n'}\Big|\sum_{n_{2}}\sigma(n, n_{2})\overline{\sigma({n',n_{2}})}\Big|^{2}\right)^{1/2}.
\]Now one can simply mimic the argument in the proof of  Lemma \ref{lem: linfinity} to go from a single $t$ to a collection of $\{t_{n}\} $ in $[0,1]$ so that $|t_{i}-t_{j}|\leq N_{1}^{-3}.$ and then go to $L_{t}^{\infty}[0,1]$. Then, finally, one can use the fact that there is a Schwartz function $\psi(t)$ multiplied inside each $r_{1}$ to go from $L_{t}^{\infty}[0,1]$ to $L_{t}^{\infty}(\mathbb{R})$. We omit the details.
\end{rem}

In the following and throughout the rest of the article, we will estimate instead the term 
\[
\max_{n}\left(\sum_{n_{2}}|\sigma(n, n_{2})|^{2}\right)+\left(\sum_{n\neq n'}\Big|\sum_{n_{2}}\sigma(n, n_{2})\overline{\sigma({n',n_{2}})}\Big|^{2}\right)^{1/2},\]and we don't repeat the similar reduction in the rest of the article.

We fix $n_{3}$. Note that for each $n$, there are at most $\sim N_{2}N_{1}^{100s_{0}}$ choices of $n_{2}$ so that 
$|n_{2}|\sim N_{2}$ and $\langle n_{3}-n, n_{3}-n_{2}\rangle=O(N_{1}^{100s_{0}})$. Hence,
\begin{equation}\label{eq: c6}
\max_{n}\left(\sum_{n_{2}}|\sigma(n, n_{2})|^{2}\right)\lesssim N_{2}N_{1}^{Cs_{0}}N_{1}^{-2}.
\end{equation}

For the non-diagonal term, we first observe that for all $n\neq n'$ fixed, up to an exceptional set of probability $e^{-N_{1}^{c_\epsilon}}$, one can apply Lemma \ref{lem: generalwiener} to derive
\begin{equation}\label{eq: c7}
\left|\sum_{n_{2}}\sigma(n, n_{2})\overline{\sigma(n',n_{2})}\right|\lesssim_{\epsilon} N_{1}^{\epsilon/2} \EEE\left(\Big|\sum_{n_{2}}\sigma(n, n_{2})\overline{\sigma(n',n_{2})}\Big|^{2}\right)^{1/2}.
\end{equation}This implies that
\begin{equation}
\begin{aligned}
&\quad \left|\sum_{n_{2}}\sigma(n, n_{2})\overline{\sigma(n',n_{2})}\right|^{2}\\ & \lesssim_{\epsilon} N_{1}^{\epsilon} \EEE\Big|\sum_{n_{2}}\sigma(n, n_{2})\overline{\sigma(n',n_{2})}\Big|^{2}\\ &
\sim N_{1}^{\epsilon-4}\sharp\{n_{2}:\, \langle n_{3}-n_{2}, n_{3}-n\rangle=O(N_{1}^{100s_{0}}), \,\langle n_{3}-n_{2}, n_{3}-n'\rangle=O(N_{1}^{100s_{0}}),\, n_{2}\neq n_1, n_{3}\}.
\end{aligned}
\end{equation}
Therefore, dropping an exceptional set of probability $N_{1}^{4}e^{-N_{1}^{c_{\epsilon}}}\sim e^{-N_{1}^{c}}$, we have
\begin{equation}\label{eq: c8}
\begin{aligned}
 & \sum_{n\neq n'}\Big|\sum_{n_{2}}\sigma(n, n_{2})\overline{\sigma({n',n_{2}})}\Big|^{2}\\
  \lesssim &N_{1}^{\epsilon-4}\sharp\{(n,n',n_{2}):\, n\neq n', \,\langle n_{3}-n_{2}, n_{3}-n\rangle=O(N_{1}^{100s_{0}}),  \,\langle n_{3}-n_{2}, n_{3}-n'\rangle=O(N_{1}^{100s_{0}}), \, n_{2}\neq n_1, n_{3}\}.
 \end{aligned}
\end{equation}
Counting first all the possible pairs of $(n,n_{2})$ by $N_{1}^{1+Cs_{0}}N_{2}$ (Lemma \ref{lem: count strong}), and by the Wick ordered condition $n_{3}\neq n_{2}$, which further gives at most $\sim N_{1}^{1+Cs_0}$ possible $n'$, we derive
\begin{equation}
\left(\sum_{n\neq n'}\Big|\sum_{n_{2}}\sigma(n, n_{2})\overline{\sigma({n',n_{2}})}\Big|^{2}\right)^{1/2}\lesssim N_{1}^{-1+C s_{0}}N_{2}^{1/2},
\end{equation}which obviously dominates the bound (\ref{eq: c6}) for the diagonal term.

To summarize, we can go back to \eqref{eq: c4} and derive our third estimate
\begin{equation}\label{eq: c9}
N_1^{s_0} \sum_{
\begin{subarray}
\quad n_{i}\sim N_{i}, n_{1}-n_{2}+n_{3}=n,  \, n_{2}\neq n_{1},n_{3}\\
n_{1}^{2}-n_{2}^{2}+n_{3}^{2}-n^{2}=O(M)=O(N_{1}^{100s_{0}})
\end{subarray}
}
\left|\int r_{1}(n_{1})\overline{d_{2}(n_{2})}d_{3}(n_{3})\overline{H(n)}\,dt\right|\lesssim N_{1}^{Cs_{0}}N_{1}^{-1/2}N_{2}^{1/4}N_{3}.
\end{equation}

To complete the argument, note that the case $N_{1}\sim N_{2}$ will follow from estimates \eqref{eq: c3} and \eqref{eq: c9}. Indeed, consider two subcases. In the case $N_{1}\geq N_{3}^{100}$, we use estimate \eqref{eq: c9}, and when $N_{1}< N_{3}^{100}$, we use estimate \eqref{eq: c3}. When $N_{1}\gg N_{2}$ (hence estimate \eqref{eq: c2} also holds), we also consider several subcases. In subcase $N_{2}N_{3}\geq N_{1}^{11/10}$, we use estimate \eqref{eq: c3}. In the case $N_{2}N_{3}< N_{1}^{11/10}$ and $N_{3}\geq N_{1}^{1/3}$, we use estimate \eqref{eq: c2}. In the case $N_{2}N_{3}< N_{1}^{11/10}$ and $N_{3}< N_{1}^{1/3}$, if $N_{2}\leq N_{1}^{9/10}$, we use estimate \eqref{eq: c2}. Finally, if $N_{2}N_{3}< N_{1}^{11/10}$, $N_{3}< N_{1}^{1/3}$ but $N_{2}> N_{1}^{9/10}$, there must hold $N_{3}\leq N_{1}^{1/5}$, hence one can use estimate \eqref{eq: c9}.

\section{Proof of Proposition \ref{prop: onemore}: Remaining cases}\label{sec: ProofpartIV}
We present the proof of the remaining cases. Note that in each case, the desired estimate will be reduced to the resonant part similarly as in the previous two sections, and we will only briefly sketch the reduction. It is unclear whether Case (a) and Case (c) are the hardest two cases, however, all the essential arguments required to treat the rest of the cases have already appeared in the previous two sections.

We will use the following notations throughout the section. Let $\|w_{i}\|_{X^{s_{0},b_{0}}}\lesssim 1$, $i=1,2,3$, $v_{i}=\phi(t)\sum_{|n_{i}|\sim N_{i}}\frac{g_{n_{i}}(\omega)}{|n_{i}|}e^{in_{i}\cdot x+in_{i}^{2}t}$, and $\|h\|_{X^{0,1-b_{0}}}\lesssim 1$.

Let $M$ be a parameter that will be specified in each of the cases, $r_{i}(n_{i},t)$ be the space Fourier transform of ${\color{black}{(P_{|\tau|<M}\phi(t))v_{i}}}$, and $d_{i}(n_{i},t)$ be the space Fourier transform of $P_{|\tau_i-n_i^2|<M}\phi(t)w_{i}$, $i=1,2,3$, and $H(n,t)$ be the space Fourier transform\footnote{As one sees in the previous two sections, the function $h$ here is actually $\phi(t/\delta)h$, whose $X^{0,1-b_{0}}$ norm is also bounded uniformly in $\delta$.} of $P_{|\tau-n^{2}|\leq M}\phi(t)h$. We will sometimes abbreviate $r_{i}(n_{i},t), d_{i}(n_{i},t), H(n,t)$ as $r_{i}(n_{i}), d_{i}(n_{i}), H(n)$ respectively. 

Similarly to \eqref{eq: basea} and \eqref{eq: basec},  one always has the following estimates:
\begin{equation}\label{eq: baseremaining}
\begin{cases}
N_{i}^{2s_{0}}\sum_{n_{i}}\|d_{i}(n_{i})e^{-n_{i}^{2}t}\|_{H_{t}^{b_{0}}}^{2}\lesssim \|P_{N_{i}}w_{i}\|_{X^{s_{0},b_{0}}}^{2}\lesssim 1,\\
r_{i}(n_{i},t)=\psi(t)\frac{g_{n_{i}}(\omega)}{|n_{i}|}e^{in_{i}\cdot x+in_{i}^{2}t}, \text{ where } \psi \text{ is a Schwartz function},\\
\sum_{n\sim N_{1}} \|H(n)e^{-in^{2}t}\|_{H_{t}^{1-b_{0}}}^{2}\lesssim \|P_{N_{1}}h\|_{X^{0,1-b_{0}}}^{2}.
\end{cases}
\end{equation}

\subsection{Case (b): $N_1(II)\geq N_3(I)\geq N_2(II)$}
This part is similar to Case (a). After handling the low-frequency part using deterministic estimates and localization in time, 
we aim to prove for all $N_{3}\geq N_{3,0}$ (where $N_{3,0}^{100}=\delta^{-1}$), one has up to an exceptional set  of probability $e^{-N_{3}^{c}}$ that
\begin{itemize}
\item when $N_{1}\sim N_{3}$,
\begin{equation}\label{eq: reduced1b}
\begin{aligned}
N_{1}^{s_{0}}\left|\int \NN_{1}(P_{N_{1}}w_{1}, P_{N_{2}}{w}_{2},P_{N_{3}}v_{3})\bar{h}\phi(t)\right|&\lesssim N_{3}^{-\epsilon_{1}}\|P_{N_{1}}w_{1}\|_{X^{s_{0},b_{0}}}\|P_{<N_{1}}h\|_{X^{0,1-b_{0}}}\\
&\sim  N_{1}^{-\epsilon_{1}}\|P_{N_{1}}w_{1}\|_{X^{s_{0},b_{0}}}\|P_{<N_{1}}h\|_{X^{0,1-b_{0}}};
\end{aligned}
\end{equation}\item  when $N_{1}\gg N_{3}$,
\begin{equation}\label{eq: reduced2b}
N_{1}^{s_{0}}\left|\int \NN_{1}(P_{N_{1}}w_{1}, P_{N_{2}}{w}_{2},P_{N_{3}}v_{3})\bar{h}\phi(t)\right|\lesssim N_{3}^{-\epsilon_{1}}\|P_{N_{1}}w_{1}\|_{X^{s_{0},b_{0}}}\|P_{N_{1}}h\|_{X^{0,1-b_{0}}}.
\end{equation}
\end{itemize}

Note that up to an exceptional set of probability $e^{-N_{3}^{c}}$ , we can assume that
\begin{equation}\label{eq: crudeproc}
|g_{n_{3}}(\omega)|+\|P_{N_{3}}\phi(t)v_{3}\|_{X^{0,b_{0}}}\leq N_{3}^{s_{1}}
\end{equation}
 for all $N_{3}\geq N_{3,0}$, $|n_{3}|\sim N_{3}$.

It also suffices to assume 
\begin{equation}\label{eq: bnum}
N_{3}\geq N_{2}^{1000}.
\end{equation}
Indeed, if $N_3<N_2^{1000}$, from the deterministic estimate \eqref{deter2} and the bound \eqref{eq: crudeproc}, one obtains
\begin{equation}\label{eq: deterministicb}
\begin{aligned}
N_{1}^{s_{0}}\left|\int \NN_{1}(P_{N_{1}}w_{1}, P_{N_{2}}{w}_{2},P_{N_{3}}v_{3})\bar{f}\phi(t)\right|&\lesssim N_{3}^{Cs_1}N_2^{-s_0}\|P_{N_1}w_{1}\|_{X^{s_{0},b_{0}}}\|f\|_{X^{0,1-b_0}},
\end{aligned}
\end{equation}
where $f=P_{N_{1}}h$ or $P_{<N_{1}}h$, hence \eqref{eq: reduced1b} and \eqref{eq: reduced2b} follow.

In the following, we will only prove \eqref{eq: reduced1b}, as estimate \eqref{eq: reduced2b} follows similarly (almost line by line). Note that in all the summations below we always have $|n_{i}|\sim N_{i}$, which we sometimes omit from the notation. Let $M=N_{3}^{100s_{1}}$, similarly as in Case (a), one can reduce \eqref{eq: reduced1b} to the following estimate:
\begin{equation}\label{eq: b1}
\begin{aligned}
&N_{1}^{s_{0}}\left|\sum_{\substack{
n_{1}-n_{2}+n_{3}=n, \,n_{2}\neq n_{1},n_{3}\\
n_{1}^{2}-n_{2}^{2}+n_{3}^{2}-n_{4}^{2}=O(M)}}\int d_{1}(n_{1})\overline{d_{2}(n_{2})}r_{3}(n_{3})\overline{H(n)}\,dt\right|\\
 \lesssim &N_{3}^{-\epsilon_{1}}\|P_{N_{1}}w_{1}\|_{X^{s_{0},b_{0}}}\|P_{<N_{1}}h\|_{X^{0,1-b_{0}}}\|P_{N_{2}}w_{2}\|_{X^{s_{0},b_{0}}}.
 \end{aligned}
\end{equation}

To see this, note that one automatically has $|n|\lesssim N_1$, and recall that $\sum_{n_2}\|d_{2}(n_{2})\|_{L^{\infty}}^{2}\lesssim 1$. By first applying Cauchy-Schwarz in $n$ and then in $n_{2}$, one has 
\begin{equation}\label{eq: b2}
\begin{aligned}
&N_{1}^{s_{0}}\left|\sum_{\substack{
n_{1}-n_{2}+n_{3}=n, \,n_{2}\neq n_{1},n_{3}\\
n_{1}^{2}-n_{2}^{2}+n_{3}^{2}-n_{4}^{2}=O(M)}}\int d_{1}(n_{1})\overline{d_{2}(n_{2})}r_{3}(n_{3})\overline{H(n)}\,dt\right|\\
\lesssim  &N_{1}^{s_{0}}\|P_{<N_{1}}h\|_{X^{0,1-b_{0}}}\left(\sum_{n}\Big\|
\sum_{
\substack{
n_{1}-n_{2}+n_{3}=n, \, n_{2}\neq n_{1},n_{3}\\
n_{1}^{2}-n_{2}^{2}+n_{3}^{2}-n_{4}^{2}=O(M)
}
}
d_{1}(n_{1})\overline{d_{2}(n_{2})}r_{3}(n_{3})\Big\|_{L_{t}^{2}}^{2}\right)^{1/2}\\
\lesssim &N_{1}^{s_{0}}\|P_{<N_{1}}h\|_{X^{0,1-b_{0}}}\left( \sum_{n,n_{2}}\Big\|\sum_{
\substack{
n_{1}-n_{2}+n_{3}=n, n_{2}\neq n_{1},n_{3},\\
n_{1}^{2}-n_{2}^{2}+n_{3}^{2}-n_{4}^{2}=O(M)
}}d_{1}(n_{1})r_{3}(n_{3})
\Big\|_{L_{t}^{2}}^{2} \right)^{1/2}.
\end{aligned}
\end{equation}

Recall also $n_{1}^{2}-n_{2}^{2}+n_{3}^{2}-n_{4}^{2}=O(M)$ together with $n_{1}-n_{2}+n_{3}=n$ imply that
\begin{equation}
\langle n_{3}-n_{2}, n_{3}-n\rangle=O(M).
\end{equation}By applying Cauchy-Schwarz in $n_3$ (note that the inner sum can be viewed as over $n_3$ only) and recalling $\|r_{3}(n_3)\|_{L_t^\infty}\lesssim N_3^{-1+s_1}$ (outside an exceptional set), the above can be further bounded as
\[
\begin{split}
\lesssim & N_{1}^{s_{0}}\|P_{<N_{1}}h\|_{X^{0,1-b_{0}}}\left( \sum_{n,n_2, n_3}\|d_1(n_1)\|_{L_t^2}^2\right)^{1/2}N_3^{-1+s_1}\left(\sup_{n,n_2}\#\{n_3:\, \langle n_{3}-n_{2}, n_{3}-n\rangle=O(M)\}\right)^{1/2}\\
\lesssim & \|P_{<N_{1}}h\|_{X^{0,1-b_{0}}}\|P_{N_1}w_1\|_{X^{s_0,b_0}} N_3^{Cs_1}N_3^{-2/3}\left(\sup_{n_1}\#\{n_2,n_3:\, \langle n_2-n_1,n_2-n_3\rangle=O(M),\,n_2\neq n_1,n_3\} \right)^{1/2},
\end{split}
\]where in the second step above, we have applied a variant of Lemma \ref{lem: circle} to conclude 
\[
\sup_{n,n_2}\#\{n_3:\, \langle n_{3}-n_{2}, n_{3}-n\rangle=O(M)\}\lesssim N_3^{2/3+100s_1}.
\]Indeed, since $N_1\sim N_3$, after dividing into $\lesssim N_3^{100s_1}$ parts, all the $n_3$ lie in an annulus of radius $\sim R\lesssim N_3$ with thickness $\sim O(\frac{1}{R})$. By Lemma \ref{lem: circlepart}, there are at most $\sim R^{2/3}\lesssim N_3^{2/3}$ such points.

Furthermore, we apply Lemma \ref{lem: linecounting} to count
\[
\sup_{n_1}\#\{n_2,n_3:\, \langle n_2-n_1,n_2-n_3\rangle=O(M),\,n_2\neq n_1,n_3\}\lesssim N_2^2N_3^{1+100s_1},
\]which implies that
\[
\begin{split}
&N_{1}^{s_{0}}\left|\sum_{\substack{
n_{1}-n_{2}+n_{3}=n, \,n_{2}\neq n_{1},n_{3}\\
n_{1}^{2}-n_{2}^{2}+n_{3}^{2}-n_{4}^{2}=O(M)}}\int d_{1}(n_{1})\overline{d_{2}(n_{2})}r_{3}(n_{3})\overline{H(n)}\,dt\right|\\
\lesssim &\|P_{<N_{1}}h\|_{X^{0,1-b_{0}}}\|P_{N_1}w_1\|_{X^{s_0,b_0}} N_3^{Cs_1}N_2N_3^{-1/6}.
\end{split}
\]Recall that we have reduced to the case $N_3\geq N_2^{1000}$, hence the desired estimate is obtained.

\subsection{Case (d): $N_1(I)\geq N_3(II)\geq N_2(II)$}

This case is almost identical to Case (c). By the deterministic estimates, it suffices to show for all $N_1\geq N_{1,0}$ (where $N_{1,0}=\delta^{-1}$), one has up to an exceptional set of probability $e^{-N_1^c}$ that
\begin{equation}\label{eqn: case d}
N_1^{s_0}\left| \sum_{
\begin{subarray}
\quad n_{i}\sim N_{i}, n_{1}-n_{2}+n_{3}=n,  \, n_{2}\neq n_{1},n_{3}\\
n_{1}^{2}-n_{2}^{2}+n_{3}^{2}-n^{2}=O(M)=O(N_{1}^{100s_{0}})
\end{subarray}
}
\int r_{1}(n_{1})\overline{d_{2}(n_{2})}d_{3}(n_{3})\overline{H(n)}\,dt\right|\lesssim N_{1}^{-\epsilon_{1}}, \text{ for some } \epsilon_{1}\gg \epsilon_{0}.
\end{equation}

Note that by removing an exceptional set of probability $e^{-N_{1,0}^{-cs_1}}$ if necessary, we will assume in this subsection that
\begin{equation}
|g_{n_1}(\omega)|+\|P_{N_1}\phi(t)v_1\|_{X^{0,b_0}}+\|P_{N_1}\phi(t)v_1\|_{L_{t,x}^\infty}\lesssim N_1^{s_1},\quad \forall |n_1|\sim N_1>N_{1,0}.
\end{equation}

There still holds the same bound (\ref{eq: c3}) as in Case (c). Moreover, one still has (\ref{eq: c1}), as it has nothing to do with the relative sizes of $N_2, N_3$, and when $N_1\gg N_3$, the bound (\ref{eq: c2}) still holds true as well (with the choice $M=N_1^{100s_0}$). Indeed, the only step that one needs to check here is that
\[
\sup_{n,n_2}\#\{n_3:\, \langle n_3-n_2,n_3-n\rangle=O(M)\}\lesssim N_1^{100s_0}\max\left(\frac{N_3}{N_1^{1/3}},1\right),
\]which follows from the same proof of Lemma \ref{lem: circle} and the assumption that $N_1\gg N_3$.

We claim that the desired bound follows from (\ref{eq: c3}) and (\ref{eq: c2}). To see this, when $N_1\sim N_3$, if one further has $N_2>N_3^{1/9}\sim N_1^{1/9}$, one can apply (\ref{eq: c3}). Otherwise, $N_2\leq N_3^{1/9}$, (\ref{eq: c2}) suffices. When $N_1\gg N_3$, we address two difference subcases. If we are in the subcase that $N_1^{1/3}\geq N_3$, then one automatically has $N_2\lesssim N_1^{1/3}$ hence (\ref{eq: c2}) implies the desired estimate. In the subcase that $N_1^{1/3}<N_3$, suppose in addition one has $N_2N_3>N_1^{\frac{10}{9}}$, then we apply (\ref{eq: c3}), otherwise the desired decay in $N_1$ follows from (\ref{eq: c2}). The proof of Case (d) is complete.

\subsection{Case (e): $N_1(II)\geq N_2(I)\geq N_3(I)$}
By a similar reduction process as in Case (a), let $N_{2,0}$ be a large parameter satisfying $N_{2,0}^{100}=\delta^{-1}$, we will focusing on proving for all $N_{2}\geq N_{2,0}$ that, up to an exceptional set of probability of $e^{-N_{2}^{c}}$ and a common exceptional set independent of $N_{2}$, with probability $e^{-N_{2,0}^{c}}$, we have 
\begin{itemize}
\item when $N_{1}\sim N_{2}$, 
\begin{equation}\label{eq: reducede1}
N_{1}^{s_{0}}\left|\int \phi(t)\bar{h}\NN
_{1}(P_{N_{1}}w_{1}, P_{N_{2}}{v}_{2}, P_{N_{3}}v_{3})\right|\lesssim N_{2}^{-\epsilon_{1}}\sim N_{1}^{-\epsilon_{1}},
\end{equation}
\item when $N_{1}\gg N_{2}$,
\begin{equation}\label{eq: reducede}
N_{1}^{s_{0}}\left|\int \phi(t)\bar{h}\NN_{1}(P_{N_{1}}w_{1},P_{N_{2}}{v}_{2}, P_{N_{3}}v_{3})\right|\lesssim N_{2}^{-\epsilon_{1}}\|P_{N_{1}}w_{1}\|_{X^{s_{0},b_{0}}}\|P_{N_{1}}h\|_{X^{0,1-b_{0}}}.
\end{equation}
\end{itemize}
As usual, the part $N_{2}\leq N_{2,0}$ will be handled by the purely deterministic estimate \eqref{deter2}, and by localizing in time $\sim N_{2,0}^{-100}$.

One may assume, by dropping a set of probability $e^{-N_{2,0}^{c}}$, that
\begin{equation}\label{eq: crudecasec}
\begin{cases}
|g_{n}|\leq N_{2,0}^{s_{1}}, & |n|\leq N_{2,0},\\
|g_{n}|\leq N_{2}^{s_{1}}, & |n|\sim N_{2}\geq N_{2,0}.
\end{cases}
\end{equation}

\begin{rem}
In the original paper of Bourgain \cite{bourgain1996invariant}, Case (e) is not the hardest case, however, one should be particularly careful in our irrational setting. This is because our counting lemma in the irrational case is weaker compared to the ones in \cite{bourgain1996invariant}, hence any loss of $N_{1}^{\epsilon}$ will be unfavorable. Since the random data argument can gain at most a (negative) power of $N_{2}$, our counting Lemma \ref{lem: count strong} becomes useless in Case (e). 
\end{rem}

\begin{rem}
One should also be very careful about dropping exceptional sets of small probability  when the highest frequency is of type (II).  For example, in our current Case (e),  all large deviation type arguments require one to drop a set of probability $e^{-N_{2}^{c}}$, thus  one cannot apply random data type argument for too many times.  For instance, if one drops $N_{1}^{2}$ different sets with  probability $e^{-N_{2}^{c}}$, one immediately loses control of the total probability. Moreover, in Case (e), one also needs to sum in $N_{1}$. Therefore, it is crucial that, for a fixed $N_{2}$ and for all $N_{1}$, one can apply at most $N_{2}^{C}$ times essentially different random data type arguments. This is an issue existing even in the rational tori case. We will add some more details along the proof for the convenience of the reader.
\end{rem}

From the remark above, one observes that the potentially most troublesome situation will be when $\ln N_{1}\gg \ln N_{2}$. Hence, in the following we will only focus on proving \eqref{eq: reducede}, and only briefly comment on necessary changes needed for proving \eqref{eq: reducede1}. 

Let $M=N_{2}^{100s_{1}}$, we may further reduce \eqref{eq: reducede} to the following estimate:
\begin{equation}
\begin{aligned}
& N_{1}^{s_{0}}\left|\sum_{
\begin{subarray}
\quad |n_{i}|\sim N_{i}, \,n_{1}-n_{2}+n_{3}-n=0\\
n_{1}^{2}-n_{2}^{2}+n_{3}^{2}-n^{2}=O(M)
\end{subarray} }\int d_{1}(n_{1})\overline{r_{2}(n_{2})}r_{3}(n_{3})\overline{H(n)}\,dt\right|\lesssim N_{2}^{-\epsilon_{1}}\|P_{N_{1}}w_{1}\|_{X^{s_{0},b_{0}}}\|P_{N_{1}}h\|_{X^{0,1-b_{0}}}.\\
\end{aligned}
\end{equation}

By applying Cauchy-Schwarz in $n$, we have
\begin{equation}
\begin{aligned}
&N_{1}^{s_{0}}\left|\sum_{
\begin{subarray}
\quad |n_{i}|\sim N_{i}, \,n_{1}-n_{2}+n_{3}-n=0\\
n_{1}^{2}-n_{2}^{2}+n_{3}^{2}-n^{2}=O(M)
\end{subarray} }\int d_{1}(n_{1})\overline{r_{2}(n_{2})}r_{3}(n_{3})\overline{H(n)}\,dt\right|\\
\lesssim 
  &N_{1}^{s_{0}}\|P_{N_{1}}h\|_{X^{0,1-b_{0}}}\left(\sum_{n}\Big\|
\sum_{
\begin{subarray}
\quad n_{1}-n_{2}+n_{3}=n, \,n_{2}\neq n_{1},n_{3}\\
n_{1}^{2}-n_{2}^{2}+n_{3}^{2}-n_{4}^{2}=O(M)
\end{subarray}
}
d_{1}(n_{1})\overline{r_{2}(n_{2})}r_{3}(n_{3})
\Big\|_{L_{t}^{2}}^{2}\right)^{1/2}.
\end{aligned}
\end{equation}

For the sake of brevity, in the following we oftentimes omit the condition $|n_i|\sim N_i$ in the summation. Dividing $\{n_1:\, |n_{1}|\sim N_{1}\}$ into finitely overlapping balls $\{J\}$ of radius $\sim N_{2}$, we are left with showing for each $J$ that, up to some exceptional set of small probability $e^{-N_{2}^{c}}$, 

\begin{equation}\label{eq: e1}
\sum_{n\in J}\Big\|
\sum_{
\begin{subarray}
\quad n_{1}-n_{2}+n_{3}=n, \,n_{2}\neq n_{1},n_{3}\\
n_{1}^{2}-n_{2}^{2}+n_{3}^{2}-n_{4}^{2}=O(M)
\end{subarray}
}
d_{1}(n_{1})\overline{r_{2}(n_{2})}r_{3}(n_{3})
\Big\|_{L_{t}^{2}}^{2}
\lesssim N_{2}^{-s}\|P_Jw_1\|_{X^{0,b_0}}^2,
\end{equation}where we have observed that $n_1\in J$ implies $n\in \tilde{J}$ (a doubling of $J$) and we still denote $\tilde{J}$ as $J$ for the sake of notational convenience. Moreover, we will prove the above estimate for some $s\gg s_1$. In particular, any loss of $N_{2}^{Cs_{1}}$ in the estimate will be irrelevant. %Thus we may assume in the summation that $n_1^{2}-n_{2}^{2}+n_{3}^{2}-n^{2}=\mu+O(1)$, where $\mu\in \ZZZ$ and $|\mu|\lesssim N_{2}^{100s_{1}}$.

Note that for each fixed $N_{1}$, there are $\sim N_{1}^{2}/N_{2}^{2}$ such $J$, hence one should be careful when applying random data type argument to avoid dropping too many exceptional sets. Observe, every time one applies large deviation type argument to estimate sums of Gaussians and multiple Gausssians, one needs to drop an exceptional set of probability $e^{-N_{2}^c}$, and such set, a priori may depend on $J$. If one naively drops all such sets, a priori one may need to drop in total a set of probability $\sim \frac{N_{1}^{2}}{N_{2}^{2}}e^{-N
 _{2}^{c}}$, which could be enormous when $N_1\gg N_2$. Also recall we also need to sum for all $N_{1}\geq N_{2}$.  This problem will even arise when one studies the problem on rational tori. We will  explain how to address this issue in detail  in  Subsubsection \ref{eeeee}, and other cases will follow similarly.
 
 Note that for the case $N_1\sim N_2$, the decomposition into $\{J\}$ is unnecessary.
 
To prove (\ref{eq: e1}),  we first define
\[
A(n,n_1)=A(n,n_1)(t)=\begin{cases} 
\sum_{\begin{subarray}
\quad n_{1}-n_{2}+n_{3}=n, \,n_{2}\neq n_{1},n_{3}\\
n_{1}^{2}-n_{2}^{2}+n_{3}^{2}-n_{4}^{2}=O(M)
\end{subarray}} r_{2}(n_2)r_3(n_3), & \text{if } n_1,n\in J,\\
0, & \text{otherwise},
\end{cases}
\]and
\[
\sigma(n,n_1)=\begin{cases} 
\sum_{\begin{subarray}
\quad n_{1}-n_{2}+n_{3}=n, \,n_{2}\neq n_{1},n_{3}\\
n_{1}^{2}-n_{2}^{2}+n_{3}^{2}-n_{4}^{2}=O(M)
\end{subarray}} \frac{g_{n_2}(\omega)g_{n_3}(\omega)}{N_2N_3}, & \text{if } n_1,n\in J,\\
0, & \text{otherwise}.
\end{cases}
\]Then, similarly as in Remark \ref{rem: matrixlinfty}, one has the left hand side of (\ref{eq: e1}) bounded by
\begin{equation}
\lesssim \|P_Jw_1\|_{X^{0,b_0}}^2\left[\max_{n\in J}\sum_{n_1\in J} |\sigma(n,n_1)|^2 + \left(\sum_{n\neq n'}\Big| \sum_{n_1\in J} \sigma(n,n_1)\overline{\sigma(n',n_1)}\Big|^2\right)^{1/2}\right],
\end{equation}where we have applied  (\ref{eq: m2}) and recalled that $\sum_{n_1\in J}\|d_1(n_1)\|_{L_t^2}^2\lesssim \|P_{J}w_1\|_{X^{0,b_0}}^2$. In the following, it suffices to bound the two terms in the brackets by $N_2^{-s}$.

The diagonal term is easier. Note that if $n_2\neq n_3$, for $n\in J$ fixed, one has
\begin{equation}\label{eq: e diag}
\begin{split}
&\sum_{n_1\in J} |\sigma(n,n_1)|^2\\
\lesssim &(N_2N_3)^{-2}N_2^{Cs_1}\sup_n\#\{(n_2,n_3):\, n=n_1-n_2+n_3,\, n_2\neq n_1, n_3,\, \langle n_3-n_2,n-n_3\rangle =O(M)\}\\
\lesssim &(N_2N_3)^{-2}N_2^{Cs_1}N_3^2N_2=N_2^{-1+Cs_1}.
\end{split}
\end{equation}In the first step above, we applied Lemma \ref{lem: generalwiener} to get $|\sum g_{n_2}g_{n_3}|^2\lesssim N_2^{Cs_1}\sum 1$ , by dropping an exceptional set if necessary, and in the second step, we counted $n_3$ naively and then $n_2$ using Lemma \ref{lem: linecounting}.
 
We are thus left with the non-diagonal term. Expanding $\sigma(n,n_1)$ and $\sigma(n',n_1)$, our goal is to show that
\begin{equation}\label{eq: e2}
\left(\sum_{n\neq n'}\Big| \sum_{n_1\in J} \sigma(n,n_1)\overline{\sigma(n',n_1)}\Big|^2\right)^{1/2}=(N_2N_3)^{-2}\left(\sum_{n\neq n'}\Big| \sum_{(\ast)} \overline{g_{n_2}}g_{n_3}g_{n'_2}\overline{g_{n'_3}}  \Big|^2\right)^{1/2}\lesssim N_2^{-s}
\end{equation}for some number $s\gg s_1$. Here we have simplified the notation by using $(\ast)$ to denote the set of $(n_1,n_2,n_3,n'_2,n'_3)$ satisfying
\begin{equation}
\begin{cases}
n_1\in J,\\
n=n_1-n_2+n_3,\,n_2\neq n_1,n_3,\, \langle n-n_1,n-n_3\rangle=O(M),\\
n'=n_1-n'_2+n'_3,\,n'_2\neq n_1,n'_3,\, \langle n'-n_1,n'-n'_3\rangle=O(M).
\end{cases}
\end{equation}In the following, we will prove (\ref{eq: e2}) case by case.

\subsubsection{Case I: $n_2,n_3,n'_2,n'_3$ are distinct}\label{eeeee}

Denoting the corresponding summation in $\sum_{(\ast)}$ by $\sum_{(\ast),1}$ and applying again Lemma \ref{lem: generalwiener} up to dropping a set of measure $e^{-N_{2}^{c}}$, one has
\[
\Big| \sum_{(\ast),1} \overline{g_{n_2}}g_{n_3}g_{n'_2}\overline{g_{n'_3}}  \Big|^2\lesssim N_2^{Cs_1}\sum_{(\ast),1} 1.
\]Hence, denoting the corresponding contribution of Case I in the left hand side of (\ref{eq: e2}) by $(\ref{eq: e2})_1$, one obtains

\begin{equation}\label{eq: ef1}
(\ref{eq: e2})_1\lesssim (N_2N_3)^{-2}N_2^{Cs_1}\left(\#\{n_1,n_2,n_3,n'_2,n'_3:\,(\ast\ast)\} \right)^{1/2},
\end{equation}
where $(\ast\ast)$ denotes the conditions
\begin{equation}\label{eq: star star}
\begin{cases}
n_1\in J,\\
n_2\neq n_1,n_3,\, n_2'\neq n_1,n_3',\\
\langle n_2-n_1, n_2-n_3\rangle=O(M),\, \langle n_2'-n_1, n_2'-n_3'\rangle = O(M).
\end{cases}
\end{equation}

By first counting naively $n_1\in J$, one has
\[
\#\{n_1,n_2,n_3,n'_2,n'_3:\,(\ast\ast)\}\lesssim N_2^2\left(N_3^2\max\left(\frac{N_2}{N_1^{1/3}},1\right)\right)^2,
\]where we have then counted $n_3$ naively, and applied Lemma \ref{lem: circle} (recalling that $N_1\gg N_2$). Therefore,
\[
(\ref{eq: e2})_1\lesssim N_2^{Cs_1}\max(N_1^{-1/3},N_2^{-1})\leq N_2^{-1/3+Cs_1}.
\]

Note that in the case $N_1\sim N_2$, the same estimate remains true, as the counting Lemma \ref{lem: circle} still implies the same bound. Similarly, in Case II, III, IV, V below, Lemma \ref{lem: circle} always provides the same counting result.

Before we go to the next case, we explain the issue about not dropping too many exceptional sets. This needs to be taken care of since the relation
\begin{equation}\label{errrr}
\begin{cases}
n_1\in J,\\
n=n_1-n_2+n_3,\,n_2\neq n_1,n_3,\, \langle n-n_1,n-n_3\rangle=O(M),\\
n'=n_1-n'_2+n'_3,\,n'_2\neq n_1,n'_3,\, \langle n'-n_1,n'-n'_3\rangle=O(M).
\end{cases}
\end{equation}
a priori depends on $J$. Note that we need only to worry about the case $N_{1}\gg N_{2}, N_{3}$. Without loss of generality, we may assume also $n\in J$. We write $J=a_{J}+B_{N_{2}}$,   $n=a_{J}+m, n'=a_{J}+m', n_{1}=a_{J}+m_{1}$, and  relation \eqref{errrr} as
\begin{equation}\label{err1111}
\begin{cases}
m_{1}\in B_{N_{2}},\\
m=m_1-n_2+n_3,\,n_2\neq n_1,n_3,\, \langle m-m_{1},m-n_3\rangle=-\langle m-m_{1}, a_{J} \rangle+O(M),\\
m'=m_1-n'_2+n'_3,\,n'_2\neq n_1,n'_3,\, \langle m'-m_1,m'-n'_3\rangle=-\langle m'-m_{1}, a_{J} \rangle+O(M).
\end{cases}
\end{equation}
The point is, though, there are potentially many choice of $a_{J}$, the above relation is empty unless 

$$-\langle m-m_{1}, a_{J} \rangle+O(M)=O(N_{2}^{2}),\quad 
-\langle m'-m_{1}, a_{J} \rangle+O(M)=O(N_{2}^{2}).$$ Thus, we can always write  relation \eqref{err1111} into $O(M)$ many union of the following,
\begin{equation}
\begin{cases}
m_{1}\in B_{N_{2}},\\
m=m_1-n_2+n_3,\,n_2\neq n_1,n_3,\, \langle m-m_{1},m-n_3\rangle=a+O(1),\\
m'=m_1-n'_2+n'_3,\,n'_2\neq n_1,n'_3,\, \langle m'-m_1,m'-n'_3\rangle=b+O(1),
\end{cases}
\end{equation}
where $a,b\in \ZZZ$ and $|a|, |b|\lesssim N_{2}^{2}$. Thus, the total exceptional set one needs to drop, for all $N_{1}$ and $J$, will be at most $N_{2}^{2}e^{-N_{2}^{c}}$, which is allowed.

We don't repeat this discussion of the exceptional set in the later part of the article.
\subsubsection{Case II: $n_2=n'_2$ ($n_3\neq n'_3$)}

Denote the corresponding summation in $\sum_{(\ast)}$ as $\sum_{(\ast), 2}$, one has from  $n_2=n_2'$, (\ref{eq: crudecasec}) and Lemma \ref{lem: generalwiener} that, {\color{black}{up to  dropping an exceptional set}}
\begin{equation}\label{eq: ef2}
\Big| \sum_{(\ast),2 }\overline{g_{n_2}}g_{n_3}g_{n'_2}\overline{g_{n'_3}}  \Big|^2\lesssim N_2^{Cs_1}\sum_{n_3,n_3'} (\#S(n,n',n_3,n_3'))^2,
\end{equation}where by Lemma \ref{lem: linecounting}
\[
\#S(n,n',n_3,n_3'):=\#\{n_1,n_2:\, (n_1,n_2,n_3,n_2,n_3') \text{ satisfies } (\ast)\}\lesssim N_2^{1+Cs_1}.
\]Hence, remember the definition of $(\ast\ast)$ in \eqref{eq: star star}, one further has
\begin{equation}
(\ref{eq: e2})_2\lesssim (N_2N_3)^{-2}N_2^{Cs_1}N_2^{1/2}\left(\#\{n_1,n_2,n_3, n_3':\, (n_1,n_2,n_3,n_2,n_3') \text{ satisfies } (\ast\ast)\} \right)^{1/2}.
\end{equation}By counting $n_1\in J$ naively first, then counting $(n_2,n_3)$ using Lemma \ref{lem: circle}, lastly counting $n_3'$ via Lemma \ref{lem: linecounting}, one obtains
\[
\#\{n_1,n_2,n_3, n_3':\, (n_1,n_2,n_3,n_2,n_3') \text{ satisfies } (\ast\ast)\}\lesssim N_2^2 N_3^2 \max\left(\frac{N_2}{N_1^{1/3}},1\right) N_3 N_2^{Cs_1},
\]which implies that
\begin{equation}
(\ref{eq: e2})_2\lesssim N_2^{Cs_1}N_3^{-1/2}\max(N_1^{-1/6}, N_2^{-1/2})\leq N_2^{-1/6+Cs_1}.
\end{equation}

\subsubsection{Case III: $n_3=n'_3$ ($n_2\neq n'_2$)}

Denoting the corresponding summation in $\sum_{(\ast)}$ as $\sum_{(\ast), 3}$, similarly as in Case II, one has
\begin{equation}\label{eq: ef3}
\Big| \sum_{(\ast),3 }\overline{g_{n_2}}g_{n_3}g_{n'_2}\overline{g_{n'_3}}  \Big|^2\lesssim N_2^{Cs_1}\sum_{n_2,n_2'} (\#S(n,n',n_2,n_2'))^2,
\end{equation}where by trivially counting $n_3$,
\[
\#S(n,n',n_2,n_2'):=\#\{n_1,n_3:\, (n_1,n_2,n_3,n_2',n_3) \text{ satisfies } (\ast)\}\lesssim N_3^2.
\]
Hence, remember the definition of $(\ast\ast)$ in \eqref{eq: star star}
\begin{equation}
(\ref{eq: e2})_3\lesssim (N_2N_3)^{-2}N_2^{Cs_1}N_3\left(\#\{n_1,n_2,n_2', n_3:\, (n_1,n_2,n_3,n_2',n_3) \text{ satisfies } (\ast\ast)\} \right)^{1/2}.
\end{equation}By trivially counting $n_1\in J$, $n_3$, and applying Lemma \ref{lem: circle} to $n_2$ and $n_2'$, one obtains
\[
\#\{n_1,n_2,n_2', n_3:\, (n_1,n_2,n_3,n_2',n_3) \text{ satisfies } (\ast\ast)\}\lesssim N_2^{Cs_1}N_2^2N_3^2\max\left(\frac{N_2^2}{N_1^{2/3}},1\right).
\]Therefore,
\begin{equation}
(\ref{eq: e2})_3\lesssim N_2^{Cs_1}\max\left(N_1^{-1/3}, N_2^{-1}\right)\leq N_2^{-1/3+Cs_1}.
\end{equation}

\subsubsection{Case IV: $n_2=n'_3$, $n_3\neq n'_2$}

Denoting the corresponding summation in $\sum_{(\ast)}$ as $\sum_{(\ast), 4}$, similarly as above, one has
\begin{equation}\label{eq: e 8}
\Big| \sum_{(\ast),4 }\overline{g_{n_2}}g_{n_3}g_{n'_2}\overline{g_{n'_3}}  \Big|^2\lesssim N_2^{Cs_1}\sum_{n_3,n_2'} (\#S(n,n',n_3,n_2'))^2,
\end{equation}where by Lemma \ref{lem: linecounting},
\[
\#S(n,n',n_3,n_2'):=\#\{n_1,n_2:\, (n_1,n_2,n_3,n_2',n_2) \text{ satisfies } (\ast)\}\lesssim N_2^{1+Cs_1}.
\]Plugging into the above, one obtains, again remember the definition of $(\ast\ast)$ in \eqref{eq: star star}
\begin{equation}
(\ref{eq: e2})_4\lesssim (N_2N_3)^{-2}N_2^{Cs_1}N_2^{1/2}\left(\#\{n_1,n_2, n_3, n_2':\, (n_1,n_2,n_3,n_2',n_2) \text{ satisfies } (\ast\ast)\} \right)^{1/2}.
\end{equation}Same as in Case III, by trivially counting $n_1\in J$, $n_3$, and applying Lemma \ref{lem: circle} to $n_2$ and $n_2'$, one obtains
\[
\#\{n_1,n_2, n_3, n_2':\, (n_1,n_2,n_3,n_2',n_2) \text{ satisfies } (\ast\ast)\}\lesssim N_2^{Cs_1}N_2^2N_3^2\max\left(\frac{N_2^2}{N_1^{2/3}},1\right).
\]Therefore, observing that in this case one must have $N_2\sim N_3$,
\begin{equation}
(\ref{eq: e2})_4\lesssim  N_2^{Cs_1}N_3^{-1}N_2^{-1/2}\max\left(\frac{N_2}{N_1^{1/3}},1\right)\lesssim N_2^{-1/2+Cs_1}\max(N_1^{-1/3}, N_2^{-1})\leq N_2^{{-5/6+Cs_1}}.
\end{equation}

\subsubsection{Case V: $n_3=n'_2$, $n_2\neq n'_3$}

This case can be treated in the exact same way as Case IV.

\subsubsection{Case VI: $n_3=n'_2$, $n_2= n'_3$}

In this case we again have $N_2\sim N_3$. Denoting as $\sum_{(\ast),6}$ the corresponding sum in $\sum_{(\ast)}$, one has
\begin{equation}
\Big| \sum_{(\ast),6 }\overline{g_{n_2}}g_{n_3}g_{n'_2}\overline{g_{n'_3}}  \Big|^2\lesssim N_2^{Cs_1}\left(\#S(n,n')\right)^2,
\end{equation}where
\[
S(n,n'):=\{n_1,n_2, n_3:\, (n_1,n_2,n_3,n_3,n_2) \text{ satisfies } (\ast)\},
\]and in this case means that $n_1,n_2,n_3$ are distinct, $n=n_1-n_2+n_3$, $n'=n_1-n_3+n_2$, and
\[
\langle n-n_1, n-n_3\rangle=O(M),\quad \langle n'-n_1, n'-n_2\rangle=O(M).
\]The above implies that $n+n'=2n_1$, hence $\#S(n,n')\lesssim N_2^{1+Cs_1}$ by Lemma \ref{lem: linecounting}.
 
As a result, again remember the definition of $(\ast\ast)$ in \eqref{eq: star star}
\begin{equation}
(\ref{eq: e2})_6\lesssim (N_2N_3)^{-2}N_2^{Cs_1}N_2^{1/2}\left(\#\{n_1,n_2, n_3:\, (n_1,n_2,n_3,n_3,n_2) \text{ satisfies } (\ast\ast)\} \right)^{1/2}.
\end{equation}From $(\ast \ast)$, one has
\[
\langle n_2-n_1,n_2-n_3\rangle=O(M),\quad \langle n_3-n_1, n_3-n_2\rangle=O(M),
\]hence $|n_2-n_3|^2=O(M)$. Trivially counting $n_2,n_3$, and applying Lemma \ref{lem: linecounting} to count $n_1\in J$, one has 
\[
\#\{n_1,n_2, n_3:\, (n_1,n_2,n_3,n_3,n_2) \text{ satisfies } (\ast\ast)\}\lesssim N_2^{Cs_1}N_2^2N_3^2N_2,
\]thus
\begin{equation}
(\ref{eq: e2})_6\lesssim  N_2^{Cs_1}N_3^{-1}\lesssim N_2^{-1+Cs_1}.
\end{equation}This concludes the proof of Case (e).

\subsection{Case (f): $N_1(II)\geq N_3(I)\geq N_2(I)$}
 
By the same reduction as in Case (e), let $N_{3,0}^{100}=\delta^{-1}$ and $M=N_{3}^{100s_1}$. It suffices to consider the high frequency part $N_3\geq N_{3,0}$. Our goal is to show that, up to an exceptional set of probability $\sim e^{-N_{3}^c}$ and a common exceptional set (independent of $N_3$) of probability $e^{-N_{3,0}^c}$, there hold
\begin{itemize}
\item when $N_{1}\sim N_{3}$, 
\begin{equation}\label{eq: reducedf1}
N_{1}^{s_{0}}\left|\int \phi(t)\bar{h}\NN
_{1}(P_{N_{1}}w_{1}, P_{N_{2}}{v}_{2}, P_{N_{3}}v_{3})\right|\lesssim N_{3}^{-\epsilon_{1}}\sim N_{1}^{-\epsilon_{1}},
\end{equation}
\item when $N_{1}\gg N_{3}$,
\begin{equation}\label{eq: reducedf}
N_{1}^{s_{0}}\left|\int \phi(t)\bar{h}\NN_{1}(P_{N_{1}}w_{1},P_{N_{2}}{v}_{2}, P_{N_{3}}v_{3})\right|\lesssim N_{3}^{-\epsilon_{1}}\|P_{N_{1}}w_{1}\|_{X^{s_{0},b_{0}}}\|P_{N_{1}}h\|_{X^{0,1-b_{0}}}.
\end{equation}
\end{itemize}We will also assume, by dropping a set of probability $e^{-N_{3,0}^c}$ that
\begin{equation}\label{eq: f0}
\begin{cases}
|g_n|\leq N_{3,0}^{s_1}, & |n|\leq N_{3,0},\\
|g_n|\leq N_3^{s_1}, &|n|\sim N_3\geq N_{3,0}.
\end{cases}
\end{equation}Again, we will focus on proving (\ref{eq: reducedf}). By a similar reduction as for Case (c), it suffices to show
\begin{equation}
\begin{aligned}
& N_{1}^{s_{0}}\left|\sum_{
\begin{subarray}
\quad |n_{i}|\sim N_{i}, \,n_{1}-n_{2}+n_{3}-n=0\\
n_{1}^{2}-n_{2}^{2}+n_{3}^{2}-n^{2}=O(M)
\end{subarray} }\int d_{1}(n_{1})\overline{r_{2}(n_{2})}r_{3}(n_{3})\overline{H(n)}\,dt\right|\lesssim N_{3}^{-\epsilon_{1}}\|P_{N_{1}}w_{1}\|_{X^{s_{0},b_{0}}}\|P_{N_{1}}h\|_{X^{0,1-b_{0}}},\\
\end{aligned}
\end{equation}which, by the same argument as in Case (e), will follow from showing for each $J$ (of size $\sim N_3$) that, up to some exceptional set of small probability $e^{-N_{3}^{c}}$, 

\begin{equation}\label{eq: f1}
\sum_{n\in J}\Big\|
\sum_{
\begin{subarray}
\quad n_{1}-n_{2}+n_{3}=n, \,n_{2}\neq n_{1},n_{3}\\
n_{1}^{2}-n_{2}^{2}+n_{3}^{2}-n_{4}^{2}=O(M)
\end{subarray}
}
d_{1}(n_{1})\overline{r_{2}(n_{2})}r_{3}(n_{3})
\Big\|_{L_{t}^{2}}^{2}
\lesssim N_{3}^{-s}\|P_Jw_1\|_{X^{0,b_0}}^2
\end{equation}for some constance $s\gg s_1$. We will derive two different bounds, each of which works better in different regimes of $N_1,N_2, N_3$. 

First, by applying H\"older's inequality to the inner sum, one obtains that the left hand side of (\ref{eq: f1}) is bounded by
\begin{equation}
\begin{split}
\lesssim & \sum_{n\in J}\sup_{n}\#\{n_2,n_3:\, n=n_1-n_2+n_3,\,  \langle n_3-n_2,n-n_3\rangle=O(M)\} \sum_{n_1,n_2,n_3} \|d_1(n_1)\overline{r_2(n_2)}r_3(n_3)\|_{L_t^2}^2\\
\lesssim & N_3^{Cs_1}N_2^2\max\left(\frac{N_3}{N_1^{1/3}},1\right) \sum_{n,n_1,n_2,n_3} N_2^{-2}N_3^{-2}\|d_1(n_1)\|_{L_t^2}^2\\
\lesssim & N_3^{Cs_1}\max\left(\frac{N_3}{N_1^{1/3}},1\right)N_3^{-2} \|P_Jw_1\|_{X^{0,b_0}}^2\sup_{n_1}\#\{n_2,n_3:\, n_2\neq n_1,n_3,\, \langle n_2-n_1,n_2-n_3\rangle=O(M)\}.
\end{split}
\end{equation}In the second line above, we applied (\ref{eq: f0}) and Lemma \ref{lem: circle} (the estimate holds true in both cases $N_1\gg N_3$ and $N_1\sim N_3$). One can then trivially count $n_2$ and apply Lemma \ref{lem: linecounting} to count $n_3$ to further bound the above by
\begin{equation}\label{eq: f2}
\lesssim N_3^{Cs_1}\max\left(\frac{N_3}{N_1^{1/3}},1\right)N_3^{-2} \|P_Jw_1\|_{X^{0,b_0}}^2 N_2^2N_3\leq N_3^{Cs_1}N_2^2\max(N_1^{-1/3}, N_3^{-1}) \|P_Jw_1\|_{X^{0,b_0}}^2.
\end{equation}

We now turn to a different estimate of the left hand side of (\ref{eq: f1}). In fact, we claim that by the same argument as in Case (e), one also obtains in our current case that
\begin{equation}\label{eq: f3}
\sum_{n\in J}\Big\|
\sum_{
\begin{subarray}
\quad n_{1}-n_{2}+n_{3}=n, \,n_{2}\neq n_{1},n_{3}\\
n_{1}^{2}-n_{2}^{2}+n_{3}^{2}-n_{4}^{2}=O(M)
\end{subarray}
}
d_{1}(n_{1})\overline{r_{2}(n_{2})}r_{3}(n_{3})
\Big\|_{L_{t}^{2}}^{2}
\lesssim N_3^{Cs_1} N_2^{-1/6} \|P_Jw_1\|_{X^{0,b_0}}^2.
\end{equation}

We first explain how to complete the proof of (\ref{eq: f1}) using the two estimates (\ref{eq: f2}) and (\ref{eq: f3}) above. In the case $N_2>N_1^{\frac{1}{6}-\frac{1}{100}}$, the desired estimate follows from (\ref{eq: f3}). Now suppose $N_2\leq N_1^{\frac{1}{6}-\frac{1}{100}}$. If $N_3\leq N_1^{1/3}$, the bound in (\ref{eq: f2}) becomes $N_3^{Cs_1}N_2^{2}N_3^{-1}$. Hence, in the subcase when $N_2\leq N_3^{\frac{1}{2}-\frac{1}{100}}$, this is good enough. If $N_2> N_3^{\frac{1}{2}-\frac{1}{100}}$, one can again apply (\ref{eq: f3}) to obtain the desired bound. It is left to check the case $N_3>N_1^{1/3}$, where (\ref{eq: f2}) becomes $N_3^{Cs_1}N_2^2N_1^{-1/3}$. Observe that this does imply the desired result, since we are already in the regime $N_2\leq N_1^{\frac{1}{6}-\frac{1}{100}}$.

It thus suffices to verify (\ref{eq: f3}). Note that most of the estimates in Case (e) still hold, as they do not depend on  the relative sizes of $N_2$, $N_3$. More precisely, the estimate (\ref{eq: e diag}) for the diagonal term, and the final bounds obtained in Case VI, V, VI for the non-diagonal term still hold true (with an extra factor $N_3^{Cs_1}$). We are now left to examine Case I, II, III.

\subsubsection{Case I: $n_2,n_3,n'_2,n'_3$ are distinct}

We start with estimate (\ref{eq: ef1}), which still holds true in Case (f). Note that
\[
\#\{n_1,n_2,n_3,n_2',n_3':\, (\ast\ast)\}=\#\{n,n_2,n_3, n_2',n_3':\,(\ast\ast\ast)\},
\]where $(\ast\ast)$ is  in Case (e), given in (\ref{eq: star star}), and $(\ast \ast\ast)$ denotes the conditions
\begin{equation}\label{ast-ast-ast}
\begin{cases}
n\in J,\\
n_3\neq n,n_2,\, n_3'\neq n, n_2',\\
\langle n_3-n_2, n-n_3=O(M),\quad \langle n_3'-n_2', n-n_3'\rangle =O(M).
\end{cases}
\end{equation}We first count $n\in J$ naively (recalling that $|J|\sim N_3$), we then count $n_2$ naively and $n_3$ using Lemma \ref{lem: circle}, and repeat for $(n_2',n_3')$. This leads to
\[
\#\{n,n_2,n_3, n_2',n_3':\,(\ast\ast\ast)\}\lesssim N_3^{Cs_1}N_3^2 \left(N_2^2\max\left(\frac{N_3}{N_1^{1/3}},1\right)\right)^2.
\]Note that this bound holds true in both cases $N_1\gg N_3$ and $N_1\sim N_3$.
Hence, one obtains
\begin{equation}
\begin{split}
(\ref{eq: e2})_1\lesssim &(N_2N_3)^{-2}N_3^{Cs_1}\left(\#\{n_1,n_2,n_3,n'_2,n'_3:\,(\ast\ast)\} \right)^{1/2}\\
\lesssim &(N_2N_3)^{-2}N_3^{Cs_1} N_2^2 N_3\max\left(\frac{N_3}{N_1^{1/3}},1\right)\\
\lesssim & N_3^{Cs_1}\max(N_1^{-1/3}, N_3^{-1})\leq N_3^{-1/3+Cs_1}\lesssim N_3^{Cs_1}N_2^{-1/3}.
\end{split}
\end{equation}

\subsubsection{Case II: $n_2=n'_2$ ($n_3\neq n'_3$)}

In our current case, after dropping an exceptional set, one still has estimate (\ref{eq: ef1}), where
\[
\begin{split}
\#S(n,n',n_3,n_3'):=&\#\{n_1,n_2:\, (n_1,n_2,n_3,n_2,n_3') \text{ satisfies } (\ast)\}\\
\lesssim& \#\{n_2:\, \langle n_3-n_2, n-n_3\rangle=O(M)\}\lesssim N_3^{Cs_1}N_2.
\end{split}
\]Hence, remembering the definition of $(\ast\ast)$ in \eqref{eq: star star}, one has
\begin{equation}
(\ref{eq: e2})_2\lesssim (N_2N_3)^{-2}N_3^{Cs_1}N_2^{1/2}\left(\#\{n_1,n_2,n_3, n_3':\, (n_1,n_2,n_3,n_2,n_3') \text{ satisfies } (\ast\ast)\} \right)^{1/2}.
\end{equation}By counting $n_2$ first naively, then $n_3$ naively, then $n_1$ by Lemma \ref{lem: linecounting}, and lastly $n_3'$ by Lemma \ref{lem: linecounting} as well, one obtains
\[
\#\{n_1,n_2,n_3, n_3':\, (n_1,n_2,n_3,n_2,n_3') \text{ satisfies } (\ast\ast)\}\lesssim N_3^{Cs_1}N_2^2N_3^2N_3N_3,
\]which implies
\begin{equation}
(\ref{eq: e2})_2\lesssim N_3^{Cs_1}N_2^{-1/2}.
\end{equation}

\subsubsection{Case III: $n_3=n'_3$ ($n_2\neq n'_2$)}

Note that after dropping an exceptional set, estimate (\ref{eq: ef3}) still holds true in Case (f). But this time, we count $\#S(n,n',n_2,n_2')$ more carefully. It suffices to count $n_3$ satisfying $\langle n_3-n_2, n-n_3\rangle =O(M)$. By Lemma \ref{lem: circle}, one has
\[
\#S(n,n',n_2,n_2')\lesssim \begin{cases} N_3^{Cs_1}N_3^{2/3}, & \text{if } N_1\sim N_3,\\ N_3^{Cs_1}\max\left(\frac{N_3}{N_1^{1/3}},1\right), & \text{if }N_1\gg N_3.  \end{cases}
\]In the second estimate above, we used the fact that $|n|\sim N_1$ when $N_1\gg N_3$. Same as before, where remembering the definition of $(\ast\ast)$ in \eqref{eq: star star},
\begin{equation}
(\ref{eq: e2})_3\lesssim (N_2N_3)^{-2}N_3^{Cs_1}\left(\#S(n,n',n_2,n_2') \right)^{1/2}\left(\#\{n_1,n_2,n_2', n_3:\, (n_1,n_2,n_3,n_2',n_3) \text{ satisfies } (\ast\ast)\} \right)^{1/2}.
\end{equation}

By counting $n_1$, $n_2$, $n_2'$ trivially, and then $n_3$ using Lemma \ref{lem: linecounting}, one has
\[
\#\{n_1,n_2,n_2', n_3:\, (n_1,n_2,n_3,n_2',n_3) \text{ satisfies } (\ast\ast)\}\lesssim N_3^{Cs_1}N_3^2N_2^4 N_3.
\]Combining the above bounds together, one obtains
\begin{equation}
\begin{split}
(\ref{eq: e2})_3\lesssim &(N_2N_3)^{-2}N_3^{Cs_1} \left(\#S(n,n',n_2,n_2') \right)^{1/2} N_3^{3/2}N_2^2\\
\lesssim & N_3^{Cs_1} N_3^{-1/2}\cdot  \begin{cases} N_3^{1/3}, & \text{if } N_1\sim N_3,\\ \max\left(\frac{N_3^{1/2}}{N_1^{1/6}},1\right), & \text{if }N_1\gg N_3  \end{cases}\\
\lesssim & N_3^{Cs_1}N_3^{-1/6}\leq N_3^{Cs_1}N_2^{-1/6}.
\end{split}
\end{equation}The proof of Case (f) is thus complete.

\subsection{Case (g): $N_1(I)\geq N_2(I)\geq N_3(II)$}

By deterministic estimates, it suffices to show for all $N_1\geq N_{1,0}$ (where $N_{1,0}=\delta^{-1}$), one has up to an exceptional set of probability $e^{-N_1^c}$ that
\begin{equation}\label{eqn: case g}
N_1^{s_0}\left| \sum_{
\begin{subarray}
\quad n_{i}\sim N_{i}, n_{1}-n_{2}+n_{3}=n,  \, n_{2}\neq n_{1},n_{3}\\
n_{1}^{2}-n_{2}^{2}+n_{3}^{2}-n^{2}=O(M)=O(N_{1}^{100s_{0}})
\end{subarray}
}
\int r_{1}(n_{1})\overline{r_{2}(n_{2})}d_{3}(n_{3})\overline{H(n)}\,dt\right|\lesssim N_{1}^{-\epsilon_{1}}, \text{ for some } \epsilon_{1}\gg \epsilon_{0},
\end{equation}where $M=O(N_1^{100s_0})$.

By dropping an exceptional set of probability $e^{-N_{1,0}^c}$, one may assume that
\begin{equation}
|g_{n(\omega)}|\lesssim N_1^{s_1},\quad \forall |n|\sim N_1\geq N_{1,0}.
\end{equation}Our goal is to show that, up to an exceptional set of probability $e^{-N_1^c}$, there holds
\begin{equation}\label{eq: reducedg}
N_{1}^{s_{0}}\left|\int \phi(t)\bar{h}\NN_{1}(P_{N_{1}}v_{1},P_{N_{2}}{v}_{2}, P_{N_{3}}w_{3})\right|\lesssim N_{1}^{-\epsilon_{1}}.
\end{equation}

%In the following, we only prove (\ref{eq: reducedg}), as we have already seen up to this point that the other case $N_1\sim N_2$ follows very similarly. We will comment along the proof when the counting argument needs to be done differently in the $N_1\sim N_2$ case. 
Apparently, (\ref{eq: reducedg}) can be reduced to the following estimate:
\begin{equation}
\begin{aligned}
& N_{1}^{s_{0}}\left|\sum_{
\begin{subarray}
\quad |n_{i}|\sim N_{i}, \,n_{1}-n_{2}+n_{3}-n=0\\
n_{1}^{2}-n_{2}^{2}+n_{3}^{2}-n^{2}=O(M)
\end{subarray} }\int r_{1}(n_{1})\overline{r_{2}(n_{2})}d_{3}(n_{3})\overline{H(n)}\,dt\right|\lesssim N_{1}^{-\epsilon_{1}}.\\
\end{aligned}
\end{equation}

We first derive an estimate that will handle the regime $N_3\leq N_1^{1-\frac{1}{100}}$. To see this, applying Cauchy-Schwarz in $n$, one obtains
\begin{equation}\label{eq: g3}
\begin{split}
&N_{1}^{s_{0}}\left|\sum_{
\begin{subarray}
\quad |n_{i}|\sim N_{i}, \,n_{1}-n_{2}+n_{3}-n=0\\
n_{1}^{2}-n_{2}^{2}+n_{3}^{2}-n^{2}=O(M)
\end{subarray} }\int r_{1}(n_{1})\overline{r_{2}(n_{2})}d_{3}(n_{3})\overline{H(n)}\,dt\right|\\
\lesssim &N_1^{s_0}\left(\sum_n \|H(n)\|_{L_t^2}^2 \right)^{1/2} \left( \sum_n\Big\| \sum_{\begin{subarray}
\quad n_{1}-n_{2}+n_{3}=n, \,n_{2}\neq n_{1},n_{3}\\
n_{1}^{2}-n_{2}^{2}+n_{3}^{2}-n_{4}^{2}=O(M)
\end{subarray}} r_1(n_1)\overline{r_2(n_2)}d_3(n_3)  \Big\|_{L_t^2}^2 \right)^{1/2}\\
\lesssim &N_1^{s_0}\left( \sum_n\Big\| \sum_{\begin{subarray}
\quad n_{1}-n_{2}+n_{3}=n, \,n_{2}\neq n_{1},n_{3}\\
n_{1}^{2}-n_{2}^{2}+n_{3}^{2}-n_{4}^{2}=O(M)
\end{subarray}} r_1(n_1)\overline{r_2(n_2)}d_3(n_3)  \Big\|_{L_t^2}^2 \right)^{1/2}.
\end{split}
\end{equation}

Then, applying Cauchy-Schwarz again to the inner sum above, one has, after dropping an exceptional set of probability $e^{-N_1^c}$,
\begin{equation}\label{eq: g4}
\begin{split}
\lesssim & N_1^{s_0}N_1^{-1}N_2^{-1}\left( \sum_n \Big|\sum_{\begin{subarray}
\quad n_{1}-n_{2}+n_{3}=n, \,n_{2}\neq n_{1},n_{3}\\
n_{1}^{2}-n_{2}^{2}+n_{3}^{2}-n_{4}^{2}=O(M)
\end{subarray}} \|d_3(n_3)\|_{L_t^2}   \Big|^2 \right)^{1/2}\\
\lesssim & N_1^{s_0}N_1^{-1}N_2^{-1}\left(\sum_n\sum_{n_1,n_2,n_3} \|d_3(n_3)\|_{L_t^2}^2  \right)^{1/2}\left(\sup_n\#\{n_2,n_3:\, \langle n_3-n_2, n-n_3\rangle =O(M)\}\right)^{1/2}\\
\lesssim & N_1^{Cs_0}N_1^{-1}N_2^{-1} N_2^{1/2}N_3^{1/2} \left(\sup_{n_3}\#\{n_1,n_2:\, n_2\neq n_1,n_3,\, \langle n_3-n_1, n_3-n_2\rangle=O(M)\}\right)^{1/2},
\end{split}
\end{equation}where in the last step above, we have applied Lemma \ref{lem: count strong}. Another application of Lemma \ref{lem: count strong} implies that
\begin{equation}
\lesssim N_1^{Cs_0}N_1^{-1}N_2^{-1} N_2^{1/2}N_3^{1/2}N_1^{1/2}N_2^{1/2}\lesssim N_1^{Cs_0}N_1^{-1/2}N_3^{1/2}.
\end{equation}Hence, the desired estimate follows if $N_3\leq N_1^{1-\frac{1}{100}}$.

The other case $N_3>N_1^{1-\frac{1}{100}}$ in fact follows directly from the estimates in Case (e). Note that the relative sizes of $N_1,N_2,N_3$ in these two cases are the same, so all the counting in Case (e) remain true here. We briefly sketch the argument here. It suffices to show, up to an exceptional set of small probability $e^{-N_1^c}$, that
\begin{equation}
\sum_{n}\Big\|
\sum_{
\begin{subarray}
\quad n_{1}-n_{2}+n_{3}=n, \,n_{2}\neq n_{1},n_{3}\\
n_{1}^{2}-n_{2}^{2}+n_{3}^{2}-n_{4}^{2}=O(M)
\end{subarray}
}
r_{1}(n_{1})\overline{r_{2}(n_{2})}d_{3}(n_{3})
\Big\|_{L_{t}^{2}}^{2}
\lesssim N_{1}^{-\epsilon_1}.
\end{equation}

We would like to apply again version of the Cauchy-Schwarz inequality in \eqref{eq: m2}, but this time with
\[
\sigma(n,n_3)= 
\sum_{\begin{subarray}
\quad n_{1}-n_{2}+n_{3}=n, \,n_{2}\neq n_{1},n_{3}\\
n_{1}^{2}-n_{2}^{2}+n_{3}^{2}-n_{4}^{2}=O(M)
\end{subarray}} \frac{g_{n_1}(\omega)g_{n_2}(\omega)}{N_1N_2}. \]For the diagonal term, the exact same counting as in (\ref{eq: e diag}) gives, for any fixed $n$,
\begin{equation}
\sum_{n_3} |\sigma(n,n_3)|^2\lesssim (N_1N_2)^{-2}N_1^{Cs_0}N_3^2N_2\lesssim N_1^{Cs_0}N_2^{-1},
\end{equation}which is good enough since $N_2\geq N_3> N_1^{1-\frac{1}{100}}$.

The six cases for the non-diagonal term work similarly. In particular, Case I (all $n_1, n_1', n_2, n_2'$ distinct), Case III ($n_1=n_1'$, $n_2\neq n_2'$), and Case VI ($n_1=n_2'$, $n_2=n_1'$) can be carried out in the exact same way. 

In Case II ($n_2=n_2'$, $n_1\neq n_1'$), one has up to an exceptional set that
\begin{equation}
\Big| \sum_{(\ast),2 }\overline{g_{n_1}}g_{n_2}g_{n'_1}\overline{g_{n'_2}}  \Big|^2\lesssim N_1^{Cs_0}\sum_{n_1,n_1'} (\#S(n,n',n_1,n_1'))^2
\end{equation}Here, by Lemma \ref{lem: linecounting},
\[
\#S(n,n',n_1,n_1'):=\#\{n_2,n_3:\, n_1,n_2,n_3, n_1', n_2 \text{ satisfies } (\ast')\}\lesssim N_1^{Cs_0}N_3,
\]with $(\ast')$ denotes conditions
\begin{equation}
\begin{cases}
n_1\in J,\\
n=n_1-n_2+n_3,\,n_2\neq n_1,n_3,\, \langle n-n_1,n-n_3\rangle=O(M),\\
n'=n_1'-n'_2+n_3,\,n'_2\neq n'_1,n_3,\, \langle n'-n'_1,n'-n_3\rangle=O(M).
\end{cases}
\end{equation}The rest of the argument proceeds in the exact same way as Case II in Case (e).

Since Case IV and V can be dealt with in the same way, we only briefly discuss Case IV ($n_2=n_1'$, $n_1\neq n_2'$) here. By dropping an exceptional set, one has the following analogue of (\ref{eq: e 8}):
\begin{equation}
\Big| \sum_{(\ast),4 }\overline{g_{n_1}}g_{n_2}g_{n'_1}\overline{g_{n'_2}}  \Big|^2\lesssim N_1^{Cs_0}\sum_{n_1,n_2'} (\#S(n,n',n_1,n_2'))^2,
\end{equation}where by Lemma \ref{lem: linecounting},
\[
\#S(n,n',n_1,n_2'):=\#\{n_2, n_3:\, (n_1, n_2, n_3, n_2, n_2') \text{ satisfies } (\ast')\}\lesssim N_1^{Cs_0}N_3.
\]The rest of the argument again proceeds in the same way as in Case IV of Case (e), which is left to the reader. The proof of Case (g) is complete.

\subsection{Case (h): $N_1(I)\geq N_3(II)\geq N_2(I)$}

This case can be estimated in the same way as Case (g), where again if $N_3>N_1^{1-\frac{1}{100}}$, the bounds in Case (f) apply.  Note that, compared to Case (f),  one can think about $N_{1}(I)$,  in Case (h),  as $N_{1}(II)$, except one suffers a loss $N_{1}^{2s_{0}}$. The computation in Case (f) gives a gain of $N_{3}^{-s^{*}}$, where $s^{*}$ can be computed explicitly. Thus, when $s_{0}$ is small enough and $N_{3}\geq N_{1}^{\frac{1}{100}}$, the extra loss of  $N_{1}^{2s_{0}}$ can be neglected. We omit the details.

\subsection{Case (i), (j): $N_1(I)\geq N_2(II)\geq N_3(I)$ or $N_1(I)\geq N_3(I)\geq N_2(II)$}

As before, we focus on showing for all $N_1\geq N_{1,0}$ (where $N_{1,0}=\delta^{-1}$), up to an exceptional set of probability $e^{-N_1^c}$, there holds
\begin{equation}\label{eqn: case ij}
N_1^{s_0}\left| \sum_{
\begin{subarray}
\quad n_{i}\sim N_{i}, n_{1}-n_{2}+n_{3}=n,  \, n_{2}\neq n_{1},n_{3}\\
n_{1}^{2}-n_{2}^{2}+n_{3}^{2}-n^{2}=O(M)=O(N_{1}^{100s_{0}})
\end{subarray}
}
\int r_{1}(n_{1})\overline{d_{2}(n_{2})}r_{3}(n_{3})\overline{H(n)}\,dt\right|\lesssim N_{1}^{-\epsilon_{1}}, \text{ for some } \epsilon_{1}\gg \epsilon_{0},
\end{equation}where $M=O(N_1^{100s_0})$.

We may assume  dropping an exceptional set of probability $e^{-N_{1,0}^c}$ that
\begin{equation}
|g_{n(\omega)}|\lesssim N_1^{s_1},\quad \forall |n|\sim N_1\geq N_{1,0}.
\end{equation}We aim to show that, up to an exceptional set of probability $e^{-N_1^c}$, there holds
\begin{equation}\label{eq: reducedij}
N_{1}^{s_{0}}\left|\int \phi(t)\bar{h}\NN_{1}(P_{N_{1}}v_{1},P_{N_{2}}{w}_{2}, P_{N_{3}}v_{3})\right|\lesssim N_{1}^{-\epsilon_{1}}.
\end{equation}

%In the following, we only prove (\ref{eq: reducedg}), as we have already seen up to this point that the other case $N_1\sim N_2$ follows very similarly. We will comment along the proof when the counting argument needs to be done differently in the $N_1\sim N_2$ case. 
Similarly as before, (\ref{eq: reducedij}) will follow from the following estimate:
\begin{equation}
\begin{aligned}
& N_{1}^{s_{0}}\left|\sum_{
\begin{subarray}
\quad |n_{i}|\sim N_{i}, \,n_{1}-n_{2}+n_{3}-n=0\\
n_{1}^{2}-n_{2}^{2}+n_{3}^{2}-n^{2}=O(M)
\end{subarray} }\int r_{1}(n_{1})\overline{d_{2}(n_{2})}r_{3}(n_{3})\overline{H(n)}\,dt\right|\lesssim N_{1}^{-\epsilon_{1}}.\\
\end{aligned}
\end{equation}

We will first introduce an estimate that allows one to reduce to the regime $N_2>N_1^{1-\frac{1}{100}}$. To see this, following the same Cauchy-Schwarz argument as in (\ref{eq: g3}), (\ref{eq: g4}), one obtains, after dropping a set of probability $e^{-N_1^c}$,
\begin{equation}
\begin{split}
& N_{1}^{s_{0}}\left|\sum_{
\begin{subarray}
\quad |n_{i}|\sim N_{i}, \,n_{1}-n_{2}+n_{3}-n=0\\
n_{1}^{2}-n_{2}^{2}+n_{3}^{2}-n^{2}=O(M)
\end{subarray} }\int r_{1}(n_{1})\overline{d_{2}(n_{2})}r_{3}(n_{3})\overline{H(n)}\,dt\right|\\
\lesssim &N_1^{Cs_0}N_1^{-1}N_3^{-1}N_2^{1/2}N_3^{1/2}\left(\sup_{n_2} \#\{n_1,n_3:\, n_2\neq n_1,n_3,\, \langle n_3-n_1, n_3-n_2\rangle=O(M)\}\right)^{1/2}\\
\lesssim & N_1^{Cs_0}N_1^{-1}N_3^{-1}N_2^{1/2}N_3^{1/2}N_1^{1/2}N_3^{1/2}\lesssim N_1^{Cs_0}N_1^{-1/2}N_2^{1/2},
\end{split}
\end{equation}where in the last two steps we have applied Lemma \ref{lem: count strong}. Hence, if $N_2\leq N_1^{1-\frac{1}{100}}$, the desired estimate follows.

Next, we may apply the same estimate as in Case (c) to deal with the case $N_3\leq N_1^{1/5}$. Indeed, introducing an extra factor of $N_3^{s_0}$ (since the third input function in the current case is random) to (\ref{eq: c9}), one has
\begin{equation}
N_{1}^{s_{0}}\left|\sum_{
\begin{subarray}
\quad |n_{i}|\sim N_{i}, \,n_{1}-n_{2}+n_{3}-n=0\\
n_{1}^{2}-n_{2}^{2}+n_{3}^{2}-n^{2}=O(M)
\end{subarray} }\int r_{1}(n_{1})\overline{d_{2}(n_{2})}r_{3}(n_{3})\overline{H(n)}\,dt\right|\lesssim N_1^{Cs_0}N_1^{-1/2}N_2^{1/4}N_3,
\end{equation}which is enough to handle the case $N_3\leq N_1^{1/5}$.

To summarize, we have reduced the desired estimate to the regime $N_2>N_1^{1-\frac{1}{100}}$ and $N_3>N_1^{1/5}$. The rest of the argument is essentially repeating that of Case (e). Define
\[
\sigma(n,n_2)= 
\sum_{\begin{subarray}
\quad n_{1}-n_{2}+n_{3}=n, \,n_{2}\neq n_{1},n_{3}\\
n_{1}^{2}-n_{2}^{2}+n_{3}^{2}-n_{4}^{2}=O(M)
\end{subarray}} \frac{g_{n_1}(\omega)g_{n_3}(\omega)}{N_1N_3}.
\]Then, one has
\begin{equation}
\begin{split}
&N_{1}^{s_{0}}\left|\sum_{
\begin{subarray}
\quad |n_{i}|\sim N_{i}, \,n_{1}-n_{2}+n_{3}-n=0\\
n_{1}^{2}-n_{2}^{2}+n_{3}^{2}-n^{2}=O(M)
\end{subarray} }\int r_{1}(n_{1})\overline{d_{2}(n_{2})}r_{3}(n_{3})\overline{H(n)}\,dt\right|\\
\lesssim &N_1^{Cs_0}\left[\max_{n}\sum_{n_2} |\sigma(n,n_2)|^2 + \left(\sum_{n\neq n'}\Big| \sum_{n_2} \sigma(n,n_2)\overline{\sigma(n',n_2)}\Big|^2\right)^{1/2}\right]^{1/2}.
\end{split}
\end{equation}

Again, the diagonal term is easier to deal with. After dropping a set of small probability and applying Lemma \ref{lem: count strong}, one has
\begin{equation}
\max_{n}\sum_{n_2} |\sigma(n,n_2)|^2\lesssim N_1^{Cs_0}N_1^{-2}N_3^{-2}N_2N_3\lesssim N_1^{-1+Cs_0}.
\end{equation}

For the non-diagonal term, rewrite
\begin{equation}\label{eq: ij3}
\left(\sum_{n\neq n'}\Big| \sum_{n_2} \sigma(n,n_2)\overline{\sigma(n',n_2)}\Big|^2\right)^{1/2}= (N_1N_3)^{-2}\left(\sum_{n\neq n'}\Big| \sum_{(\ast)''} g_{n_1}g_{n_3}\overline{g_{n_1'}}\overline{g_{n_3'}}   \Big|^2\right)^{1/2},
\end{equation}where $(\ast'')$ denotes the set of $(n_1,n_2,n_3,n_1',n_3')$ satisfying
\begin{equation}\label{ast''}
\begin{cases}
n=n_1-n_2+n_3,\, n_2\neq n_1,n_3,\, \langle n-n_1, n-n_3\rangle=O(M),\\
n'=n'_1-n_2+n'_3,\, n_2\neq n'_1,n'_3,\, \langle n'-n'_1, n'-n'_3\rangle=O(M).
\end{cases}
\end{equation}We discuss three subcases in the following, and omit the symmetric ones. Note that they proceed very similarly as the corresponding cases in Case (e).

\subsubsection{Case I: $n_1,n_3,n_1',n_3'$ are distinct}

By dropping an exceptional set, one obtains 
\begin{equation}
(\ref{eq: ij3})_1\lesssim N_1^{Cs_0}(N_1N_3)^{-2}(\#S)^{1/2},
\end{equation}where $S$ denotes the set $(n_1,n_2,n_3,n_1',n_3')$ satisfying
\begin{equation}\label{eq: ij4}
\begin{cases}
n_2\neq n_1,n_3,n_1',n_3',\\
\langle n_2-n_1, n_2-n_3\rangle=O(M), \quad  \langle n_2-n_1', n_2-n_3'\rangle=O(M).
\end{cases}
\end{equation}Counting $n_2$ trivially first, then applying Lemma \ref{lem: count strong} twice, one has
\[
\#S\lesssim N_1^{Cs_0}N_2^2(N_1N_3)^2,
\]which implies
\begin{equation}
(\ref{eq: ij3})_1\lesssim N_1^{Cs_0}N_1^{-1}N_2N_3^{-1}\lesssim N_1^{-1/5+Cs_0}.
\end{equation}

\subsubsection{Case II: $n_1=n_1'$ ($n_3\neq n_3'$)}

Similarly as in Case II of Case (e), one obtains
\begin{equation}
(\ref{eq: ij3})_2\lesssim N_1^{Cs_0}(N_1N_3)^{-2}\left(\#S(n,n',n_3,n_3')\right)^{1/2}\left(\#S\right)^{1/2},
\end{equation}where remembering the definition of $(\ast'')$ in \eqref{ast''}
\[
\#S(n,n',n_3,n_3')=\#\{n_1,n_2:\, (n_1,n_2,n_3, n_1,n_3') \text{ satisfies } (\ast'')\}\lesssim N_1^{Cs_0}N_1
\]by Lemma \ref{lem: linecounting}, and $S$ denotes the set of $(n_1,n_2, n_3,n_3')$ satisfying (\ref{eq: ij4}). Hence, by counting $n_2$ trivially, then counting $n_1, n_3$ via Lemma \ref{lem: count strong}, and finally counting $n_3'$ according to Lemma \ref{lem: linecounting}, one obtains
\[
\#S\lesssim N_1^{Cs_0}N_2N_1N_3N_3.
\]Combining together, one has
\begin{equation}
(\ref{eq: ij3})_2\lesssim N_1^{Cs_0}(N_1N_3)^{-2}N_1^{1/2}N_1^{1/2}N_2^{1/2}N_3\lesssim N_1^{Cs_0}N_1^{-1}N_2^{1/2}N_3^{-1}\lesssim N_1^{-1/2+Cs_0}.
\end{equation}

\subsubsection{Case III: $n_1=n_3'$ ($n_3\neq n_1'$)}

Following the same calculation as in Case III of Case (e), one has, up to a small exceptional set,
\begin{equation}
(\ref{eq: ij3})_2\lesssim N_1^{Cs_0}(N_1N_3)^{-2}\left(\#S(n,n',n_3,n_1')\right)^{1/2}\left(\#S\right)^{1/2},
\end{equation}where
\[
\#S(n,n',n_3,n_1')=\#\{n_1,n_2:\, (n_1,n_2,n_3,n_1',n_1) \text{ satisfies } (\ast'')\}\lesssim N_1^{Cs_0}N_1
\]according to Lemma \ref{lem: linecounting}. In the above, $S$ consists of $(n_1,n_2,n_3,n_1')$ so that $(n_1,n_2,n_3,n_1',n_1)$ satisfies (\ref{eq: ij4}). One thus has
\[
\#S\lesssim N_3^2 (N_1^{Cs_0+1}N_2)N_1
\]by Lemma \ref{lem: count strong} and \ref{lem: linecounting} similarly as before. Therefore,
\begin{equation}
(\ref{eq: ij3})_2\lesssim N_1^{Cs_0}(N_1N_3)^{-2}N_1^{1/2}N_1N_2^{1/2}N_3\lesssim N_1^{Cs_0}N_1^{-1/2}N_2^{1/2}N_3^{-1}\lesssim N_1^{-1/5+Cs_0}.
\end{equation}This concludes the proof of Case (i) and (j).

\subsection{Case (k), (l): $N_1(I)\geq N_2(I)\geq N_3(I)$ or $N_1(I)\geq N_3(I)\geq N_2(I)$}

Considering only the case $N_1\geq N_{1,0}$, where $N_{1,0}^{100}=\delta^{-1}$. Our goal is to show that 
\begin{equation}\label{eq: reducedkl}
N_1^{s_0}\left|\int \NN_{1}(P_{N_{1}}v_{1}, P_{N_{2}}v_{2},P_{N_{3}}v_{3})\bar{h}\phi(t)\right|\lesssim N_{1}^{-\epsilon_{1}}, \text{ for some } \epsilon_{1}\gg \epsilon_{0}
\end{equation}up to an exceptional set of probability $e^{-N_1^c}$.

Let $M=O(N_1^{100s_0})$, by a similar reduction argument as in Case (c), it suffices to prove that 
\begin{equation}\label{eq: mainkl}
N_1^{s_0}\left| \sum_{
\begin{subarray}
\quad n_{i}\sim N_{i}, n_{1}-n_{2}+n_{3}=n,  \, n_{2}\neq n_{1},n_{3}\\
n_{1}^{2}-n_{2}^{2}+n_{3}^{2}-n^{2}=O(M)
\end{subarray}
}
\int r_{1}(n_{1})\overline{r_{2}(n_{2})}r_{3}(n_{3})\overline{H(n)}\,dt\right|\lesssim N_{1}^{-\epsilon_{1}}, \text{ for some } \epsilon_{1}\gg \epsilon_{0}.
\end{equation}

Following the same argument in (\ref{eq: g3}), one has
\begin{equation}\label{eq: kl1}
\begin{split}
& N_1^{s_0}\left| \sum_{
\begin{subarray}
\quad n_{i}\sim N_{i}, n_{1}-n_{2}+n_{3}=n,  \, n_{2}\neq n_{1},n_{3}\\
n_{1}^{2}-n_{2}^{2}+n_{3}^{2}-n^{2}=O(M)
\end{subarray}
}
\int r_{1}(n_{1})\overline{r_{2}(n_{2})}r_{3}(n_{3})\overline{H(n)}\,dt\right|\\
\lesssim &N_1^{s_0}\left( \sum_n\Big\| \sum_{\begin{subarray}
\quad n_{1}-n_{2}+n_{3}=n, \,n_{2}\neq n_{1},n_{3}\\
n_{1}^{2}-n_{2}^{2}+n_{3}^{2}-n_{4}^{2}=O(M)
\end{subarray}} r_1(n_1)\overline{r_2(n_2)}r_3(n_3)  \Big\|_{L_t^2}^2 \right)^{1/2}.
\end{split}
\end{equation}

Suppose $n_1,n_2,n_3$ are all distinct, then by dropping a set of probability $e^{-N_1^c}$ and recalling the presence of Schwartz function $\psi(t)$ in each $r_i(n_i, t)$, one can bound the above by
\begin{equation}
\begin{split}
\lesssim &N_1^{Cs_0} (N_1N_2N_3)^{-1} \left(\#\{n_1\neq n_2\neq n_3:\, \langle n_2-n_1, n_2-n_3\rangle=O(M)\} \right)^{1/2}\\
\lesssim & N_1^{Cs_0} (N_1N_2N_3)^{-1} (N_2^2N_3^2N_1)^{1/2}\lesssim N_1^{-1/2+Cs_0},
\end{split}
\end{equation}where in the second line above we trivially counted $n_2, n_3$ and applied Lemma \ref{lem: linecounting} to count $n_1$.

Now, suppose $n_1=n_3$. For fixed $n$, one thus has $n=2n_3-n_2$ and $|n_3-n_2|^2=O(M)$, hence $|n_3-n|^2=O(M)$. By losing a factor of $N_1^{Cs_0}$, we can assume that $|n_3-n|^2=\mu+O(1)$, where $\mu\lesssim N_1^{Cs_0}$, in other words, $n_3$ lies in an annulus of radius $\sim R$ and thickness $\sim O(\frac{1}{R})$, with $R\lesssim N_1^{Cs_0}$. Applying Lemma \ref{lem: circlepart}, one sees that the total number of such $n_3$ is at most $\lesssim R^{2/3}\lesssim N_1^{Cs_0}$. 

Therefore, by Cauchy-Schwarz, outside an exceptional set of probability $e^{-N_1^c}$, one has
\begin{equation}
\begin{split}
(\ref{eq: kl1})\lesssim &N_1^{Cs_0}\left(\#\{n_3:\, \text{fixing }n\} \right)^{1/2}(N_1N_2N_3)^{-1}\left( \#\{n_2,n_3\}\right)^{1/2}\\
\lesssim &N_1^{Cs_0}(N_1N_2N_3)^{-1} N_2N_3\lesssim N_1^{-1+Cs_0}
\end{split}
\end{equation}by trivially counting $n_2$, $n_3$.

The proof of Case (k) and (l) is hence complete, so is the proof of Proposition \ref{prop: onemore}.

\appendix
\section{Time localization of  $X^{s,b}$}\label{asectimeloc}
In this section, we summarize several standard time localization facts for the $X^{s,b}$ space, and also briefly recall the proof of  Lemma \ref{lem: smallpower}. The presentation mainly follows that from \cite{caffarelli1999hyperbolic}.
Here $\phi$ is a fixed time cut off function.
There are several basic facts about   the $X^{s,b}$ space that we can recall below.  We have 
\begin{equation}
\|\phi(t/\delta)u\|_{X^{s,b}}\lesssim_{b} \|u\|_{X^{s,b}}, \quad 0< b<\frac{1}{2}
\end{equation}
\begin{equation}
\|\phi(t/\delta)u\|_{X^{s,b}}\lesssim_{b} \delta^{\frac{1-2b}{2}}\|u\|_{X^{s,b}}  \quad \frac{1}{2}<b<1.
\end{equation}
Also, Hausdorff-Young inequality gives the following estimate which is useful in the interpolation
\begin{equation}\label{eq: hy}
\|\phi(t)u\|_{L_{t,x}^{4}}\lesssim_{\epsilon} \|u\|_{X^{1/2,\frac{1}{4}+\epsilon}},
\end{equation}
which can be compared to  estimates (95), (96) on page 26 of \cite{caffarelli1999hyperbolic}.

In what follows, one should think $1\gg s_{p}\gg \epsilon>0$.  We will only do proof for \eqref{eqq: xsbstrloc} in Lemma \ref{lem: smallpower}.

 Via Strichartz estimate and interpolation of Hausdorff Young inequality, one can obtain

\begin{equation}
\|\phi(t)u\|_{L_{t,x}^{4}}\lesssim_{\epsilon}\|u\|_{X^{3\epsilon,\frac{1}{2}-\epsilon}}
\end{equation}
(One may change the $3$ in the above to any number larger than $2$.)
Similarly, for $p>4$, one can obtain
\begin{equation}
\|\phi(t)u\|_{L_{t,x}^{p}}\lesssim_{\epsilon}\|u\|_{X^{s_{p}+10\epsilon,\epsilon}}
\end{equation}
There are the  following two H\"older inequalities,
\begin{enumerate}
\item 
\begin{equation}
\|\phi(t/\delta)u\|_{L_{t,x}^{4}}\leq \|\phi(t/\delta)u\|^{\theta_{p}}_{L_{t,x}^{2}}\|\phi(t)u\|^{1-\theta_{p}}_{L^{	p}_{t,x}},
\end{equation}
where one has $\frac{1}{4}=\frac{\theta_{p}}{2}+\frac{1-\theta_{p}}{p}$. $\theta_{p}=\frac{s_{p}}{1+s_{p}}\geq\frac{1}{2}s_{p}$
\item 
\begin{equation}
\|\phi(t/\delta)u\|_{L_{t,x}^{2}}\leq \delta^{1/4}\|\phi(t)u\|_{L_{t,x}^{4}}
\end{equation}
\end{enumerate}

One derives

\begin{equation}
\begin{aligned}
&\|\phi(t/\delta) u\|_{L_{t,x}^{4}}\\
\leq &\|\phi(t\delta)u\|_{L_{t,x}^{2}}^{\theta_{p}}\|\phi(t)u\|_{L_{t,x}^{p}}^{1-\theta_{p}}\\
\leq &\|\phi(t\delta)u\|_{L_{t,x}^{4}}^{\theta_{p}}\delta^{\theta_{p}/4}\|\phi(t)u\|_{L_{t,x}^{p}}^{1-\theta_{p}}\\
\lesssim_{\epsilon}& \|u\|_{X^{3\epsilon,\frac{1}{2}-\epsilon}}^{\theta_{p}}\|u\|^{1-\theta_{p}}_{X^{s_{p}+10\epsilon,\frac{1}{2}-\epsilon}}\delta^{\frac{\theta_{p}}{4}}
\end{aligned}
\end{equation}
Note that
\begin{equation}
\theta_{p}=\frac{s_{p}}{1+s_{p}}\geq\frac{1}{2}s_{p}.
\end{equation}

Thus, to summarize,
when $s\ll 1$, and $\epsilon\ll s$, one has
\begin{equation}
\|\phi(t/\delta) u\|_{L_{t,x}^{4}}\lesssim_{\epsilon} \|u\|_{X^{s+10\epsilon,\frac{1}{2}-\epsilon}}\delta^{\frac{s}{4}},
\end{equation}

which, for convenience, can be written as 
\begin{equation}
\|\phi(t/\delta)u\|_{L_{t,x}^{4}}\lesssim_{\epsilon}\|u\|_{X^{s,\frac{1}{2}-\epsilon}}\delta^{s/8}.
\end{equation}

Localizing at frequency $N$, this gives Lemma \ref{lem: smallpower} for balls $B$  of radius $N$, which are centered at origin point. To prove general $B$ centered at $n_{0}$, one simply observes
\begin{equation}
\sum_{n\in B}a_{n}e^{inx}e^{in^{2}t}e^{i\lambda t}=\sum_{|n-n_{0}|\leq N}a_{n}e^{i(n-n_{0})(x+2n_{0})}e^{i(n-n_{0})^{2}t}e^{i\lambda t}e^{in_{0}x}e^{-in_{0}^{2}t}
\end{equation}
and the $L_{t,x}^{4}$ norm of a function is invariant under multiplying $e^{in_{0}x}e^{-in_{0}^{2}t}$ and doing space translation in $x$ variable by $n_{0}$.
This ends the proof.

\section{Proof of Lemma \ref{lem: deter}, \ref{lem: onemoredeter}, \ref{lemma: N2}}\label{asecdeter}
We briefly sketch the proof of those three Lemmata here. 

We start with Lemma \ref{lemma: N2}. Let $h(n,t), f_{i}(n,t)$ be space Fourier transform of $h,f_{i}$, and we will also 
short handedly write  them as $h(n), f_{i}(n)$. We  only prove 
\begin{equation}\label{eq: 1a}
\left|\int \phi(t/\delta)h\NN_{2}(P_{N_{1}}f_{1}, P_{N_{1}}f_{2}, P_{N_{1}}f_{3})\right|\lesssim (\delta^{1/4}\|P_{N_{1}}h\|_{X_{0,1-b_{0}}}\|f_{1}\|_{X^{0,b_{0}}}\sup_{|n|\sim N_{1}}\Pi_{j\neq 1}\|f_{j}(n)e^{inx}\|_{X_{0,b_{0}}}).
\end{equation} 
To see this, observe
\begin{equation}\label{eq: 1b}
\begin{aligned}
&\left|\int \phi(t/\delta)h\NN_{2}(P_{N_{1}}f_{1}, P_{N_{1}}f_{2}, P_{N_{1}}f_{3})\right|=\left|\sum_{|n|\sim N_{1}}\int \phi(t/\delta)\bar{h}(n)f_{1}(n)\bar{f}_{2}(n)f_{3}(n)\right|\\
\lesssim &\|\phi(t/\delta)\|_{L_{t}^{2}}\|\phi(t)h(n)\|_{L_{t}^{2}}\|\phi(t)f_{1}\|_{L_{t}^{\infty}}\sup_{|n|\sim N_{1}}\Pi_{j\neq 1}\|f_{j}\|_{L^{\infty}_{t}}.
\end{aligned}
\end{equation}

Now we have, (note that one only has one mode in all the estimates below)
\begin{equation}
\|\phi(t)h(n)\|_{L_{t}^{2}}\lesssim \|h(n)e^{inx}\|_{X^{0,1-b_{0}}}, \|f_{i}(n)\|_{L_{t}^{\infty}}\lesssim \|f_{i}(n)e^{inx}\|_{X^{0,b_{0}}}
\end{equation}
then, \eqref{eq: 1a} will follow from \eqref{eq: 1b} by Cauchy Schwarz.

We turn to Lemma \ref{lem: deter}.  We start with \eqref{deter1} to \eqref{deter1.2}.
Estimates \eqref{deter1}, \eqref{deter1.1} follows from \eqref{eq: xsbstr}, \eqref{eqq: xsbstrloc} via H\"older inequality. We point out that the  naive loss will be $N_{1}^{C\epsilon}$ rather than $\max(N_{2},N_{3})^{C\epsilon}$, but this can be handled by a standard $L^{2}$ orthogonality argument, See, for example,\cite{caffarelli1999hyperbolic}, \cite{bourgain1993fourier} for more details.
Now we show  how to derive \eqref{deter1.2} from \eqref{deter1.1}. We shall see that \eqref{deter2} can be derived similarly form \eqref{deter1}.

Recall that we used the notation 
\begin{equation}
f_{i}(x,t)=\sum_{n}f_{i}(n,t)e^{inx}, \quad i=1,2.3
\end{equation}
i.e. $f_{i}(n,t)$ is the space Fourier transform. For the sake of convenience, we denote $f_{i}(n,t)$ with  $f_{i}(n)
$.
Similarly, we wrtie $h=\sum h(n,t)e^{inx}$.

Given \eqref{deter1.1}, in order to derive \eqref{deter1.2}, we need to further prove
\begin{itemize}
\item If $N_{1}\sim N_{2}$
\begin{equation}\label{11}
\sum_{n_{1}\sim |N_{1}|, n_{3}\sim N_{3}}\left|\int \phi(t/\delta\bar{h}(n_{1})f_{1}(n_{1})\bar{f}_{2}(n_{3})f_{3}(n_{3})\right|\lesssim  \delta^{1/10}(\max(N_{2},N_{3}))^{C\epsilon_{0}}\|h\|_{X^{0,1-b_{0}}}\prod_{i}\|P_{N_{i}}f_{i}\|_{X^{0,b_{0}}}
\end{equation}
\item If $N_{2}\sim N_{3}$
\begin{equation}\label{22}
\sum_{n_{1}\sim N_{1}, n_{2}\sim N_{2}}\left|\int \phi(t/\delta)\bar{h}(n_{1})f_{1}(n_{1})\bar{f}_{2}(n_{2})f_{3}(n_{2})\right|\lesssim \delta^{1/10}(\max(N_{2},N_{3}))^{C\epsilon_{0}}\|h\|_{X^{0,1-b_{0}}}\prod_{i}\|P_{N_{i}}f_{i}\|_{X^{0,b_{0}}}
\end{equation}
\item If $N_{1}\sim N_{2}\sim N_{3}$,
\begin{equation}\label{33}
\int |\phi(t/\delta)h\NN_{2}(P_{N_{1}}f_{1}, P_{N_{1}}f_{2}, P_{N_{1}}f_{3})|\lesssim \delta^{1/10}(\max(N_{2},N_{3}))^{C\epsilon_{0}}\|h\|_{X_{0,1-b_{0}}}\prod_{i}\|P_{N_{i}}f_{i}\|_{X^{0,b_{0}}}
\end{equation}
\end{itemize}
Estimate \eqref{33} follows from Lemma \ref{lemma: N2}. The proof of estimates \eqref{11} and \eqref{22} are similar, and we only work on \eqref{11}.
Note that the integration on the  left side is only in $t$. One has, (by Sobolev embedding in the $t$ variable if necessary), that
\begin{equation}
\begin{aligned}
&\|h(n_{1})\|_{L_{t}^{2}}\lesssim \|h(n_{1})e^{in_{1}x}\|_{X^{0,1-b_{0}}},\\
&\|f(n_{1})\|_{L_{t}^{\infty}}\lesssim \|f(n_{1})e^{in_{1}x}\|_{X^{0,b_{0}}},\\
&\|f_{i}(n_{3})\|_{L_{t}^{\infty}}\lesssim \|f_{i}(n_{3})e^{in_{3}x}\|_{X^{0,b_{0}}}.
\end{aligned}
\end{equation} 
Then the desired estimates follow from  a H\"older inequality in $t$ and Cauchy Schwarz inequality in $n_{1},n_{3}$.
Estimates \eqref{deter3}, \eqref{deter4}, \eqref{deter4.5} are of similar flavor. We prove \eqref{deter3} and leave the rest to the interested readers.
Estimate \eqref{deter3} follows from the following  four estimates.
\begin{itemize}
\item 
\begin{equation}\label{44}
\left|\int \psi(t)\bar{h} P_{N_{1}}f_{1}P_{N_{2}}\bar{f}_{2}P_{N_{3}}f_{3}\right|\lesssim (\max(N_{2},N_{3}))^{C\epsilon_{0}}|h\|_{X^{0,1/3}}\|\sup_{J}\|P_{J}f_{1}\|_{L_{t,x}^\infty}\|f_{2}\|_{X^{0,1/3}}\|f_{3}\|_{X^{0,1/3}},
\end{equation}
\item If $N_{1}\sim N_{2}$,
\begin{equation}\label{55}
\sum_{n_{1}\sim |N_{1}|, n_{3}\sim N_{3}}\left|\int \psi(t)\overline{h(n_{1})}f_{1}(n_{1})\bar{f}_{2}(n_{3})f_{3}(n_{3})\right|\lesssim 
\|P_{N_{1}}f_{1}\|_{X^{0,b_{0}}}\|P_{N_{1}}f_{2}\|_{X^{0,1/3}}\|P_{N_{3}}f_{3}\|_{X^{0,1/3}}\|P_{N_{3}}h\|_{X^{0,1/3}},
\end{equation}
\item If $N_{2}\sim N_{3}$,
\begin{equation}\label{66}
\sum_{n_{1}\sim N_{1}, n_{2}\sim N_{2}}\left|\int \psi(t)\overline{{h}(n_{1})}f_{1}(n_{1})\bar{f}_{2}(n_{2})f_{3}(n_{3})\right|\lesssim
\|P_{N_{1}}f\|_{X^{0,b_{0}}}\|P_{N_{2}}f_{2}\|_{X^{0,1/3}}\|P_{N_{2}}f_{3}\|_{X^{0,1/3}}\|P_{N_{1}}h\|_{X^{0,1/3}}
\end{equation}
\item If $N_{1}\sim N_{2}\sim N_{3}$
\begin{equation}\label{77}
\left|\int \psi(t)\bar{h}\NN_{2}(P_{N_{1}}f_{1}P_{N_{2}}\bar{f}_{2}P_{N_{3}}f_{3})\right|\lesssim \min(\|P_{N_{1}}h\|_{X^{0,1-b_{0}}}\|f_{i}\|_{X^{0,b_{0}}}\sup_{|n|\sim N_{1}}\prod_{j\neq i}\|f_{j}(n)e^{inx}\|_{X^{0,b_{0}}}
\end{equation}
\end{itemize}
Again estimate \eqref{77} follows from Lemma \ref{lemma: N2}. We will only prove estimate \eqref{44}, \eqref{55}. The proof of \eqref{66} is similar to that for \eqref{55}.

We start with  \eqref{44}. We may only consider the case $N_{2}\geq N_{3}$, as the case $N_{2}\leq N_{3}$ can be proved similarly.

We may further only consider the  case $N_{1}\gg N_{2}$, otherwise one may replace  $P_{J}$ by $P_{<N_{1}}$.
Observe that  (using $L^{2}$ orthogonality),
\begin{equation}
\int \psi(t)\bar{h} P_{N_{1}}f_{1}P_{N_{2}}\bar{f}_{2}P_{N_{3}}f_{3}\\
=
 \sum_{J} \int \psi(t)\bar{h} P_{J}f_{1}P_{N_{2}}\bar{f}_{2}P_{N_{3}}f_{3}\\
=
\sum_{J}\int \psi(t)P_{J}\bar{h}P_{J}f_{1}P_{N_{2}}\bar{f}P_{N_{3}}f_{3}.
\end{equation}

For each $J$, we may estimate as  follows,
\begin{equation}\label{eq: okla}
\left|\int \int \psi(t)P_{J}\bar{h}P_{J}f_{1}P_{N_{2}}\bar{f}P_{N_{3}}f_{3}\right|\lesssim \|\phi(t)P_{J}h\|_{L_{L_{t,x}^{3}}}\|\phi(t)P_{J}f_{1}\|_{L_{t,x}^{\infty}}\|\phi(t)P_{N_{2}}f_{2}\|_{L_{t,x}^{3}}\|\phi(t)P_{N_{3}}f_{3}\|_{L_{t,x}^{3}},
\end{equation}
where without loss of generality assumed $\psi(t)=\phi(t)^{4}$ for some  well localized $\phi(t)$.

Using Estimate \eqref{eq: interpo} to control the $L^{3}$ norm in  \eqref{eq: okla} and applying a Cauchy Schwarz in $J$,  the desired estimate then follows.

Lemma \ref{lem: onemoredeter} can be proved similarly as Lemma \ref{lem: deter}.

\section{A Cauchy-Schwarz type inequality}
We summarize a (deterministic) Cauchy-Schwarz type  inequality,  that is often used  in  random data type problems.
For simplicity, let $a_{ij}, b_{j}$ be real numbers, assume that 
\begin{equation}
\sum_{j}b_{j}^{2}\lesssim 1,
\end{equation}
which of course implies
\begin{equation}
\sum_{j,j'}b_{j}^{2}b_{j'}^{2}\lesssim 1.
\end{equation}
Then, we have
\begin{equation}
\begin{aligned}
&\sum_{i}|\sum_{j}a_{ij}b_{j}|^{2}
=\sum_{i}\sum_{j,j'}a_{ij}a_{ij'}b_{j}b_{j'}
=\sum_{i}\sum_{j}a_{ij}a_{ij}b_{j}^{2}
+\sum_{i}\sum_{j\neq j'}a_{ij}a_{ij'}b_{j}b_{j'}
\end{aligned}
\end{equation}
Note that 
\begin{equation}
\sum_{i}|\sum_{j}a_{ij}a_{ij}b_{j}|^{2}\lesssim \sup_{j}\sum_{i}a^{2}_{ij} 
\end{equation}
and, by Cauchy inequality,
\begin{equation}
\sum_{i}\sum_{j\neq j'}a_{ij}a_{ij'}b_{j}b_{j'}=\sum_{j\neq j'} b_{j}b_{j'}\sum_{i}a_{ij}a_{ij'}\lesssim (\sum_{j,j'}b^{2}_{j}b_{j'}^{2})^{1/2}\{\sum_{j\neq j'}|\sum_{i}a_{ij}a_{ij}|^{2}\}^{1/2}
\end{equation}
To summarize, and by simple generalization to the complex case, one has
\begin{lem}\label{lem: matrixtrick1}
Assume $\sum_{j}|b_{j}|^{2}\lesssim 1$, then
\begin{equation}
\begin{aligned}
\sum_{i}|\sum_{j}a_{ij}b_{j}|^{2}\lesssim \max_{j}\sum_{i} |a_{ij}|^{2}+\left(\sum_{j\neq j'}|\sum_{i}a_{ij}\bar{a}_{ij'}|^{2}\right)^{1/2}
\end{aligned}
\end{equation}
\end{lem}
One can also easily write down,   via the dual estimate, 
\begin{lem}\label{lem: matrixtrick2}
Assume $\sum_{j}|b_{j}|^{2}\lesssim 1$, then 
\begin{equation}\label{eq: m2}
\begin{aligned}
\sum_{i}|\sum_{j}a_{ij}b_{j}|^{2}\lesssim \max_{i}\sum_{j} |a_{ij}|^{2}+\left(\sum_{i\neq i'}|\sum_{j}a_{i'j}\bar{a}_{ij}|^{2}\right)^{1/2}
\end{aligned}
\end{equation}
\end{lem}
\bibliographystyle{plain}
\bibliography{BG}
\end{document}